\documentclass[letterpaper, 11pt]{amsart}
\usepackage{amsmath,amssymb,amsthm,pinlabel,tikz,hyperref,mathrsfs,color,thmtools}
\usepackage{verbatim}
\newcommand{\nc}{\newcommand}
\nc{\dmo}{\DeclareMathOperator}
\dmo{\ra}{\rightarrow}
\dmo{\N}{\mathbb{N}}
\dmo{\Z}{\mathbb{Z}}
\dmo{\Q}{\mathbb{Q}}
\dmo{\R}{\mathbb{R}}
\dmo{\C}{\mathcal{C}}
\dmo{\Ha}{\mathbb{H}}
\dmo{\AC}{\mathcal{AC}}
\dmo{\Mod}{Mod}
\dmo{\Sp}{Sp}
\dmo{\PMod}{PMod}
\dmo{\B}{B}
\dmo{\PB}{PB}
\dmo{\PR}{PSL(2,\mathbb{R})}
\dmo{\I}{\mathcal{I}}
\dmo{\el}{\ell_{\C}}
\dmo{\NN}{\mathcal{N}}
\dmo{\T}{\mathcal{T}}
\dmo{\rk}{rk}
\dmo{\w}{\omega}
\dmo{\F}{\mathcal{F}}

\usepackage{thm-restate}

\usepackage{subcaption}

\usepackage{tikz-cd}

\tikzcdset{row sep/normal=1.5cm}
\tikzcdset{column sep/normal=1.5cm}

\usetikzlibrary{decorations.pathreplacing}
\usetikzlibrary{decorations.markings}
\tikzset{->-/.style={decoration={
  markings,
  mark=at position #1 with {\arrow{>}}},postaction={decorate}}}

\theoremstyle{plain}
\newtheorem{thm}{Theorem}[section]
\newtheorem{lem}[thm]{Lemma}

\newtheorem{proposition}[thm]{Proposition}

\numberwithin{equation}{section}
\newtheorem{property}[thm]{Property}
\newtheorem{fact}[thm]{Fact}

\theoremstyle{definition}
\newtheorem{definition}[thm]{Definition}
\newtheorem{notation}[thm]{Notation}

\theoremstyle{remark}
\newtheorem{remark}[thm]{Remark}

\title[Simple length spectra as moduli]
{Simple length spectra as moduli for hyperbolic surfaces\\and Rigidity of length identities}

\date{\today}
\author{Hyungryul Baik}
\address{%
		Department of Mathematical Sciences, KAIST\\
		291 Daehak-ro Yuseong-gu, Daejeon, 34141, South Korea 
}
\email{%
        hrbaik@kaist.ac.kr
}
\thanks{The first author was partially supported by Samsung Science \& Technology Foundation grant No.~SSTF-BA1702-01. The revision of the manuscript was supported by the National Research Foundation of Korea (NRF) grant funded by the Korean government (MSIT) (No.~RS-2025-00513595).}

\author{Inhyeok Choi}

\address{%
		Department of Mathematics, Cornell University\\
		310 Malott Hall, Ithaca, New York, 14850, USA 
}
\email{%
        ic393@cornell.edu
        }
\thanks{The second author was partially supported by Samsung Science \& Technology Foundation grant No. SSTF-BA1702-01.}

\author{Dongryul M. Kim}

\address{%
		Department of Mathematics, Yale University\\
		219 Prospect Street, New Haven, CT 06511, USA
}
\email{%
        dongryul.kim@yale.edu
}



\begin{document}

\begin{abstract}
In this article, we revisit classical length identities enjoyed by simple closed curves on hyperbolic surfaces. We state and prove the rigidity of such identities over Teichm{\"u}ller spaces. Due to this rigidity, certain collections of simple closed curves which minimally intersect are characterized on generic hyperbolic surfaces by their lengths.

As an application, we construct a meagre set $V$ in the Teichm{\"u}ller space of a topological orientable surface $S$, possibly of infinite type. Then the isometry class of a (Nielsen-convex) hyperbolic structure on $S$ outside $V$ is characterized by its unmarked simple length spectrum. Namely, we show that the simple length spectra can be used as moduli for generic hyperbolic surfaces. In the case of compact surfaces, an analogous result using length spectra was obtained by Wolpert.
\end{abstract}

\maketitle

%
%

\section{Introduction}	\label{sec:introduction}

Given a closed Riemannian manifold $M$, one can define the Laplace-Beltrami operator $\Delta$ acting on $\mathcal{L}^{2}(M)$. The \emph{(eigenvalue) spectrum} of $M$ is the collection of eigenvalues of $\Delta$, counting multiplicities. A closely related notion is the \emph{(unmarked) length spectrum} (\emph{simple length spectrum}, resp.) of $M$, the collection of lengths of closed geodesics (simple closed geodesics, resp.) in $M$ counting multiplicities. A classical result of Huber (\cite{huber1959hyperbol}, \cite{huber1961hyperbol}) asserts that the spectrum determines the length spectrum in general, and vice versa in the case of hyperbolic surfaces of constant curvature.

The $9g-9$ theorem implies that the marked length spectrum of a closed orientable surface determines its isometry class. We note its generalization to negatively curved surfaces in \cite{otal1990negative} and \cite{croke1990negative}. In this spirit, Gel'fand conjectured in~\cite{gelfand1962icm} that the length spectrum of a closed surface determines its isometry class. (See ~\cite{kac1966drum} for an analogous question by Kac for planar domains.) Since then, various attempts have been made to extract the Riemannian structure of manifolds from their spectra. 

In general, a closed surface of genus $g$ is isospectral to at most finitely many other surfaces. McKean provided the first upper bound on this number as a function of $g$ in \cite{mcKean1972compact}. See \cite{buser1992compact} and \cite{parlier2018length} for further development. Moreover, M\"uller also proved in \cite{muller1992spectral} that a (possibly non-compact) hyperbolic surface of finite area is isospectral to at most finitely many surfaces. Meanwhile, Vign{\'e}ras constructed in ~\cite{vigneras1978exemples} the first examples of isospectral, non-isometric closed surfaces. Sunada explained in ~\cite{sunada1985iso} a general recipe for isospectral closed manifolds in general dimensions including 2. These works answer Gelfand's conjecture in the negative.

Meanwhile, the length spectrum indeed determines the isometry class of some surfaces of low complexity. The case of one-holed torus with a fixed boundary length was proved by Haas in \cite{haas1985moduli}, and the restriction on the boundary length was removed by Buser and Semmler in \cite{buser1988torus}. We note that both of their approaches work when the length spectrum is replaced with the simple length spectrum.

In contrast, it is not known whether the simple length spectrum determines the isometry class of a general hyperbolic surface. In~\cite{maungchang2013sunada}, Maungchang investigated examples of isospectral, non-isometric hyperbolic surfaces constructed in \cite{sunada1985iso} and showed that their simple length spectra differ. \cite{aougab2020covers} observed the relationship between this question and characterizing finite covers of surfaces via simple closed curves. See also \cite{mondal2017rigidity} for a variation of this question involving the length-angle spectrum.

One can instead focus on the simple length spectra of \emph{generic} hyperbolic surfaces. In~\cite{mcshane2008simple}, McShane and Parlier asked whether there is a surface for which all the multiplicities are 1. They showed that the set of marked hyperbolic surfaces (of a given finite type) that does not have this property is meagre, providing a strongly affirmative answer. They also related this set to other questions on low-genus surfaces, including the Markoff conjecture. 

The strategy of McShane and Parlier is to investigate the length equality $l_{X}(\alpha) = l_{X}(\beta)$ for simple closed curves $\alpha$, $\beta$ over the Teichm{\"u}ller space. If two curves are same, then the equality clearly holds on the entire space; otherwise, the equality holds only on a submanifold of the Teichm{\"u}ller space. We note that this strategy is not applicable for non-simple closed curves. Indeed, there are arbitrary many distinct curves on a surface that have the same length with respect to any hyperbolic structure~\cite{randol1980length}.

Motivated by McShane and Parlier, we consider other length identities enjoyed by few-intersecting simple closed curves. As in McShane and Parlier's work, we construct a meagre subset $V$ of the Teichm{\"u}ller space which is a union of countably many analytic submanifolds; few-intersecting simple closed curves and their topological configuration are characterized by their lengths on any hyperbolic surface outside $V$. In other words, we prove the following proposition (which is an interpretation of Proposition~\ref{prop:mainProp} in plain words). This argument does not rely on the finiteness of the surface type. Indeed, it equally applies to Nielsen-convex hyperbolic structures on surfaces of infinite type, which means that the hyperbolic structure admits a nice pants decomposition (see Fact \ref{fact.nielsenconvex} for the precise definition). Teichm\"uller space for such hyperbolic structures is defined in a similar way to the usual Teichm\"uller spaces (Definition \ref{def.Teich}).

\begin{proposition}[Interpretation of the main proposition]
 Let $S$ be a topological orientable surface. Then there exists a meagre subset $V \subseteq \T(S)$ such that for $X \in \T(S) \setminus V$ and a hyperbolic surface $X'$ homeomorphic to either one-holed torus or $p$-punctured $b$-holed sphere where $p + b = 4$, the following implication holds: $$ \mathcal{L}(X') \subseteq \mathcal{L}(X) \Rightarrow \begin{matrix} \exists \mbox{ isometric immersion} \\ X' \to X\end{matrix}$$
	
	Here, $\mathcal{L}$ stands for the simple length spectrum.
\end{proposition}

As an application of this result, we prove that the simple length spectra of generic surfaces determine their isometry classes. An analogous result for the length spectra was obtained by Wolpert in~\cite{wolpert1979moduli}. Wolpert considered a subvariety $V_{g}$ of the Teichm{\"u}ller space $\T_{g}$ of genus $g$ and proved the following: if $[f_{1}, X] \in \T_{g} \setminus V_{g}$ and $[f_{2}, X'] \in \T_{g}$ have the same length spectrum, then $[f_{1}, X]$ and $[f_{2}, X']$ belong to the same orbit of the extended mapping class group $\Mod^{\pm}_{g}$. 

We note that Wolpert's argument requires length information of some non-simple closed curves, which are not available from the simple length spectrum. In addition, Wolpert's argument heavily relies on Mumford's compactness theorem, which is hard to be generalized to infinite-type surfaces. Our main result replaces the length spectrum in Wolpert's theorem with the simple length spectrum, using techniques that apply to both finite-type and infinite-type surfaces. 

The following is the main theorem of this paper. By a meagre subset, we mean a union of countably many analytic submanifolds of positive codimension.

\begin{restatable}[Simple length spectra as moduli]{thm}{aeLengthSpectra} \label{thm:aeLengthSpectra}
Let $S$ be a topological orientable surface with compact boundaries and with non-abelian fundamental group and let $\T(S)$ be the Teichm{\"u}ller space of $S$. Then there exists a meagre subset $V$ of $\T(S)$ satisfying the following: if $[f_{1}, X] \in \mathcal{T}(S) \setminus V$ and $[f_{2}, X'] \in \mathcal{T}(S)$ have the same simple length spectra, then $[f_{1}, X]$ and $[f_{2}, X']$ belong to the same orbit of $\Mod^{\pm}(S)$.
\end{restatable}

We emphasize again that in the above theorem $S$ does not have to be of finite type. 

\subsection*{Organization of the paper}
In Section~\ref{sec:backgrounds}, we cover some background for the paper. It especially includes Teichm\"uller spaces and pants decompositions of infinite-type surfaces, and the fractional Dehn twists. The relation between topological configurations of curves and identities among their lengths is dealt with in Section~\ref{sec:lengthidentities}. In Section~\ref{sec:lowcomplexity}, the main theorems for surfaces of low complexity are proved. They serve as base cases for the induction argument in the proof of the main theorem provided in Section~\ref{sec:proofofmainthm}. Further questions are asked in Section~\ref{sec:furtherquestions}. For some lemmas which seem to be well-known to the experts while the authors could not find explicit references, we provide their proofs in Appendix~\ref{appendix:A1},~\ref{appendix:A}, and~\ref{appendix:B} for the sake of completeness. 

\subsection*{Acknowledgments}
We appreciate Changsub Kim, KyeongRo Kim, Bram Petri, Philippe A. Tranchida, Scott Wolpert for helpful conversations. We would like to thank the anonymous referee for valuable comments.

%
%

\section{Backgrounds} \label{sec:backgrounds}

\subsection{Surfaces, curves and hyperbolic geometry}

In this section, we introduce basic notions. For details, we refer the readers to~\cite{FarbMargalit12} and~\cite{alessandrini2011fenchel}.

In this article, (topological) surfaces are second-countable, connected, oriented 2-dimensional manifolds with compact boundaries. Those with finitely generated fundamental groups are said to be of finite type; others are said to be of infinite type. A finite-type surface is characterized by the genus, the number of boundary components, and the number of ends. Topologically, a finite-type surface is homeomorphic to the connected sum of a sphere and finitely many tori, with finitely many open discs and points removed. We denote by $S_{g, p, b}$ the genus $g$ surface with $p$ punctures and $b$ boundaries.  Throughout, we only consider surfaces which are not sphere, disc, and punctured disc.

A homotopy on a surface is required to preserve each boundary component of the surface setwise, but not necessarily pointwise. A \emph{loop} on a surface is a continuous map from $S^{1}$ to the surface. A loop is said to be \emph{simple} if it is injective. A \emph{curve} on a surface is a nontrivial free homotopy class of simple loops. A curve bounding an annulus is said to be \emph{peripheral}; otherwise it is said to be \emph{essential}. Each peripheral curve either bounds a puncture or a boundary component.

An \emph{arc} on a surface is either an essential curve or the homotopy class of an essential simple arc connecting ends or boundary components. A \emph{multicurve} (\emph{multi-arc}, respectively) is a finite union of disjoint essential curves (arcs, respectively). All curves, arcs, multicurves, and multi-arcs are unoriented in this article, unless stated otherwise.

A \emph{(properly embedded) subsurface} of a surface $S$ is the image of a proper embedding $\psi$ of a surface $S'$ into $S$. The properness forbids an open end of the subsurface from accumulating on the boundary of the ambient surface. Abusing the notation, we sometimes refer to the image $\psi(S')$ in $S$ as a subsurface. An \emph{immersed subsurface} of a surface $S$ is the image of a proper immersion $\psi$ of a surface $S'$ into $S$ whose restriction on $\textrm{int}(S')$ is an embedding. For example, a subsurface of type $S_{1, 0, 2}$ can be viewed as an immersed subsurface of type $S_{0, 0, 4}$.

The following are terminologies for surfaces with small complexity. A \emph{generalized pair of pants} is a surface whose interior is homeomorphic to a 3-punctured sphere. These include $S_{0, p, b}$ with $b+p = 3$. A \emph{generalized shirt} is a surface whose interior is homeomorphic to a 4-punctured sphere. These include $S_{0, p, b}$ with $b+p=4$. Note that some literature use \emph{$Y$-piece} and \emph{$X$-piece} to denote the generalized pair of pants and the generalized shirt, respectively.

\begin{definition}
A \emph{pants decomposition} $\mathcal{P}$ of a surface $S$ is a collection of disjoint, distinct curves $\{C_{i}\}_{i \in I}$ satisfying the following: \begin{enumerate}
\item each component of $S \setminus \bigcup_{i} C_{i}$ is a generalized  pair of pants without boundary, and
\item there exist disjoint tubular neighborhoods $N_{i}$ of $C_{i}$ in $S$.
\end{enumerate}
\end{definition}

Note that a pants decomposition of a surface should include all of its boundary components. Condition (2) is to prevent the case when some of the $C_i$'s accumulate to one of the $C_j$. Such a collection can be made for instance when $S$ is $\mathbb{S}^2$ minus a Cantor set.
The following notion will be useful when we discuss Fenchel-Nielsen coordinates in the later section: recall that an arc on a surface is either an essential curve or the homotopy class of an essential simple arc connecting ends or boundary components.

\begin{definition}
A \emph{seam} for a pants decomposition $\mathcal{P} = \{C_{i}\}_{i\in I}$ is a collection of mutually disjoint arcs $\{A_{j}\}_{j \in J}$ that satisfies the following:\begin{enumerate}
\item $\{A_{j}\}_{j}$ and $\{C_{i}\}_{i}$ are in a general position, i.e., $( \bigcup_{i} C_{i}) \cap (\bigcup_{j} A_{j})$ is a discrete subset of $S$;
\item On each generalized pair of pants of $S \setminus \bigcup_{i} C_{i}$, $\bigcup_{j} A_{j}$ connects each pair of ends and decomposes the pair of pants into two generalized hexagons.
\end{enumerate}
When a seam $\{A_{j}\}_{j \in J}$ is given for a pants decomposition $\mathcal{P}= \{C_{i}\}_{i \in I}$, we call the pair $(\{C_{i}\}_{i \in I}, \{A_{j}\}_{j \in J})$ a \emph{seamed pants decomposition}. By abuse of notation, we also denote it by $\mathcal{P}$.
\end{definition}

Infinite-type surfaces are characterized by their genus, number of boundary components and the nested space of ends~\cite{richards1963surface}. From this characterization, we obtain pants decompositions of surfaces that will be used in the proof of Theorem~\ref{thm:aeLengthSpectra}. The construction is apparent in Figure~\ref{fig:pantsDec}; nonetheless, we include a proof in Appendix~\ref{appendix:A1} for the sake of completeness. (See also ~\cite{hernandez2019alex}.)

\begin{proposition}\label{prop:goodPantsDec}
Let $S$ be a topological surface. Then there exist a seamed pants decomposition $\mathcal{P} = (\{C_{i}\}_{i \in I}, \mathcal{A}_{j \in J})$ and finite-type subsurfaces $\{S_{n}\}$ satisfying the following:
\begin{enumerate}
\item $\{S_{n}\}$ is an exhaustion of $S$, i.e., $S_{n} \subseteq S_{n+1}$ for each $n$, $S = \cup_{n} S_{n}$ and each compact subset of $S$ is contained in some $S_{n}$;
\item each of $S_{n}$ is bounded by some $C_{i}$'s and $\mathcal{P}$ restricts to $S_{n}$ as a seamed pants decomposition, and
\item $S_{n+1}$ is made by attaching a generalized pair of pants or a one-holed torus to $S_{n}$ along only one curve.
\end{enumerate}
\end{proposition}

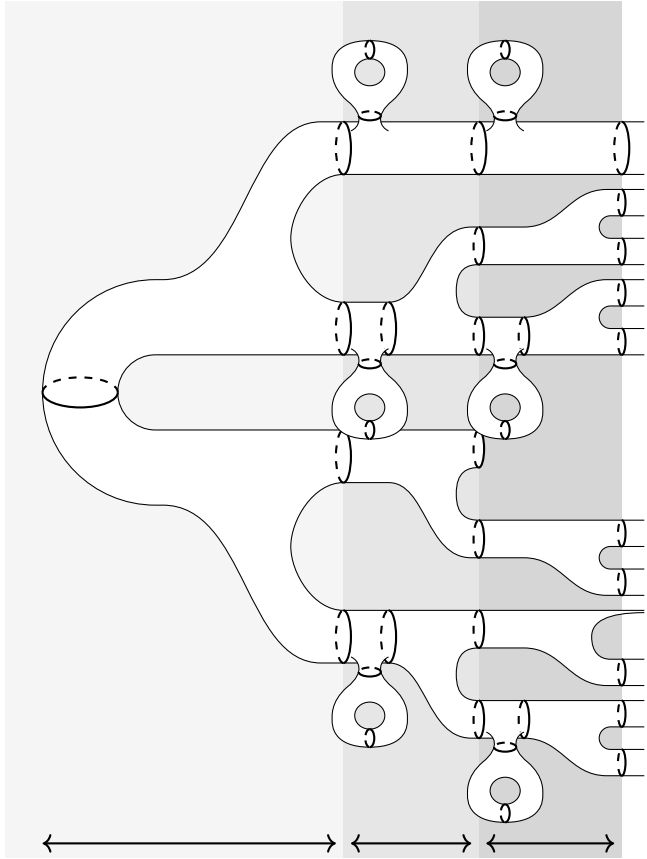
\begin{figure}[ht]
	\begin{tikzpicture}
	
	\fill[black!4] (4, -6.2) -- (-0.5, -6.2) -- (-0.5, 5.2) -- (4, 5.2);
	\fill[black!10] (4, -6.2) -- (5.8, -6.2) -- (5.8, 5.2) -- (4, 5.2);
	\fill[black!16] (5.8, -6.2) -- (7.7, -6.2) -- (7.7, 5.2) -- (5.8, 5.2);
	
	
	\fill[black!0](8, 3.6) -- (3.7, 3.6) .. controls (2.7, 3.6) and (2.7, 1.5) .. (1.6, 1.5) -- (1.5, 1.5) arc (90:270:1.5) -- (1.6, -1.5)  .. controls (2.7, -1.5) and (2.7, -3.6) .. (3.7, -3.6) -- (4.6, -3.6) .. controls (5.1, -3.6) and (5.1, -4.6) .. (5.7, -4.6)--(6.4, -4.6) .. controls (6.9, -4.6) and (7, -5.1) .. (7.5, -5.1) -- (8, -5.1)
	-- (8, -4.75) -- (7.55, -4.75) arc (270:90:0.15) -- (8, -4.45)
	-- (8, -4.1) -- (5.8, -4.1) .. controls (5.65, -4.1) and (5.5, -4.05) .. (5.5, -3.75) .. controls (5.5, -3.45) and (5.65, -3.4) .. (5.8, -3.4) -- (6.4, -3.4) .. controls (6.9, -3.4) and (7, -3.9) .. (7.5, -3.9) -- (8, -3.9)
	-- (8, -3.55) -- (7.55, -3.55) .. controls (7.4, -3.55) and (7.3, -3.45) .. (7.3, -3.25) .. controls (7.3, -3.05) and (7.5, -2.93) .. (8, -2.93)
	--(8, -2.9) -- (4, -2.9) .. controls (3.6, -2.9) and (3.3, -2.4) .. (3.3, -2.05) .. controls (3.3, -1.7) and (3.6, -1.2) .. (4, -1.2) -- (4.6, -1.2) .. controls (5.1, -1.2) and (5.1, -2.2) .. (5.7, -2.2) -- (6.4, -2.2) ..  controls (6.9, -2.2) and (7, -2.7) .. (7.5, -2.7) -- (8, -2.7)
	-- (8, -2.35) -- (7.55, -2.35) arc (270:90:0.15) -- (8, -2.05)
	-- (8, -1.7) -- (5.8, -1.7) .. controls (5.65, -1.7) and (5.5, -1.65) .. (5.5, -1.35) .. controls (5.5, -1.05) and (5.65, -1) .. (5.8, -1) arc (-90:90:0.07 and 0.25) -- (1.5, -0.5) arc (270:90:0.5) -- (8, 0.5)
	-- (8, 0.85) -- (7.55, 0.85) arc (270:90:0.15) -- (8, 1.15) -- (8, 1.5) -- (7.5, 1.5) ..controls (7, 1.5) and (6.9, 1) .. (6.4, 1) -- (5.8, 1) .. controls (5.65, 1) and (5.5, 1.05) .. (5.5, 1.35) .. controls (5.5, 1.65) and (5.65, 1.7) .. (5.8, 1.7) -- (8, 1.7)
	-- (8, 2.05) -- (7.55, 2.05) arc (270:90:0.15) -- (8, 2.35)
	-- (8, 2.7) -- (7.5, 2.7) .. controls (7, 2.7) and (6.9, 2.2) .. (6.4, 2.2) -- (5.7, 2.2) .. controls (5.1, 2.2) and (5.1, 1.2) .. (4.6, 1.2) -- (4, 1.2) .. controls (3.6, 1.2) and (3.3, 1.7) .. (3.3, 2.05) .. controls (3.3, 2.4) and (3.6, 2.9) .. (4, 2.9) -- (8, 2.9);
	
	
	\draw (8, 3.6) -- (3.7, 3.6) .. controls (2.7, 3.6) and (2.7, 1.5) .. (1.6, 1.5) -- (1.5, 1.5) arc (90:270:1.5) -- (1.6, -1.5)  .. controls (2.7, -1.5) and (2.7, -3.6) .. (3.7, -3.6) -- (4.6, -3.6) .. controls (5.1, -3.6) and (5.1, -4.6) .. (5.7, -4.6)--(6.4, -4.6) .. controls (6.9, -4.6) and (7, -5.1) .. (7.5, -5.1) -- (8, -5.1);
	\draw (8, 0.5) -- (1.5, 0.5) arc (90:270:0.5) -- (5.8, -0.5);

	\draw (8, 2.9) -- (4, 2.9) .. controls (3.6, 2.9) and (3.3, 2.4) .. (3.3, 2.05) .. controls (3.3, 1.7) and (3.6, 1.2) .. (4, 1.2) -- (4.6, 1.2) .. controls (5.1, 1.2) and (5.1, 2.2) .. (5.7, 2.2) -- (6.4, 2.2) ..  controls (6.9, 2.2) and (7, 2.7) .. (7.5, 2.7) -- (8, 2.7);
	
	\draw[yscale=-1]  (8, 2.9) -- (4, 2.9) .. controls (3.6, 2.9) and (3.3, 2.4) .. (3.3, 2.05) .. controls (3.3, 1.7) and (3.6, 1.2) .. (4, 1.2) -- (4.6, 1.2) .. controls (5.1, 1.2) and (5.1, 2.2) .. (5.7, 2.2) -- (6.4, 2.2) ..  controls (6.9, 2.2) and (7, 2.7) .. (7.5, 2.7) -- (8, 2.7);

	\draw (8, 1.7) -- (5.8, 1.7) .. controls (5.65, 1.7) and (5.5, 1.65) .. (5.5, 1.35) .. controls (5.5, 1.05) and (5.65, 1) .. (5.8, 1) -- (6.4, 1) .. controls (6.9, 1) and (7, 1.5) .. (7.5, 1.5) -- (8, 1.5);
	
	\draw[yscale=-1] (8, 1.7) -- (5.8, 1.7) .. controls (5.65, 1.7) and (5.5, 1.65) .. (5.5, 1.35) .. controls (5.5, 1.05) and (5.65, 1) .. (5.8, 1);

	\draw (8, 2.35) -- (7.55, 2.35) arc (90:270:0.15) -- (8, 2.05);
	\draw (8, 1.15) -- (7.55, 1.15) arc (90:270:0.15) -- (8, 0.85);
	
	\draw[yscale=-1] (8, 2.35) -- (7.55, 2.35) arc (90:270:0.15) -- (8, 2.05);

	\draw (8, -4.1) -- (5.8, -4.1) .. controls (5.65, -4.1) and (5.5, -4.05) .. (5.5, -3.75) .. controls (5.5, -3.45) and (5.65, -3.4) .. (5.8, -3.4) -- (6.4, -3.4) .. controls (6.9, -3.4) and (7, -3.9) .. (7.5, -3.9) -- (8, -3.9);
	
	\draw[yscale=-1] (8, 4.75) -- (7.55, 4.75) arc (90:270:0.15) -- (8, 4.45);
	\draw (8, -2.93) .. controls (7.5, -2.93) and (7.3, -3.05) .. (7.3, -3.25) .. controls (7.3, -3.45) and (7.4, -3.55) .. (7.55, -3.55) -- (8, -3.55);


	\draw[thick] (0, 0) arc (180:360:0.5 and 0.2);
	\draw[thick, dashed] (0, 0) arc (180:0:0.5 and 0.2);

	\draw[thick] (4, 2.9) arc (-90:90:0.1 and 0.35);
	\draw[thick, dashed] (4, 2.9) arc (270:90:0.1 and 0.35);

	\draw[thick] (4, 0.5) arc (-90:90:0.1 and 0.35);
	\draw[thick, dashed] (4, 0.5) arc (270:90:0.1 and 0.35);
	
	\draw[thick] (4.6, 0.5) arc (-90:90:0.1 and 0.35);
	\draw[thick, dashed] (4.6, 0.5) arc (270:90:0.1 and 0.35);
	
	\draw[thick] (4, -1.2) arc (-90:90:0.1 and 0.35);
	\draw[thick, dashed] (4, -1.2) arc (270:90:0.1 and 0.35);
	
	\draw[thick] (4, -3.6) arc (-90:90:0.1 and 0.35);
	\draw[thick, dashed] (4, -3.6) arc (270:90:0.1 and 0.35);
	
	\draw[thick] (4.6, -3.6) arc (-90:90:0.1 and 0.35);
	\draw[thick, dashed] (4.6, -3.6) arc (270:90:0.1 and 0.35);

	\draw[thick] (5.8, 2.9) arc (-90:90:0.1 and 0.35);
	\draw[thick, dashed] (5.8, 2.9) arc (270:90:0.1 and 0.35);

	\draw[thick] (7.7, 2.9) arc (-90:90:0.1 and 0.35);
	\draw[thick, dashed] (7.7, 2.9) arc (270:90:0.1 and 0.35);

	
	\draw[thick] (5.8, 1.7) arc (-90:90:0.07 and 0.25);
	\draw[thick, dashed] (5.8, 1.7) arc (270:90:0.07 and 0.25);

	\draw[thick] (5.8, 0.5) arc (-90:90:0.07 and 0.25);
	\draw[thick, dashed] (5.8, 0.5) arc (270:90:0.07 and 0.25);
	
	\draw[thick] (6.4, 0.5) arc (-90:90:0.07 and 0.25);
	\draw[thick, dashed] (6.4, 0.5) arc (270:90:0.07 and 0.25);

	\draw[thick] (5.8, -1) arc (-90:90:0.07 and 0.25);
	\draw[thick, dashed] (5.8, -1) arc (270:90:0.07 and 0.25);
	
	\draw[thick] (5.8, -2.2) arc (-90:90:0.07 and 0.25);
	\draw[thick, dashed] (5.8, -2.2) arc (270:90:0.07 and 0.25);
	
	\draw[thick] (5.8, -3.4) arc (-90:90:0.07 and 0.25);
	\draw[thick, dashed] (5.8, -3.4) arc (270:90:0.07 and 0.25);
	
	\draw[thick] (5.8, -4.6) arc (-90:90:0.07 and 0.25);
	\draw[thick, dashed] (5.8, -4.6) arc (270:90:0.07 and 0.25);
	
	\draw[thick] (6.4, -4.6) arc (-90:90:0.07 and 0.25);
	\draw[thick, dashed] (6.4, -4.6) arc (270:90:0.07 and 0.25);
	
	\foreach \i in {2.35, 1.7, 1.15, 0.5, -2.05, -2.7, -3.9, -4.45, -5.1}{
		\draw[thick] (7.7, \i) arc (-90:90:0.05 and 0.175);
		\draw[thick, dashed] (7.7, \i) arc (270:90:0.05 and 0.175);
	}

	\begin{scope}[shift={(-0.35, -0.02)}]
	\draw[fill=black!0] (4.45, -3.5) .. controls (4.55, -3.54) and (4.65, -3.7) .. (4.4, -3.9) .. controls (4.17, -4.1) and (4.2, -4.3) .. (4.2, -4.35) .. controls (4.2, -4.4) and (4.2, -4.7) .. (4.7, -4.7) .. controls (5.2, -4.7) and (5.2, -4.4) .. (5.2, -4.35) .. controls (5.2, -4.3) and (5.23, -4.1) .. (5, -3.9) .. controls (4.75, -3.7) and (4.85, -3.54)..  (4.95, -3.5);
	\draw[fill=black!10] (4.7, -4.28) circle (0.2 and 0.18);
	\draw[thick] (4.847, -3.7) arc (0:-180: 0.147 and 0.06);
	\draw[thick, densely dashed] (4.847, -3.7) arc (0:180: 0.147 and 0.06);
	\draw[thick] (4.7, -4.46) arc (90:-90:0.06 and 0.12);
	\draw[thick, densely dashed] (4.7, -4.46) arc (90:270:0.06 and 0.12);
	\end{scope}
	
	\begin{scope}[shift={(-0.35+1.8, -0.02-1)}]
	
	\draw[fill=black!0] (4.45, -3.5) .. controls (4.55, -3.54) and (4.65, -3.7) .. (4.4, -3.9) .. controls (4.17, -4.1) and (4.2, -4.3) .. (4.2, -4.35) .. controls (4.2, -4.4) and (4.2, -4.7) .. (4.7, -4.7) .. controls (5.2, -4.7) and (5.2, -4.4) .. (5.2, -4.35) .. controls (5.2, -4.3) and (5.23, -4.1) .. (5, -3.9) .. controls (4.75, -3.7) and (4.85, -3.54)..  (4.95, -3.5);
	\draw[fill=black!16] (4.7, -4.28) circle (0.2 and 0.18);
	\draw[thick] (4.847, -3.7) arc (0:-180: 0.147 and 0.06);
	\draw[thick, densely dashed] (4.847, -3.7) arc (0:180: 0.147 and 0.06);
	\draw[thick] (4.7, -4.46) arc (90:-90:0.06 and 0.12);
	\draw[thick, densely dashed] (4.7, -4.46) arc (90:270:0.06 and 0.12);
	\end{scope}
	
	\begin{scope}[shift={(-0.35, -0.02+4.1)}]
	\draw[fill=black!0] (4.45, -3.5) .. controls (4.55, -3.54) and (4.65, -3.7) .. (4.4, -3.9) .. controls (4.17, -4.1) and (4.2, -4.3) .. (4.2, -4.35) .. controls (4.2, -4.4) and (4.2, -4.7) .. (4.7, -4.7) .. controls (5.2, -4.7) and (5.2, -4.4) .. (5.2, -4.35) .. controls (5.2, -4.3) and (5.23, -4.1) .. (5, -3.9) .. controls (4.75, -3.7) and (4.85, -3.54)..  (4.95, -3.5);
	\draw[fill=black!10] (4.7, -4.28) circle (0.2 and 0.18);
	\draw[thick] (4.847, -3.7) arc (0:-180: 0.147 and 0.06);
	\draw[thick, densely dashed] (4.847, -3.7) arc (0:180: 0.147 and 0.06);
	\draw[thick] (4.7, -4.46) arc (90:-90:0.06 and 0.12);
	\draw[thick, densely dashed] (4.7, -4.46) arc (90:270:0.06 and 0.12);
	\end{scope}
	
	\begin{scope}[shift={(-0.35+1.8, -0.02+4.1)}]
	\draw[fill=black!0] (4.45, -3.5) .. controls (4.55, -3.54) and (4.65, -3.7) .. (4.4, -3.9) .. controls (4.17, -4.1) and (4.2, -4.3) .. (4.2, -4.35) .. controls (4.2, -4.4) and (4.2, -4.7) .. (4.7, -4.7) .. controls (5.2, -4.7) and (5.2, -4.4) .. (5.2, -4.35) .. controls (5.2, -4.3) and (5.23, -4.1) .. (5, -3.9) .. controls (4.75, -3.7) and (4.85, -3.54)..  (4.95, -3.5);
	\draw[fill=black!16] (4.7, -4.28) circle (0.2 and 0.18);
	\draw[thick] (4.847, -3.7) arc (0:-180: 0.147 and 0.06);
	\draw[thick, densely dashed] (4.847, -3.7) arc (0:180: 0.147 and 0.06);
	\draw[thick] (4.7, -4.46) arc (90:-90:0.06 and 0.12);
	\draw[thick, densely dashed] (4.7, -4.46) arc (90:270:0.06 and 0.12);
	\end{scope}
	
	\begin{scope}[shift={(-0.35, -0.02)}, yscale=-1]
	\draw[fill=black!0] (4.45, -3.5) .. controls (4.55, -3.54) and (4.65, -3.7) .. (4.4, -3.9) .. controls (4.17, -4.1) and (4.2, -4.3) .. (4.2, -4.35) .. controls (4.2, -4.4) and (4.2, -4.7) .. (4.7, -4.7) .. controls (5.2, -4.7) and (5.2, -4.4) .. (5.2, -4.35) .. controls (5.2, -4.3) and (5.23, -4.1) .. (5, -3.9) .. controls (4.75, -3.7) and (4.85, -3.54)..  (4.95, -3.5);
	\draw[fill=black!10] (4.7, -4.28) circle (0.2 and 0.18);
	\draw[thick] (4.847, -3.7) arc (0:180: 0.147 and 0.06);
	\draw[thick, densely dashed] (4.847, -3.7) arc (0:-180: 0.147 and 0.06);
	\draw[thick] (4.7, -4.46) arc (90:-90:0.06 and 0.12);
	\draw[thick, densely dashed] (4.7, -4.46) arc (90:270:0.06 and 0.12);
	\end{scope}
	
	\begin{scope}[shift={(-0.35+1.8, -0.02)}, yscale=-1]
	\draw[fill=black!0] (4.45, -3.5) .. controls (4.55, -3.54) and (4.65, -3.7) .. (4.4, -3.9) .. controls (4.17, -4.1) and (4.2, -4.3) .. (4.2, -4.35) .. controls (4.2, -4.4) and (4.2, -4.7) .. (4.7, -4.7) .. controls (5.2, -4.7) and (5.2, -4.4) .. (5.2, -4.35) .. controls (5.2, -4.3) and (5.23, -4.1) .. (5, -3.9) .. controls (4.75, -3.7) and (4.85, -3.54)..  (4.95, -3.5);
	\draw[fill=black!16] (4.7, -4.28) circle (0.2 and 0.18);
	\draw[thick] (4.847, -3.7) arc (0:180: 0.147 and 0.06);
	\draw[thick, densely dashed] (4.847, -3.7) arc (0:-180: 0.147 and 0.06);
	\draw[thick] (4.7, -4.46) arc (90:-90:0.06 and 0.12);
	\draw[thick, densely dashed] (4.7, -4.46) arc (90:270:0.06 and 0.12);
	\end{scope}
	
	
	\draw[thick, <->] (0, -6) -- (3.9, -6);
	\draw[thick, <->] (4.1, -6) -- (5.7, -6);
	\draw[thick, <->] (5.9, -6) -- (7.6, -6);

	\end{tikzpicture}
	\caption{Pants decomposition of a surface of infinite type.} \label{fig:infinitesurface}
	\label{fig:pantsDec}
\end{figure}

\begin{notation} \label{not.I0}
From now on, $S$ shall be reserved for a surface with the non-abelian fundamental group, equipped with a seamed pants decomposition $\mathcal{P} = (\{C_{i}\}_{i \in I}, \{A_{j}\}_{j \in J})$ obtained from Proposition~\ref{prop:goodPantsDec}. Moreover, $I_{0} \subseteq I$ denotes the set of indices corresponding to the boundary components of $S$.
\end{notation}

A hyperbolic surface is a 2-dimensional Riemannian manifold, possibly with compact geodesic boundary, of constant curvature $-1$. Subsurfaces and generalized subsurfaces of a hyperbolic surface are always assumed to have geodesic boundary. A hyperbolic surface is \emph{convex} if every arc is homotoped to a geodesic arc, fixing endpoints. Convex hyperbolic surfaces are obtained as a quotient of a convex subset of $\mathbb{H}$ by free, properly discontinuous action of a subgroup of $\textrm{Isom}^{+} (\mathbb{H})$. A convex hyperbolic surface is \emph{Nielsen-convex} if every point is contained in a (possibly non-simple) geodesic segment whose endpoints lie on simple closed geodesics. Such hyperbolic structures are suitable for our purpose due to the following fact. 

\begin{fact}[Theorem 4.5,~\cite{alessandrini2011fenchel}] \label{fact.nielsenconvex}
Let $X$ be a hyperbolic surface. Then the following facts are equivalent: \begin{enumerate}
\item $X$ is obtained by gluing some hyperbolic pairs of pants along their boundary components.
\item $X$ is Nielsen-convex.
\item Every topological pair of pants decomposition of $S$ by a system of curves $\{C_{i}\}_{i}$ is isotopic to a geometric pair of pants decomposition (i.e. if $\gamma_{i}$ is the simple closed geodesic on S that is freely homotopic to $C_{i}$, then $\{\gamma_{i}\}_{i}$ defines a pair of pants decompositon).
\end{enumerate}
\end{fact}

Especially, a Nielsen-convex hyperbolic surface cannot contain a funnel or a hyperbolic half-plane. Thus, all isolated ends are punctures, that means, quotients of $\mathbb{H}$ by parabolic elements. Conversely, a finite-type hyperbolic surface is Nielsen-convex if it does not contain funnels, or equivalently, if it is has finite area. As a result, finite-type subsurfaces of a Nielsen-convex hyperbolic surface are again Nielsen-convex.

Let $X$ be a Nielsen-convex (hence convex) hyperbolic surface. A curve $C$ on $X$ not bounding a puncture has a unique geodesic representative. We denote its length by $l_{X}(C)$. By an abuse of notation, we sometimes refer to the geodesic representative as $C$. Similarly, each arc $A$ on $X$ attains a unique geodesic representative. If $C$ is bounding a puncture, it does not have a geodesic representative, and we conventionally set $l_{X}(C)$ by 0. Instead, it is associated to a representative called \emph{horocycle}, a simple loop around a cusps with curvature 1. Then every geodesic arc $A$ emanating from that puncture intersects with the horocycle perpendicularly.

For a hyperbolic surface $X$, we denote by $\mathsf{Sim}$ the set of essential or boundary curves and define its \emph{marked length spectrum} $\mathcal{L}^{m}(X) \in \mathbb{R}^{\mathsf{Sim}}$ by the function sending each essential or boundary curve $C$ on $X$ to its length $l_{X}(C)$. The \emph{(unmarked) length spectrum} $\mathcal{L}(X)$ is the unordered set of curve lengths on $X$ counting multiplicities. If we consider the quotient map $\varphi : \mathbb{R}^{\mathsf{Sim}} \rightarrow \mathbb{R}^{\mathsf{Sim}}/\operatorname{Sym}(\mathsf{Sim})$ by permutations $\operatorname{Sym}(\mathsf{Sim})$, then $\mathcal{L}(X)$ is the image of $\mathcal{L}^{m}(X)$ under $\varphi$.

\subsection{Teichm{\"u}ller space and moduli space}

Teichm{\"u}ller space of $S$ can be defined in various ways. Those definitions are compatible if the base surface is of finite type, but may differ if the base surface is of infinite type. For details, see~\cite{FLP},~\cite{imayoshi1992Teich},~\cite{hubbard2006Teich} or~\cite{alessandrini2011fenchel}. Our definition follows:

\begin{definition} \label{def.Teich}
The \emph{Teichm{\"u}ller space} $\T(S)$ of $S$ is the set of equivalence classes $[h, X]$ of pairs $(h, X)$, where $X$ is a Nielsen-convex hyperbolic surface of topological type $S$ and $h: S\rightarrow X$ is a homeomorphism. Here, two pairs $(h, X)$ and $(h', Y)$ are considered equivalent if $h' \circ h^{-1}$ is homotopic to an isometry.

The \emph{(extended) mapping class group} $\Mod^{\pm}(S)$ of $S$ is the set of equivalence classes $[\varphi]$ of self-homeomorphism $\varphi$ on $S$, where $\varphi$ and $\phi$ are considered equivalent if $\phi \circ \varphi^{-1}$ is isotopic to the identity.

The \emph{moduli space} $\mathcal{M}(S)$ is the set of Nielsen-convex hyperbolic surfaces of topological type $S$.
\end{definition}

We note that $\Mod^{\pm}(S)$ acts on $\T(S)$ by pre-composition, and the quotient of $\T(S)$ by $\Mod^{\pm}(S)$ is equal to $\mathcal{M}(S)$.

In contrast to the case of finite-type surfaces, there are several (different) ways to give a topology on the Teichm\"uller space of infinite-type surfaces. Since our argument deals with finitely many curves at one time, our argument works for any topology on the Teichm\"uller space satisfying the following property. We will come up with a natural topology on $\T(S)$ that satisfies the following property:

\begin{property} \label{property:topology}
	For each finite-type subsurface $S_1 \subseteq S$, the Teichmu{\"u}ller space $\T(S)$ of the ambient surface $S$ is expressed as the product of the Teichm{\"u}ller space $\T(S_{1})$ of $S_{1}$ and some other space. That means, there exists a topological space $\T(S;S_1)$ such that $\T(S)$ is homeomorphic to $\T(S_{1}) \times \T(S; S_{1})$. We denote the projection to each factor by $\pi_{S_1} : \T(S) \to \T(S_1)$ and $\pi_{S;S_1}: \T(S) \to \T(S;S_1)$.

\end{property}

An example of such a topology on $\T(S)$ satisfying Property~\ref{property:topology} can be constructed by means of \emph{Fenchel-Nielsen coordinates}. To elaborate, we first define the length and twist parameters on $\T(S)$ from a given pants decomposition on $S$.

Recall that $S$ is equipped with a seamed pants decomposition $\mathcal{P}=(\{C_{i}\}, \{A_{j}\})$ and let $[h, X] \in \T(S)$. The pants decomposition $\mathcal{P}$ induces a (topological) seamed pants decomposition $\mathcal{P}' = (\{C'_{i} := h(C_{i})\}, \{A'_{j} := h(A_{j})\})$ on $X$. Since $X$ is Nielsen-convex, $\{C'_{i}\}$ can be considered a geometric pants decomposition. We call $l_{X}(C'_{i})$ the $i$-th \emph{length parameter} of $[h, X]$.

We now construct twist parameters of $[h, X]$. Twist parameters are assigned to $C_i$'s which are not boundary components of $S$ as there is nothing to twist on the boundary of $S$. Suppose that $C_i$ is not a boundary component of $S$, hence $C_i'$ is not a boundary component of $X$. We will consider the signed length of a segment in $C_i'$ where the the signed distance along $C_i'$ is defined using the orientation of the surface in a way that the (right) Dehn twist corresponds to the positive direction (cf. Subsection~\ref{subsec:fractionalDehn}).

We choose an arc $A_j'$ that intersects $C_i'$. Along the arc $A_j'$, $A_j'$ passes through pairs of pants which are components of $X - \cup_{i} C_i'$. Fixing any orientation on $A_j'$, we enumerate the pairs of pants as $\cdots, P_{-1}, P_0, P_1, \cdots$ so that $C_i'$ is the intersection of $P_{-1}$ and $P_0$. This enumeration is finite if the arc $A_j'$ is compact and is infinite if $A_j'$ is contained in an end of $X$. For each $P_{t}$, $A_{j}' \cap P_{t}$ connects two of the three boundary components of $P_{t}$. Let $L_{t}$ be the simple geodesic segment on $P_{t}$ perpendicular to those two boundary components.  Then for each $t$, $L_{t-1}$ and $L_{t}$ divide the intersection $P_{t-1} \cap P_{t}$ into two (possibly degenerate) simple geodesic segments. We choose one of them and denote it by $K_t$ for each $t$ so that $A_j'$ is homotopic to a concatenation $\cdots L_{t-1} K_{t}L_t K_{t+1} L_{t+1} \cdots$. Noting that $C_i' = P_{-1} \cap P_0$, we define the $i$-th \emph{twist parameter} $\tau_X(C_i')$ of $X$ as the signed length of $K_0$ along $C_i'$ divided by $l_X(C_i')$. There are a priori two choices of twist parameter at $C'_{i}$, one defined with $A'_{j}$ and one defined with another arc $A'_{j'}$ passing through $C'_{i}$. Nonetheless, two values are always equal so no confusion occurs.

Using the Fenchel-Nielsen parameters, we can construct a bijection \[
FN: \T(S) \ni [h, X] \mapsto \left(  \left(\log l_{X}(C_{i}')\right)_{i \in I},  \left( l_{X}(C_{i}') \tau_{X}(C_{i}')\right)_{i \in I \setminus I_{0}} \right)
\]
between $\T(S)$ and $\mathbb{R}^{I} \times \mathbb{R}^{I \setminus I_{0}}$. (Recall $I$, $I_0$ from Notation \ref{not.I0}.) We now endow the space with the $l^{\infty}$-topology. We consider the following:

\begin{definition}
We define the \emph{Fenchel-Nielsen distance} $d_{FN}$ between two points $[h, X], [h', Y] \in \T(S)$ by \[\begin{aligned}
&d_{FN}([h, X], [h', Y]) :=\\ &\sup_{i \in I} \left\lbrace \left| \log \frac{l_{X}(h(C_{i}))}{l_{Y}(h'(C_{i}))} \right|, \left| l_{X}(h(C_{i})) \tau_{X}(h(C_{i})) - l_{Y}(h'(C_{i})) \tau_{Y}(h'(C_{i}))\right|\right\rbrace
\end{aligned}
\] where we set $\tau_{X}(h(C_i)) = \tau_{Y}(h'(C_i)) = 0$ for $i \in I_0$.
\end{definition}

Then the balls $B_{r}([h, X]) := \{[h', Y] \in \T(S) : d_{FN}([h, X], [h', Y]) < r\}$ generate the $l^{\infty}$-topology on $\T(S)$. If $S$ is of infinite type, this space contains uncountably many components, each comprised of elements $d_{FN}$-bounded to each other.

\begin{remark}
The Fenchel-Nielsen Teichm{\"u}ller space was originally constructed in~\cite{alessandrini2011fenchel}, which uses different conventions. Precisely, the convention in~\cite{alessandrini2011fenchel} picks a hyperbolic structure $[h, X_{0}]$ and considers the component of $\T(S)$ containing $[h, X_{0}]$ as the Fenchel-Nielsen Teichm{\"u}ller space $\T(X_{0})$. To see how the choice of basepoint $X_{0}$ affects the property of $\T(X_{0})$, see~\cite{alessandrini2011fenchel}.
\end{remark}

When $S$ is of finite type, $\T(S)$ is homeomorphic to a finite-dimensional Euclidean space, and there are several equivalent definitions of $\T(S)$. For instance, $\T(S)$ can be identified with the set of all discrete faithful representations $\pi_1(S) \to \PR$, modulo conjugations by $\PR$, which send elements represented by curves freely homotopic to punctures to parabolic elements.
Especially, the Fenchel-Nielsen parameters for different seamed pants decompositions give rise to the same analytic structure on $\T(S)$. The length of a curve on $S$ then becomes an analytic function on $\T(S)$.

To see that the Fenchel-Nielsen topology satisfies Property \ref{property:topology}, consider a surface $S$ made by gluing a finite-type surface $S_{1}$ with another surface $S_{2}$ along curves $\{C_{i}\}_{i \in I_{1}}$. Then pants decompositions $\{C_{i}\}_{i \in I_{1} \cup I_{2}}$ on $S_{1}$ and $\{C_{i}\}_{i \in I_{1} \cup I_{3}}$ on $S_{2}$ give rise to a pants decomposition $\{C_{i}\}_{i \in I}$ on $S$, where $I = I_{1} \cup I_{2} \cup I_{3}$. We can further construct a seamed pants decomposition $\mathcal{P} = (\{C_{i}\}_{i \in I}, \{A_{j}\}_{j \in J})$ on $S$, which gives seamed pants decomposition $\mathcal{Q}$ and $\mathcal{R}$ on $S_{1}$ and $S_{2}$, respectively, by restriction. Then the Fenchel-Nielsen parametrization of $\T(S)$ is decomposed into that of $\T(S_{1})$, that of $\T(S_{2})$ (omitting the length parameters for $I_{1}$), and twist parameters for $I_{1}$. By setting $\T(S;S_{1})$ to be the parameter space consisting of Fenchel-Nielsen parameters for $S_2$ together with twist parameters for $I_1$, $\T(S)$ is the product space of $\T(S_{1})$ and $\T(S;S_{1})$ as desired.

\subsection{Intersection number}
Let us first assume that $\alpha$, $\beta$ are oriented multicurves on $S$. Let $A$, $B$ be smooth representatives of $\alpha$ and $\beta$, respectively, transverse to each other at a point $p$. Let $A_{p}, B_{p}$ be the tangent vectors along $A$ and $B$ at $p$, respectively. The \emph{index} of $(A, B)$ at $p$ is defined as $+1$ if the oriented basis $(A_{p}, B_{p})$ agree with the orientation of the surface, and $-1$ otherwise. We further define the \emph{algebraic intersection number} of $\alpha$ and $\beta$ by the sum of indices of $(A, B)$ over all intersection points, and denote it by $i_{alg}(\alpha, \beta)$. The algebraic intersection number does not depend on the choice of representatives $A$ and $B$. It does, however, depend on the orientations of $\alpha$ and $\beta$ and is well-defined up to sign.

However, the total number of intersection points depends on the choice of representatives. The minimum such number, counted with multiplicity, is called the \emph{geometric intersection number} and denoted by $i_{geom}(\alpha, \beta)$ (or $i(\alpha, \beta)$ for short). Note that geometric intersection number is also well-defined for unoriented multicurves. Representatives $A$, $B$ of $\alpha$, $\beta$ realizing $i(\alpha, \beta)$ are said to be \emph{in minimal position}. The following fact serves as a practical criterion for representative curves in minimal position.

\begin{fact}{\cite[Proposition 3.10]{FLP}}\label{fact:bigon}
Representatives $A$, $B$ of two multicurves are in minimal position if and only if $A$ and $B$ do not form a bigon, a contractible region of $S \setminus (A \cup B)$ bounded by one simple segment of $A$ and one simple segment of $B$.
\end{fact} 

We introduce an abuse of notation as follows: curves $\alpha$ and $\beta$ may also refer to representatives $A$ and $B$ of $\alpha$ and $\beta$, respectively, in minimal position. Such representatives are chosen up to simultaneous ambient isotopy as described follows:

\begin{fact}{\cite[Lemma 2.9]{FarbMargalit12}}\label{fact:simIsotopy}
Let $S$ be a finite-type surface, $\gamma_{1}$, $\gamma_{2}$ be distinct essential curves on $S$, and $c_{i}$, $c_{i}'$ be representatives of $\gamma_{i}$. Then there exists an isotopy of $S$ that takes $c_{i}'$ to $c_{i}$ for both $i$ simultaneously.
\end{fact}

The intersection numbers with finitely many curves are sufficient to determine a multicurve \cite[Section 4.3]{FLP}. This fact is due to Dehn and Thurston (see e.g., \cite{luo2004dehn-thurston} for the context). We record one variant suited for our purpose.

\begin{fact}{cf. \cite[ Th{\'e}or{\`e}me 4.8]{FLP}} \label{fact:DehnThurston}
Let $\{B_{1}, \ldots, B_{m}, C_{1}, \ldots, C_{n}\}$ be a pants decomposition of a finite-type surface, where $B_{i}$'s are boundary curves. Then there exist curves $\{C_{1}', \ldots, C_{n}', C_{1}'', \ldots, C_n''\}$ on the surface satisfying the following: \begin{enumerate}
\item $i(C_{i}, C_{j}') = 0 \Leftrightarrow i \neq j \Leftrightarrow i(C_{i}, C_{j}'') = 0$ , and
\item if $D$, $D'$ are distinct essential multicurves (i.e., not containing boundary curves), then we have $i(D, C) \neq i(D', C)$ for at least one $C \in \{C_{i}, C_{i}', C_{i}''\}$.
\end{enumerate}
\end{fact}

Here $C_{i}'$ and $C_{i}''$ are used to measure the `twist' of multicurves along $C_{i}$'s. See also Fact~\ref{fact:onlyTwist}.

\subsection{Pinching a curve}

Let $\alpha$ be a curve on $S$. Since $\alpha$ is compact, it is contained in a finite-type subsurface $S_{1}$. For each $[h, X] \in \T(S)$. the curve $f(\alpha)$ has its geodesic representative $A$ on $X$ realising the minimum length. Abusing notation, we omit the marking and denote $l_{X}(A)$ by $l_{X}(\alpha)$. Then the function $l_{X}(\alpha)$ is continuous on $\T(S)$ and descends to an analytic function on $\T(S_{1})$.

We state a lemma regarding the pinching process, whose proof is deferred to Appendix~\ref{section:pinching}. The pinching process is, roughly speaking, choosing a simple closed curve and then making its length to converge to 0. For a detailed discussion, see~\cite{wolpert1990hyp}.

Let $\{C_{1}, C_{2}, \ldots \}$ be a pants decomposition on $S$ and $X \in \T(S)$. Pinching the length of $C_{1}$ means that we follow the path $\{X_{r}\}_{r > 0} \subseteq \T(S)$ as $r \to 0$ where \[
l_{X_{r}}(C_{i}) = \left\{ \begin{array}{cc} r & i = 1 \\ l_{X}(C_{i}) & i \neq 1 \end{array}\right., \quad \tau_{X_{r}}(C_{i}) = \tau_{X}(C_{i}) \quad \textrm{for all}\,\, i.
\]

\begin{restatable}{lem}{pinching} \label{lem:pinching}
Let $\alpha$ be a multicurve on $S$ with $i(\alpha, C_{1}) = k$. \begin{enumerate}
	\item If $k = 0$, then $l_{X_{r}}(\alpha)$ converges to a finite value as $r \rightarrow 0$.
	\item If $k > 0$, then $\lim_{r\rightarrow 0} l_{X_{r}}(\alpha) / \ln r = -2k$.
\end{enumerate}
\end{restatable}	

The proof of Lemma \ref{lem:pinching} will be given in Appendix \ref{appendix:B}.

\subsection{Fractional Dehn twists} \label{subsec:fractionalDehn}

Let $\alpha$ be a curve and $\beta$ be a multicurve with $i(\alpha, \beta) = k$. We choose their representatives to be curves in minimal position and
denote these curves respectively by $\alpha$ and $\beta$ by abusing notation. To define the fractional Dehn twist $T_{\alpha}^j(\beta)$ for $j \in \Z$, let us take an annular neighborhood $N$ of $\alpha$ in a way that
$$N = S^1 \times [-1, 1] = \{(e^{2\pi i s}, t) : s \in [0, 1], t \in [-1, 1]\}$$ where $\alpha$ is parametrized by $[0, 1] \to N$, $s \mapsto (e^{2\pi i s}, 0)$ and $\beta \cap N = \{ (e^{{2\pi n i \over k}}, t) : t \in [-1, 1], n = 1, \ldots, k  \}$.

We now define a homeomorphism $\varphi : N \to N$ by $$\varphi(z, t) = \begin{cases}
(z e^{{2\pi i t \over k}}, t) &, t \in [0, 1]\\
(z, t) &, t \in [-1, 0].
\end{cases}$$ We then extend $\varphi$ to $\Phi$ on the whole surface by setting $\Phi$ to be an identity outside of $N$. Even though $\Phi$ may not be continuous on the surface, $\Phi(\beta)$ is an unoriented multicurve on the surface. We define the fractional Dehn twist $$T_{\alpha}^j (\beta) := \Phi(\beta).$$ Here, $T_{\alpha}^j (\beta)$ is well-defined up to isotopy, thanks to Fact~\ref{fact:simIsotopy}. Note that the superscript notation is consistent with composition. Indeed, we observe that $T_{\alpha}^{ i+j}(\beta) = T_{\alpha}^i(T_{\alpha}^j(\beta))$ for $i, j \in \Z$. Also note that $T_{\alpha}^{k}(\beta)$ precisely defines a right Dehn twist of $\beta$ along $\alpha$. We now record two facts on the intersection number and fractional Dehn twists.

\begin{remark}
Fractional Dehn twists should be distinguished from the roots of Dehn twists that Margalit and Schleimer introduced in~\cite{margalit2009Dehn}. The roots of Dehn twists are mapping classes while  fractional Dehn twists are a priori not induced from a homeomorphism on the surface. As a result, the roots of Dehn twists always send a single curve to another single curve while fractional Dehn twists may send a single curve to multicurves.
\end{remark}

\begin{fact}{\cite[Proposition 3.4]{FarbMargalit12}}\label{fact:triangleIneqInters}
Let $\alpha$, $\beta$, $\gamma$ be curves on a surface $S$ and $i(\alpha, \beta) = k \ge 1$. Then \[
\left| i(T_{\alpha}^{nk}(\beta), \gamma) - nk i(\alpha, \gamma) \right| \le i(\beta, \gamma).
\]
\end{fact}

\begin{lem}\label{lem:fracDehn}
$i(T_{\alpha}^j(\beta), \alpha) = i(\alpha, \beta)$ and $i(T_{\alpha}^j(\beta), \beta) = |j| i(\alpha, \beta)$.
\end{lem}

\begin{proof}
In this proof, we denote by $N(\gamma)$ an annular neighborhood of a curve $\gamma$. 

We temporarily orient $\beta$ and fix a representative $C$ of $T_{\alpha}^{j}(\beta)$ as in Figure~\ref{fig:fracDehn}. Here, segments of $C$ parallel to $\beta$ (called type $B$) are drawn on the left side of $\beta$ if $j$ is positive, and on the right side otherwise.

\begin{figure}[ht]
	\begin{tikzpicture}

	\draw (0, 0.3) -- (0, 6.2);
	\draw[thick, decoration={markings, mark=at position 0.5 with {\draw (-0.18, 0.06) -- (0, 0) -- (-0.18, -0.06);}}, postaction={decorate}] (-3, 5.5) -- (3, 5.5);
	\draw[thick, decoration={markings, mark=at position 0.5 with {\draw (-0.18, 0.06) -- (0, 0) -- (-0.18, -0.06);}}, postaction={decorate}] (-3, 4) -- (3, 4);
	\draw[thick, decoration={markings, mark=at position 0.5 with {\draw (0.18, 0.06) -- (0, 0) -- (0.18, -0.06);}}, postaction={decorate}] (-3, 2.5) -- (3, 2.5);
	\draw[thick, decoration={markings, mark=at position 0.5 with {\draw (-0.18, 0.06) -- (0, 0) -- (-0.18, -0.06);}}, postaction={decorate}] (-3, 1) -- (3, 1);
	\draw[thick, dashed] (-3, 5.7) -- (-0.68, 5.7) -- (0.68, 2.3) -- (3, 2.3);
	\draw[thick, dashed] (-3, 4.2) -- (-0.68, 4.2) -- (0.52, 1.2) -- (3, 1.2);
	\draw[thick, dashed] (-3, 2.3) -- (-0.52, 2.3) -- (0.2, 0.5);
	\draw[thick, dashed] (-3, 1.2) -- (-0.68, 1.2) -- (-0.4, 0.5);
	\draw[thick, dashed] (-0.2, 6) -- (0.52, 4.2) -- (3, 4.2);
	\draw[thick, dashed] (0.4, 6) -- (0.52, 5.7) -- (3, 5.7);

	\draw (0, 0.1) node {$\alpha$};
	
	\draw[thick] (3.5, 1.5) -- (4.7, 1.5);
	\draw[thick, dashed] (3.5, 1) -- (4.7, 1);
	\draw (5, 1.5) node {$\beta$};
	\draw (5.28, 1) node {$T_{\alpha}^{2} (\beta)$};
	
	\end{tikzpicture}
	\caption{$\alpha$, $\beta$ and $T_{\alpha}^2(\beta)$. Here $\beta$ is equipped with an orientation in order to determine a representative of $T_{\alpha}^2(\beta)$ in minimal position with $\alpha$ and $\beta$.}
	\label{fig:fracDehn}
\end{figure}
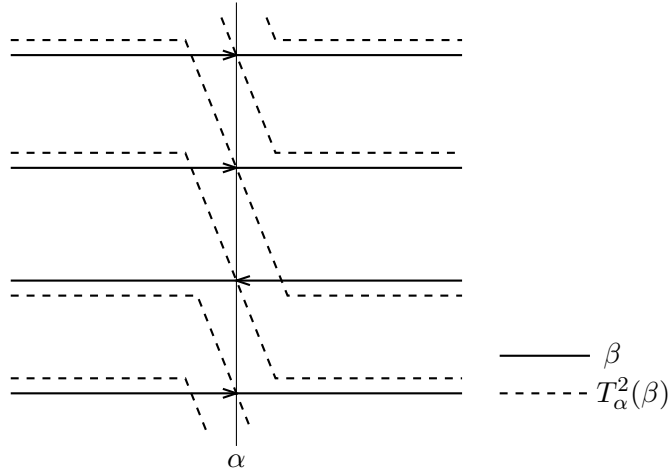

$C$ also has segments in $N(\alpha) \setminus \beta$ (called type $A$), which are classified further into two subtypes: those that are contained in $N(\beta)$ (called type $A_1$) and the others that are not contained in $N(\beta)$ (called type $A_2$). See Figure~\ref{fig:fracDehnSupp}. We observe in Figure~\ref{fig:fracDehnSupp} that \begin{enumerate}
	\item each type $B$ segment is disjoint from $\beta$;
	\item each type $B$ segment either closes itself, or is sandwiched by a type $A_1$ segment and a type $A_2$ segment;
	\item each type $A_1$ segment is adjacent to a type $B$ segment and $\beta$.
\end{enumerate}

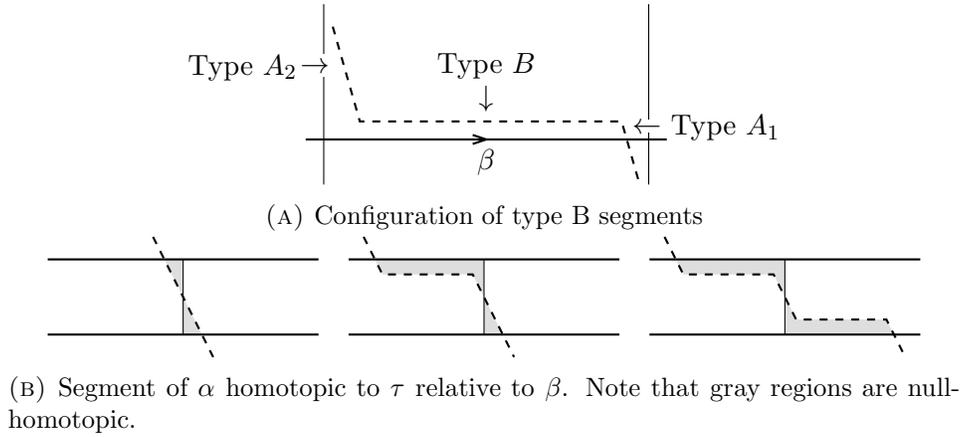
\begin{figure}[ht]

	\begin{subfigure}[c]{\textwidth}
			\centering
	\begin{tikzpicture}
	
	\begin{scope}[scale=1.2]
	\draw[thick, decoration={markings, mark=at position 0.5 with {\draw (-0.18, 0.06) -- (0, 0) -- (-0.18, -0.06);}}, postaction={decorate}] (-2, -0.5) -- (2,-0.5);
	\draw (-1.8, -1) -- (-1.8, 1);
	\draw (1.8, -1) -- (1.8, 1);
	\draw[thick, dashed] (-1.7, 0.75) -- (-1.4, -0.3) -- (1.5, -0.3) -- (1.7, -1);
	
	\fill[black!0] (-1.95, 0.2) -- (-1.7, 0.2) -- (-1.7, 0.45) -- (-1.95, 0.45);
	\draw (-1.9, 0.3) node {$\rightarrow$};
	\draw (-2.7, 0.3) node {Type $A_{2}$};
	\draw (0, -0.05) node {$\downarrow$};
	\draw (0, 0.32) node {Type $B$};
	\fill[black!0] (1.7, -0.42) -- (1.7, -0.28) -- (1.9, -0.28) -- (1.9, -0.42);
	\draw (1.78, -0.37) node {$\leftarrow$};
	\draw (2.65, -0.37) node {Type $A_{1}$};
	\draw (0, -0.75) node {$\beta$};

	\draw[blue, fill=blue!10, opacity=0.4] (-1.7, 0.75) .. controls (-1.85, 0.75) and (-1.55, -0.3) .. (-1.4, -0.3) .. controls (-1.25, -0.3) and (-1.5, 0.75) .. (-1.7, 0.75);
	\draw[red, fill=red!10, opacity=0.4] (1.52, -0.37) circle(0.1);
	\end{scope}
	
	\end{tikzpicture}
	\caption{Configuration of type B segments}
\end{subfigure}

	\begin{subfigure}[c]{\textwidth}
			\centering
	\begin{tikzpicture}
	
	\begin{scope}[shift={(-4, -3.5)}]
	

	\fill[black!13] (-0.25, 0.5) -- (0, 0.5) -- (0, 0) -- cycle;
	\fill[black!13] (0.25, -0.5) -- (0, -0.5) -- (0, 0) -- cycle;

	\draw[thick] (-1.8, 0.5) -- (1.8, 0.5);
	\draw[thick] (-1.8, -0.5) -- (1.8, -0.5);
	\draw[thick, dashed] (-0.4, 0.8) -- (0.4, -0.8);
	
	\draw (0, 0.5) -- (0, -0.5);
	
	\end{scope}
	
	\begin{scope}[shift={(0, -3.5)}]

	\fill[black!13] (-1.45, 0.5) -- (0, 0.5) -- (0, 0) -- (-0.15, 0.3) -- (-1.35, 0.3) -- cycle;
	\fill[black!13] (0.25, -0.5) -- (0, -0.5) -- (0, 0) -- cycle;
	\draw[thick] (-1.8, 0.5) -- (1.8, 0.5);
	\draw[thick] (-1.8, -0.5) -- (1.8, -0.5);
	\draw[thick, dashed] (-1.6, 0.8) --  (-1.35, 0.3) -- (-0.15, 0.3) -- (0.4, -0.8);
	\draw (0, 0.5) -- (0, -0.5);
	\end{scope}
	
	\begin{scope}[shift={(4, -3.5)}]
	\fill[black!13] (-1.45, 0.5) -- (0, 0.5) -- (0, 0) -- (-0.15, 0.3) -- (-1.35, 0.3) -- cycle;
	\fill[black!13] (1.45, -0.5) -- (0, -0.5) -- (0, 0) -- (0.15, -0.3) -- (1.35, -0.3) -- cycle;
	\draw[thick] (-1.8, 0.5) -- (1.8, 0.5);
	\draw[thick] (-1.8, -0.5) -- (1.8, -0.5);
	\draw[thick, dashed] (-1.6, 0.8) --  (-1.35, 0.3) -- (-0.15, 0.3) -- (0.15, -0.3) -- (1.35, -0.3) -- (1.6, -0.8);
	\draw (0, 0.5) -- (0, -0.5);
	\end{scope}

	\end{tikzpicture}
	\caption{Segment of $\alpha$ homotopic to $\tau$ relative to $\beta$. Note that gray regions are null-homotopic.}
\end{subfigure}

	\caption{Configurations of $\beta$ and $C$}
	\label{fig:fracDehnSupp}
\end{figure} 

We claim that the curves in Figure~\ref{fig:fracDehn} are indeed in minimal position. First, any complementary region of $T_{\alpha}^{j}(\beta) \cup \alpha$ can be isotoped to a complementary region of $\alpha \cup \beta$. Since $\alpha$ and $\beta$ are assumed to be in minimal position, such complementary regions are not bigons. Consequently, $T_{\alpha}^{j}(\beta)$ and $\alpha$ are also in minimal position.

We now discuss the minimal position of $T_{\alpha}^{j}(\beta)$ and $\beta$. To this end, suppose to the contrary that a segment $\tau$ of $T_{\alpha}^{j}(\beta)$ and a segment $\sigma$ of $\beta$ bound a bigon.  As observed above (1), each type $B$ segment is disjoint from $\beta$, and hence $\tau$ must contain at least one type $A_1$ or type $A_2$ segment. Moreover, it follows from (2) and (3) above that $\tau$ falls into one of the following (Figure \ref{fig:fracDehnSupp}(B)): \begin{itemize}
\item $\tau$ consists of only one type $A_2$ segment $a_{2}$;
\item $\tau$ is a concatenation of  type $A_1$ segment $a_{1}$, type $B$ segment $b_{1}$ and type $A_2$ segment $a_{2}$; or
\item $\tau$ is a concatenation of type $A_1$ segment $a_{1}$, type $B$ segment $b_{1}$ and type $A_2$ segment $a_{2}$, type $B$ segment $b_{2}$ and type $A_1$ segment $a_{3}$.
\end{itemize}
In any case, $\tau$ is homotopic (relative to $\beta$) to a segment of $\alpha \setminus \beta$. We deduce that $\alpha$ and $\beta$ bound a bigon, contradicting  the minimal position assumption. Thus, we conclude that $T_{\alpha}^{j}(\beta)$ and $\beta$ are also in minimal position.

Given this conclusion, the intersection numbers follow immediately.
\end{proof}

We will make use of the following variant of Fact~\ref{fact:DehnThurston} later on to characterize fractional Dehn twists.

\begin{fact}\label{fact:onlyTwist}
Let $S$ be a surface of finite type and $\{C_{j}, C_{j}', C_{j}''\}_{j=1}^{n}$ be the curves on $S$ mentioned in Fact~\ref{fact:DehnThurston}. Fix $k$ and suppose that $D$, $D'$ are essential multicurves satisfying $i(C_{j}, D) = i(C_{j}, D')$ for all $j$, $i(C_{j}', D) = i(C_{j}', D')$ for $j \neq k$, and $i(C_{j}'', D) = i(C_{j}'', D')$ for $j \neq k$. Then $D$ and $D'$ are related by a fractional  Dehn twist along $C_{k}$.
\end{fact}

\begin{proof}[Sketch of proof]
A step of the proof of \cite[Th{\'e}or{\`e}me 4.8]{FLP} concerns the construction of a model multicurve $\delta$ on $S$ for each admissible value $\{i(D, C_{j}), i(D, C_{j}'), i(D, C_{j}'')\}_{j}$. Here, the information $\{i(D, C_{j})\}_{j}$ determines the relative isotopy class of this model multicurve $\delta$ restricted to $S \setminus (\cup_{j} N(C_{j}))$. Furthermore, for each $j$, the information $(i(D, C_{j}'), i(D, C_{j}''))$ determines the relative isotopy class of $\delta$ restricted to $N(C_{j})$. Hence, under the assumption of the statement, the proof of \cite[Th{\'e}or{\`e}me 4.8]{FLP} yields model multicurves $\delta$ and $\delta'$, isotopic to $D$ and $D'$ respectively, such that $\delta|_{S \setminus N(C_{k})} = \delta'|_{S \setminus N(C_{k})}$ up to isotopy relative to $\partial N(C_{k})$. Hence, $\delta$ and $\delta'$ are related by a fractional Dehn twist along $C_{k}$, and so are $D$ and $D'$.
\end{proof}

\section{Length identities} \label{sec:lengthidentities}

In this section, we show how length identities of curves keep track of their topological configuration. This is a converse procedure of previously known result, introduced in Subsection~\ref{subsec:topconftolength}.

We begin by referring to a theorem of McShane and Parlier.

\begin{thm}{\cite[Theorem 1.1]{mcshane2008simple}}\label{thm:mcshane}
For each pair of distinct essential or boundary curves $\alpha$, $\beta$ on a surface $S$ of finite type, there exists a connected analytic submanifold $E(\alpha, \beta)$ of $\T(S)$ such that $l_{X}(\alpha)\neq l_{X}(\beta)$ for $X \in \T(S) \setminus E(\alpha, \beta)$. Consequently, points in $\T(S) \setminus \bigcup_{\alpha \neq \beta} E(\alpha, \beta)$ have simple simple length spectra (that is, simple length spectra such that multiplicity of each length is $1$).
\end{thm}

This theorem asserts that essential or boundary curves on $S$ are faithfully labelled by their lengths at almost every point of $\T(S)$, although not everywhere. Note that this can be generalized to surfaces of infinite type as follows. Let $\alpha$, $\beta$ be distinct curves on a surface $S$ of infinite type. Since curves are compact, they are contained in some finite-type subsurface $S_{1}$ of $S$ bounded by some curves $C_{i_{1}}$, $\ldots$, $C_{i_{n}}$. Then $l_{X}(\alpha) - l_{X}(\beta)$ becomes a non-constant analytic function on $\T(S_{1})$. By Theorem~\ref{thm:mcshane}, there exists a submanifold $E$ of $\T(S_{1})$ such that $l_{X}(\alpha) - l_{X}(\beta)$ does not vanish outside $E$. Since $E$ is nowhere dense, $\tilde{E} := \pi_{S_{1}}^{-1}(E) \subseteq \T(S)$ is also nowhere dense. \label{page.E}

The key observation for Theorem~\ref{thm:mcshane} is that $E(\alpha, \beta)$ is the zero locus of a non-constant analytic function $l_{X}(\alpha) - l_{X}(\beta)$ of $X$ on $\T(S)$. The purpose of this section is proving analogous results for other length identities. 

\subsection{From topological configurations to length identities} \label{subsec:topconftolength}
Here we review classical length identities of curves on hyperbolic surfaces. For details, 
see ~\cite{mcKean1972compact} or~\cite{Luo_geodesic}. Given curves $\eta_{1}$, $\eta_{2}$ on a surface $S$, we define the following functions on $\T(S)$: \[
f(X;\eta_{1}, \eta_{2}):= 2 \cosh \frac{l_{X}(\eta_{1})}{2}  \cosh \frac{l_{X}(\eta_{2})}{2},
\]\[
g(X;\eta_{1}, \eta_{2}) := \cosh \frac{l_{X}(\eta_{1})}{2} +  \cosh \frac{l_{X}(\eta_{2})}{2}.
\]
Both $f(X;\eta_{1}, \eta_{2})$ and $g(X; \eta_{1}, \eta_{2})$ are functions on $\T(S)$ with infimum 2. Moreover, if $\lim_{r} f(X_{r}; \eta_{1}, \eta_{2}) = 2$ or $\lim_{r} g(X_{r}; \eta_{1}, \eta_{2}) = 2$ for some path $\{X_{r}\} \subseteq \T(S)$, then both $l_{X_{r}}(\eta_{1})$, $l_{X_{r}}(\eta_{2})$ converge to 0. We also note that $l_{X}(\eta)$ becomes a constant function over $\T(S)$ if $\eta$ is bounding a puncture.

Now let $\alpha$, $\gamma$ be two curves on $S$ with $i(\alpha, \gamma) = 1$. Then $\alpha \cup \gamma$ becomes a spine of a one-holed/punctured torus with boundary $\delta := \alpha \gamma \alpha^{-1} \gamma^{-1}$. See Figure~\ref{fig:spineof1holed}.

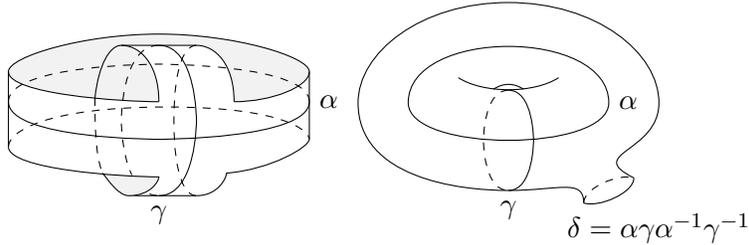
\begin{figure}[th]
	\begin{tikzpicture}[scale=1, every node/.style={scale=1}]
	\filldraw[fill=black!05, draw=white] (0, 0) .. controls (-0.5, 0) and (-2, 0.1) .. (-2, 0.4) .. controls (-2, 0.6) and (-1, 0.9) .. (0, 0.9) .. controls (1, 0.9) and (2, 0.6) .. (2, 0.4) .. controls (2, 0.2) and (1.5, 0) .. (1, 0) .. controls (1, 0.5) and (0.75, 0.75) .. (0.5, 0.75) -- (-0.5, 0.75) .. controls (-0.25, 0.75) and (0, 0.5) .. (0, 0);
	
	\draw (-0.5, 0.75) -- (0.5, 0.75);
	\draw (0.5, -1.25) -- (-0.4, -1.25);
	\draw (-2, -0.6) -- (-2, 0.4);
	\draw (2, -0.5) -- (2, 0.4);

	\draw (0,0.75) .. controls (0.5, 0.75) and (0.5, -0.25) .. (0.5, -0.25) .. controls (0.5, -0.25) and (0.5, -1.25) .. (0, -1.25);
	\draw[dashed] (0,0.75) .. controls (-0.5, 0.75) and (-0.5, -0.25) .. (-0.5, -0.25) .. controls (-0.5, -0.25) and (-0.5, -1.25) .. (0, -1.25);
	
	\draw (0, 0) .. controls (-0.5, 0) and (-2, 0.1) .. (-2, 0.4) .. controls (-2, 0.6) and (-1, 0.9) .. (0, 0.9) .. controls (1, 0.9) and (2, 0.6) .. (2, 0.4) .. controls (2, 0.2) and (1.5, 0) .. (1, 0);

	\draw (0, 0) .. controls (0, 0.5) and (-0.25, 0.75) .. (-0.5, 0.75) .. controls (-0.65, 0.75) and (-0.8, 0.5) .. (-0.85, 0.05);
	
	\draw[dashed] (-0.85, 0.05) .. controls (-0.85,  -0.25) and (-0.85, -0.7).. (-0.75, -0.95);
	\draw (1, 0) .. controls (1, 0.5) and (0.75, 0.75) .. (0.5, 0.75);
	\draw[dashed] (0.5, 0.75) .. controls (0.35, 0.75) and (0.2, 0.5) .. (0.2, -0.25) .. controls (0.2, -1) and (0.35, -1.25) .. (0.5, -1.25);
	
	\draw[dashed] (2, 0) .. controls (2, 0.5) and (0, 0.5) .. (0, 0.5) .. controls (0, 0.5) and (-2, 0.5) .. (-2, 0);
	\draw (2, 0) .. controls (2, -0.5) and (0, -0.5) .. (0, -0.5) .. controls (0, -0.5) and (-2, -0.5) .. (-2, 0);

	\filldraw[fill=black!05, draw=white] (-0.75, -0.95) .. controls (-0.7, -1) and (-0.55, -1.25) ..(-0.4, -1.25) .. controls (-0.1, -1.25) and (0, -1.05) .. (0, -1);
	\draw(-0.75, -0.95) .. controls (-0.7, -1) and (-0.55, -1.25) ..(-0.4, -1.25) .. controls (-0.1, -1.25) and (0, -1.05) .. (0, -1) .. controls (-0.5, -1) and (-2, -0.9) .. (-2, -0.6);
	\draw[dashed] (-2, -0.6) .. controls (-2, -0.4) and (-1.5, -0.07) .. (0.1, -0.07) .. controls (1.7, -0.07) and (2, -0.4) .. (2, -0.5);
	\draw (0.5, -1.25) .. controls (0.8, -1.25)  and (0.9, -1) .. (0.9, -0.95) .. controls (1.5, -0.9) and (2, -0.6) .. (2, -0.5);
	
	\draw (0, -1.25) node[below] {$\gamma$};
	\draw (2, 0) node[right] {$\alpha$};
	
	\draw (0, -2);
	
	\end{tikzpicture}
	\begin{tikzpicture}[scale=2/3, every node/.style={scale=1}]

	\draw (0,0.75) .. controls (0.5, 0.75) and (0.5, -0.25) .. (0.5, -0.25) .. controls (0.5, -0.25) and (0.5, -1.25) .. (0, -1.25);
	\draw[dashed] (0,0.75) .. controls (-0.5, 0.75) and (-0.5, -0.25) .. (-0.5, -0.25) .. controls (-0.5, -0.25) and (-0.5, -1.25) .. (0, -1.25);
	
	\draw (-1, 1) .. controls (-0.5, 0.68) and (0.5, 0.68) .. (1, 1);
	\draw (-0.3, 0.77) .. controls (-0.1, 0.9) and (0.1, 0.9) .. (0.3, 0.77);

	\draw  (1.5, -1.5) .. controls (1.5, -1) and (1, -1.25) .. (0, -1.25) .. controls (-1, -1.25) and (-3, -1) .. (-3, 0.5) .. controls (-3, 1.5) and (-1.75, 2.5) ..  (0, 2.5) .. controls (1.75, 2.5) and (3, 1.5) .. (3, 0.5) .. controls (3, -0.5) and (1.5, -0.5) .. (2.5, -1);
	
	\draw (2.5, -1) .. controls (2.5, -1.3) and (1.8, -1.6) .. (1.5, -1.5);
	\draw[dashed] (1.5, -1.5) .. controls (1.5, -1.2) and (2.2, -0.9) .. (2.5, -1);
	
	\draw (3, -1.5) node[below] {$\delta = \alpha \gamma \alpha^{-1} \gamma^{-1}$};
	
	\draw (2, 0.5) .. controls (2, -0.5) and (-2, -0.5) .. (-2, 0.5) .. controls (-2, 2) and (2, 2) .. (2, 0.5);
	
	\draw (0, -1.25) node[below] {$\gamma$};
	\draw (2, 0.5) node[right] {$\alpha$};
	\end{tikzpicture}
	\caption{One-holed torus with spine $\alpha \cup \gamma$} \label{fig:spineof1holed}
\end{figure}

\begin{fact}\label{fact:lengthID1}
Let $\alpha$ and $\gamma$ be as above. Then 
$f(X;\alpha, \gamma) = g(X;T_{\gamma}^{1} (\alpha), T_{\gamma}^{-1}(\alpha))$ identically holds on $T(S)$.
\end{fact}

We then set $\alpha_{i} := T_{\gamma}^i(\alpha)$ and $\gamma_{i} := T_{\alpha}^{i}(\gamma)$. Note that $$i(\gamma, \alpha_{i}) = i(\alpha_{i-1}, \alpha_{i}) = 1 \mbox{ and }i(\alpha, \gamma_{i}) = i(\gamma_{i-1}, \gamma_{i}) = 1.$$ Moreover, one of $\{T_{\alpha_{i}}^{\pm 1}(\alpha_{i-1})\}$ is $\gamma$; we denote the other one by $\beta_{i}$. Similarly, one of $\{T_{\gamma_{i}}^{\pm 1}(\gamma_{i-1})\}$ is $\alpha$ and we denote the other one by $\epsilon_{i}$. See Figure~\ref{fig:fractwiston1holed}.

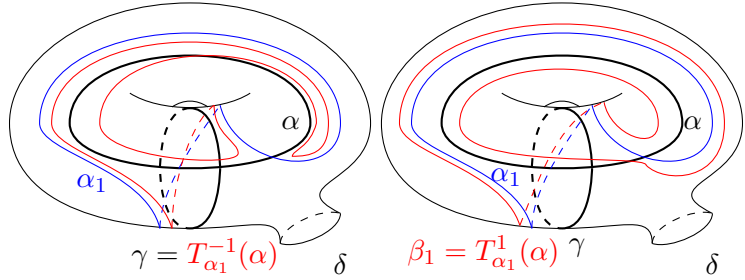
\begin{figure}[th]
		\begin{tikzpicture}[scale=0.8, every node/.style={scale=1}]

	\draw[thick] (0,0.75) .. controls (0.5, 0.75) and (0.5, -0.25) .. (0.5, -0.25);
	\draw[thick] (0, -1.25) .. controls (0.5, -1.25) and (0.5, -0.25) .. (0.5, -0.25);
	\draw[dashed, thick] (0,0.75) .. controls (-0.5, 0.75) and (-0.5, -0.25) .. (-0.5, -0.25) .. controls (-0.5, -0.25) and (-0.5, -1.25) .. (0, -1.25);
	
	\draw[thick] (2, 0.5) .. controls (2, -0.5) and (-2, -0.5) .. (-2, 0.5) .. controls (-2, 2) and (2, 2) .. (2, 0.5);
	
	\draw (-1, 1) .. controls (-0.5, 0.68) and (0.5, 0.68) .. (1, 1);
	\draw (-0.3, 0.77) .. controls (-0.1, 0.9) and (0.1, 0.9) .. (0.3, 0.77);

	\draw  (1.5, -1.5) .. controls (1.5, -1) and (1, -1.25) .. (0, -1.25) .. controls (-1, -1.25) and (-3, -1) .. (-3, 0.5) .. controls (-3, 1.5) and (-1.75, 2.5) ..  (0, 2.5) .. controls (1.75, 2.5) and (3, 1.5) .. (3, 0.5) .. controls (3, -0.5) and (1.5, -0.5) .. (2.5, -1);
	
	\draw (2.5, -1) .. controls (2.5, -1.3) and (1.8, -1.6) .. (1.5, -1.5);
	\draw[dashed] (1.5, -1.5) .. controls (1.5, -1.2) and (2.2, -0.9) .. (2.5, -1);
	
	\draw (2.5, -1.5) node[below] {$\delta$};
	
	\draw[blue]  (-0.5, -1.25) .. controls (-0.75, -0.25) and (-2.5, 0) .. (-2.5, 0.5) .. controls (-2.5, 2.5) and (2.5, 2.5) .. (2.5, 0.5) .. controls (2.5, -0.5) and (0.75, -0.25) .. (0.5, 0.8);
	\draw[blue, dashed] (-0.5, -1.25) .. controls (-0.5, -0.5) and (0.3, 0.5).. (0.5, 0.8);
	
	\draw (0.25, -1.25) node[below] {$\gamma = \color{red}{T_{\alpha_1}^{-1}(\alpha)}$};
	\draw (2, 0.5) node[left] {$\alpha$};
	\draw[blue] (-1.2, -0.5) node[left] {$\alpha_1$};
	
	\draw[red]  (-0.3, -1.25) .. controls (-0.6, -0.25) and (-2.3, 0) .. (-2.3, 0.5) .. controls (-2.3, 2.3) and (2.3, 2.3) .. (2.3, 0.5) .. controls (2.3, 0) and (2, -0.05) .. (1.7, -0.05) .. controls (1.7, 0.05) and (2, 0.1) .. (2.1, 0.5) .. controls (2.1, 2) and (-1.5, 2) .. (-1.5, 0.5) .. controls (-1.5, -0.1) and (0.8, -0.25) .. (0.8, 0) .. controls (0.8, 0.2) and (0.4, 0.3) .. (0.4, 0.8);
	
	\draw[red, dashed] (-0.3, -1.25) .. controls (-0.3, -0.5) and (0, 0.5) .. (0.4, 0.8);
	..	\end{tikzpicture}
	\begin{tikzpicture}[scale=0.8, every node/.style={scale=1}]

	\draw[thick] (0,0.75) .. controls (0.5, 0.75) and (0.5, -0.25) .. (0.5, -0.25);
	\draw[thick] (0, -1.25) .. controls (0.5, -1.25) and (0.5, -0.25) .. (0.5, -0.25);
	\draw[dashed, thick] (0,0.75) .. controls (-0.5, 0.75) and (-0.5, -0.25) .. (-0.5, -0.25) .. controls (-0.5, -0.25) and (-0.5, -1.25) .. (0, -1.25);
	
	\draw[thick] (2, 0.5) .. controls (2, -0.5) and (-2, -0.5) .. (-2, 0.5) .. controls (-2, 2) and (2, 2) .. (2, 0.5);
	
	\draw (-1, 1) .. controls (-0.5, 0.68) and (0.5, 0.68) .. (1, 1);
	\draw (-0.3, 0.77) .. controls (-0.1, 0.9) and (0.1, 0.9) .. (0.3, 0.77);

	\draw  (1.5, -1.5) .. controls (1.5, -1) and (1, -1.25) .. (0, -1.25) .. controls (-1, -1.25) and (-3, -1) .. (-3, 0.5) .. controls (-3, 1.5) and (-1.75, 2.5) ..  (0, 2.5) .. controls (1.75, 2.5) and (3, 1.5) .. (3, 0.5) .. controls (3, -0.5) and (1.5, -0.5) .. (2.5, -1);
	
	\draw (2.5, -1) .. controls (2.5, -1.3) and (1.8, -1.6) .. (1.5, -1.5);
	\draw[dashed] (1.5, -1.5) .. controls (1.5, -1.2) and (2.2, -0.9) .. (2.5, -1);
	
	\draw (2.5, -1.5) node[below] {$\delta$};
	
	\draw[blue]  (-0.5, -1.25) .. controls (-0.75, -0.25) and (-2.5, 0) .. (-2.5, 0.5) .. controls (-2.5, 2.5) and (2.5, 2.5) .. (2.5, 0.5) .. controls (2.5, -0.5) and (0.75, -0.25) .. (0.5, 0.8);
	\draw[blue, dashed] (-0.5, -1.25) .. controls (-0.5, -0.5) and (0.3, 0.5).. (0.5, 0.8);
	
	\draw (0.25, -1.25) node[below] {$\gamma$};
	\draw (2.5, 0.5) node[left] {$\alpha$};
	\draw[blue] (-0.5, -0.4) node[left] {$\alpha_1$};
	\draw[red] (-1.3, -1.2) node[below] {$\beta_1 = T_{\alpha_1}^{1}(\alpha)$};
	
	\draw[red]  (-0.7, -1.2) .. controls (-1, -0.25) and (-2.7, 0) .. (-2.7, 0.5) .. controls (-2.7, 2.7) and (2.7, 2.7) .. (2.7, 0.5) .. controls (2.7, -0.5) and (1.7, -0.5) .. (1.4, -0.2) .. controls (1.4, 0) and (-1.7, -0.3) .. (-1.7, 0.5) .. controls (-1.7, 1.7) and (1.6, 1.7) .. (1.6, 0.5) .. controls (1.6, 0) and  (0.8, 0.3) .. (0.7, 0.85);
	
	\draw[red, dashed] (-0.7, -1.2) .. controls (-0.3, -0.2) and (0, 0.6) .. (0.7, 0.85);
	\end{tikzpicture}
	\caption{Fractional Dehn twists of $\alpha$ along $\alpha_1$. Note that we sometimes have $\alpha = \alpha_0$ and $\gamma = \gamma_0$.} \label{fig:fractwiston1holed}
\end{figure}

\begin{lem}\label{lem:lengthIDSet1}
For each $i \in \mathbb{Z}$ and any of $(\eta_{1}, \eta_{2}, \eta_{3}, \eta_{4}) = (\alpha_{i}, \gamma_{0}, \alpha_{i-1}, \alpha_{i+1})$, $(\alpha_{i-1}, \alpha_{i}, \gamma_{0}, \beta_{i})$, $(\gamma_{i}, \alpha_{0}, \gamma_{i-1}, \gamma_{i+1})$, $(\gamma_{i-1}, \gamma_{i}, \alpha_{0}, \epsilon_{i})$, the identity \begin{equation}\label{eqn:lengthIDSet1}
f(X;\eta_{1}, \eta_{2}) = g(X;\eta_{3}, \eta_{4})
\end{equation}
 holds on all of $\T(S)$.
\end{lem}

This time, we consider $\alpha, \gamma$ satisfying $i_{alg}(\alpha, \gamma) = 0$ and $i_{geom}(\alpha, \gamma) = 2$. Then $\alpha \cup \gamma$ becomes a spine of an immersed subsurface $\psi : S' \rightarrow S$ where $S'$ is a generalized shirt. This shirt is accompanied by peripheral curves $\delta_{1}$, $\ldots$, $\delta_{4}$. They are labelled in such a manner that $\gamma$ separates $\{\delta_{1}, \delta_{2}\}$ from $\{\delta_{3}, \delta_{4}\}$ and $\alpha$ separates $\{\delta_{1}, \delta_{3}\}$ from $\{\delta_{2}, \delta_{4}\}$. Note that $T_{\gamma}^{\pm 1}(\alpha)$ then separates $\{\delta_{2}, \delta_{3}\}$ from $\{\delta_{1}, \delta_{4}\}$. See Figure~\ref{fig:spineof4holed}.

\begin{figure}[th]
	\begin{tikzpicture}[scale=1, every node/.style={scale=1}]
	
	\draw (-0.5, 2) -- (0.5, 2);
	\draw (-0.5, -2) -- (0.5, -2);
	\draw (-2, -0.6) -- (-2, 0.4);
	\draw (2, -0.5) -- (2, 0.4);
	
	\draw[fill=black!05, draw=white] (-0.5, 2) .. controls (0, 2) and (0, 0) .. (0, 0) .. controls (-0.5, 0) and (-2, 0.1) .. (-2, 0.4) .. controls (-2, 0.6) and (-1.5, 0.8) .. (-1, 0.8) .. controls (-1, 0.8) and (-1, 2) .. (-0.5, 2);
	\draw[fill=black!05, draw=white] (1, 0) .. controls (1.5, 0) and (2, 0.2) .. (2, 0.4) .. controls (2, 0.6) and (1.7, 0.7) .. (1, 0.8);
	\draw[fill=black!05, draw=white]  (0, -1) .. controls (0, -1) and (0, -2) .. (-0.5, -2) .. controls (-1, -2) and (-1, -0.9) .. (-1, -1);
	
	\draw (0, 2) .. controls (0.5, 2) and (0.5, 0) .. (0.5, 0) .. controls (0.5, 0) and (0.5, -2) .. (0, -2);
	\draw[dashed] (0, 2) .. controls (-0.5, 2) and (-0.5, 0) .. (-0.5, 0) .. controls (-0.5, 0) and (-0.5, -2) .. (0, -2);
	
	\draw[dashed] (2, 0) .. controls (2, 0.5) and (0, 0.5) .. (0, 0.5) .. controls (0, 0.5) and (-2, 0.5) .. (-2, 0);
	\draw (2, 0) .. controls (2, -0.5) and (0, -0.5) .. (0, -0.5) .. controls (0, -0.5) and (-2, -0.5) .. (-2, 0);
	
	\draw (0.5, 2) .. controls (1, 2) and (1, 0) .. (1, 0) .. controls (1.5, 0) and (2, 0.2) .. (2, 0.4) .. controls (2, 0.6) and (1.7, 0.7) .. (1, 0.8);
	
	\draw[dashed] (0.5, 2) .. controls (0, 2) and (0, 0.9) .. (0, 0.85) .. controls (0.5, 0.85) .. (0.9, 0.8);
	
	\draw (-0.5, 2) .. controls (0, 2) and (0, 0) .. (0, 0) .. controls (-0.5, 0) and (-2, 0.1) .. (-2, 0.4) .. controls (-2, 0.6) and (-1.5, 0.8) .. (-1, 0.8) .. controls (-1, 0.8) and (-1, 2) .. (-0.5, 2);
	
	\draw (0, -1) .. controls (0, -1) and (0, -2) .. (-0.5, -2) .. controls (-1, -2) and (-1, -0.9) .. (-1, -1);
	
	\draw (0, -1) .. controls (-0.5, -1) and (-2, -0.9) .. (-2, -0.6);
	\draw[dashed] (-2, -0.6) .. controls (-2, -0.4) and (-1.5, -0.2) .. (-1.05, -0.2) .. controls (-1.025, -0.6) .. (-1, -1);
	
	\draw[dashed] (2, -0.5) .. controls (2, -0.4) and (1.7, -0.1) .. (0.1, -0.1) .. controls (0.1, -0.1) and (0.1, -2) .. (0.5, -2);
	
	\draw (0.5, -2) .. controls (0.8, -2) and (0.9, -0.95) .. (0.9, -0.95) .. controls (1.5, -0.9) and (2, -0.6) .. (2, -0.5);
	
	\draw (0, 2) node[above] {$\alpha$};
	\draw (2, 0) node[right] {$\gamma$};
	
	\end{tikzpicture}
	\begin{tikzpicture}[scale=2/3, every node/.style={scale=1}]
	\draw[thick] (0, 2) .. controls (0.5, 2) and (0.5, 0) .. (0.5, 0) .. controls (0.5, 0) and (0.5, -2) .. (0, -2);
	\draw[thick, dashed] (0, 2) .. controls (-0.5, 2) and (-0.5, 0) .. (-0.5, 0) .. controls (-0.5, 0) and (-0.5, -2) .. (0, -2);
	
	\draw[thick, dashed] (2, 0) .. controls (2, 0.5) and (0, 0.5) .. (0, 0.5);
	\draw[thick, dashed] (-2, 0) .. controls (-2, 0.5) and (0, 0.5) .. (0, 0.5);
	\draw[thick] (2, 0) .. controls (2, -0.5) and (0, -0.5) .. (0, -0.5);
	\draw[thick] (0, -0.5) .. controls (0, -0.5) and (-2, -0.5) .. (-2, 0);
	
	\draw (-2, 3) .. controls (-2, 2.5) and (-1, 2) .. (0, 2) .. controls (1, 2) and (2, 2.5) .. (2, 3);
	\draw (-2, -3) .. controls (-2, -2.5) and (-1, -2) .. (0, -2) .. controls (1, -2) and (2, -2.5) .. (2, -3);
	\draw (3, -2) .. controls (2.5, -2) and (2, -1) .. (2, 0) .. controls (2, 1) and (2.5, 2) .. (3, 2);
	\draw (-3, -2) .. controls (-2.5, -2) and (-2, -1) .. (-2, 0) .. controls (-2, 1) and (-2.5, 2) .. (-3, 2);
	
	\draw[fill = black!05] (-2, 3) .. controls (-2, 2.5) and (-2.5, 2) .. (-3, 2);
	\draw[fill = black!05] (-2, 3) .. controls (-2.5, 3) and (-3, 2.5) .. (-3, 2);
	
	\draw[dashed] (-2, -3) .. controls (-2, -2.5) and (-2.5, -2) .. (-3, -2);
	\draw (-2, -3) .. controls (-2.5, -3) and (-3, -2.5) .. (-3, -2);
	
	\draw[fill = black!05] (2, 3) .. controls (2, 2.5) and (2.5, 2) .. (3, 2);
	\draw[fill = black!05] (2, 3) .. controls (2.5, 3) and (3, 2.5) .. (3, 2);
	
	\draw[dashed] (2, -3) .. controls (2, -2.5) and (2.5, -2) .. (3, -2);
	\draw (2, -3) .. controls (2.5, -3) and (3, -2.5) .. (3, -2);
	
	\draw (0, 2) node[above] {$\alpha$};
	\draw (2, 0) node[right] {$\gamma$};
	
	\draw (-3.3, 3) node[below] {$\delta_1$};
	\draw (3.3, 3) node[below] {$\delta_2$};
	\draw (-3.3, -3) node[above] {$\delta_3$};
	\draw (3.3, -3) node[above] {$\delta_4$};
	
	\draw[blue] (-1, 2.15) .. controls (-1.2, 1.5) and (-1.5, 1.2) .. (-2.15, 1);
	\draw[blue] (-1+3.15, 2.15-3.15) .. controls (-1.2+3.15, 1.5-3.15) and (-1.5+3.15, 1.2-3.15) .. (-2.15+3.15, 1-3.15);
	
	\draw[blue, dashed] (-1, 2.15) .. controls (0.85, 1.7) and (1.7, 0.85) .. (2.15, -1);
	\draw[blue, dashed] (-2.15, 1) .. controls (-1.7, -0.85) and (-0.85, -1.7) .. (1, -2.15);
	\draw[blue] (2.15, -1) node[right] {$T_{\gamma}^1 (\alpha)$};

	\draw[red] (1, 2.15) .. controls (1.2, 1.5) and (1.5, 1.2) .. (2.15, 1);
	\draw[red] (1-3.15, 2.15-3.15) .. controls (1.2-3.15, 1.5-3.15) and (1.5-3.15, 1.2-3.15) .. (2.15-3.15, 1-3.15);
	
	\draw[red, dashed] (1, 2.15) .. controls (-0.85, 1.7) and (-1.7, 0.85) .. (-2.15, -1);
	\draw[red, dashed] (2.15, 1) .. controls (1.7, -0.85) and (0.85, -1.7) .. (-1, -2.15);
	\draw[red] (-2.15, -1) node[left] {$T_{\gamma}^{-1} (\alpha)$};
	
	\draw (-2, 3) .. controls (-2, 2.5) and (-1, 2) .. (0, 2) .. controls (1, 2) and (2, 2.5) .. (2, 3);
	\draw (-2, -3) .. controls (-2, -2.5) and (-1, -2) .. (0, -2) .. controls (1, -2) and (2, -2.5) .. (2, -3);
	\draw (3, -2) .. controls (2.5, -2) and (2, -1) .. (2, 0) .. controls (2, 1) and (2.5, 2) .. (3, 2);
	\draw (-3, -2) .. controls (-2.5, -2) and (-2, -1) .. (-2, 0) .. controls (-2, 1) and (-2.5, 2) .. (-3, 2);
	\end{tikzpicture}
	\caption{A shirt with spine $\alpha \cup \gamma$} \label{fig:spineof4holed}
\end{figure}
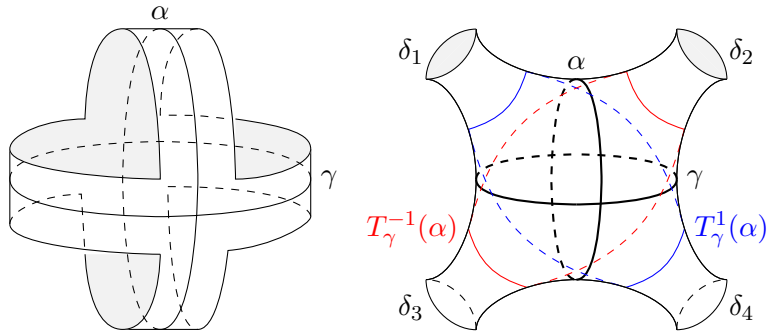

\begin{fact}\label{fact:lengthID2}
Let $\alpha$, $\gamma$, $\{\delta_{i}\}$ be as above. Then the identity $$f(X; \alpha, \gamma) = g(X; T_{\gamma}^{1}(\alpha), T_{\gamma}^{-1} (\alpha)) + f(X; \delta_{2}, \delta_{3}) + f(X; \delta_{1}, \delta_{4})$$ holds on all of $\T(S)$.
\end{fact}
 
We now set $\alpha_{i} := T_{\gamma}^i(\alpha)$ and $\gamma_{i} := T_{\alpha}^{i}(\gamma)$. Then \[
i_{alg}(\gamma, \alpha_{i})=i_{alg}(\alpha, \gamma_{i}) = i_{alg}(\alpha_{i-1}, \alpha_{i})=i_{alg}(\gamma_{i-1}, \gamma_{i})= 0\] and \[
i_{geom}(\gamma, \alpha_{i})=i_{geom}(\alpha, \gamma_{i}) = i_{geom}(\alpha_{i-1}, \alpha_{i})=i_{geom}(\gamma_{i-1}, \gamma_{i})=  2
\]
hold. Consequently, one of $\{T_{\alpha_{i}}^{\pm 1} (\alpha_{i-1})\}$ ($\{T_{\gamma_{i}}^{\pm 1}(\gamma_{i-1})\}$, resp.) is $\gamma$ ($\alpha$, resp.) and the other one is denoted by $\beta_{i}$ ($\epsilon_{i}$, resp.).

\begin{lem}\label{lem:lengthIDSet2}
For each $i \in \mathbb{Z}$ and any of \[
(\eta_{1}, \ldots, \eta_{8}) = \left\{\begin{array}{cc} (\alpha_{2i}, \gamma_{0}, \alpha_{2i-1}, \alpha_{2i+1}, \delta_{2}, \delta_{3}, \delta_{1}, \delta_{4}) \\ (\alpha_{2i+1}, \gamma_{0}, \alpha_{2i}, \alpha_{2i+2}, \delta_{1}, \delta_{3}, \delta_{2}, \delta_{4}) \\
 (\gamma_{2i},\alpha_{0}, \gamma_{2i-1}, \gamma_{2i+1}, \delta_{2}, \delta_{3}, \delta_{1}, \delta_{4}) \\ 
 (\gamma_{2i+1}, \alpha_{0}, \gamma_{2i}, \gamma_{2i+2}, \delta_{1}, \delta_{2}, \delta_{3}, \delta_{4}) \\
 (\alpha_{i-1}, \alpha_{i}, \gamma_{0}, \beta_{i}, \delta_{1}, \delta_{2}, \delta_{3}, \delta_{4}) \\
 (\gamma_{i-1}, \gamma_{i}, \alpha_{0}, \epsilon_{i}, \delta_{1}, \delta_{3}, \delta_{2}, \delta_{4}) \end{array}\right.,
\]
the identity \begin{equation}\label{eqn:lengthIDSet2}
f(X; \eta_{1}, \eta_{2}) = g(X; \eta_{3}, \eta_{4}) + f(X; \eta_{5}, \eta_{6}) + f(X; \eta_{7} ,\eta_{8})
\end{equation} holds on all of $\T(S)$.
\end{lem}

\subsection{First length identity: one-holed/punctured torus}

We now discuss the converse of Lemma~
\ref{lem:lengthIDSet1}. For the converse of Lemma \ref{lem:lengthIDSet2}, see Subsection \ref{subsec:convofIDSet2}. The following lemma partially relates the length identities to the configuration of the curves involved. Recall that for $X \in \T(S)$ and curves $\eta_1, \eta_2 \subset S$,
\[
f(X;\eta_{1}, \eta_{2}):= 2 \cosh \frac{l_{X}(\eta_{1})}{2}  \cosh \frac{l_{X}(\eta_{2})}{2},
\]\[
g(X;\eta_{1}, \eta_{2}) := \cosh \frac{l_{X}(\eta_{1})}{2} +  \cosh \frac{l_{X}(\eta_{2})}{2}.
\]

\begin{lem}\label{lem:fiveIDSuff} Let $\alpha_{-2}$, $\alpha_{-1}$, $\alpha$, $\alpha_{1}$, $\alpha_{2}$, $\gamma$, $\beta_{0}$, $\beta_{1}$ be essential or boundary curves on $S$, where $\alpha$, $\alpha_{-1}$, $\alpha_{1}$ and $\gamma$ are distinct. If the inequalities \begin{subequations}
\begin{align}
f(X;\alpha, \gamma) &\ge g(X;\alpha_{-1}, \alpha_{1}), \label{eqn:lengthTemp1} \\
f(X;\alpha_{-1}, \gamma) &\ge g(X;\alpha_{-2}, \alpha),\label{eqn:lengthTemp2} \\
f(X;\alpha_{1}, \gamma) &\ge g(X;\alpha, \alpha_{2}), \label{eqn:lengthTemp3} \\
f(X;\alpha, \alpha_{1}) &\ge g(X;\gamma, \beta_{1}), \label{eqn:lengthTemp4} \\
f(X;\alpha_{-1}, \alpha) &\ge g(X;\gamma, \beta_{0})\label{eqn:lengthTemp5}
\end{align}
are satisfied by all $X \in \T(S)$, then $$i(\alpha, \gamma) >0 \mbox{ and }\{\alpha_{1}, \alpha_{-1}\} = \{T_{\gamma}^{1}(\alpha), T_{\gamma}^{-1}(\alpha)\}.$$
\label{eqn:lengthTemp}
\end{subequations}
\end{lem}

\begin{proof}
Suppose first that $i(\alpha, \gamma) = 0$. We fix a pants decomposition containing $\alpha$ and $\gamma$, and pinch them simultaneously. Since $\alpha_{\pm 1}$ are neither $\alpha$ nor $\gamma$, their lengths tend to either infinity (if they intersect with $\alpha$ or $\gamma$) or a finite value (if they do not intersect with $\alpha$ and $\gamma$). Thus, $f(X; \alpha, \gamma)$ tends to 2 while $g(X; \alpha_{-1}, \alpha_{1})$ converges to a term greater than 2. This contradicts Inequality \ref{eqn:lengthTemp1}, and we conclude $i(\alpha, \gamma) > 0$. In particular, both $\alpha$, $\gamma$ are essential.

Now we fix a pants decomposition containing $\alpha$ and pinch $\alpha$. Then $f(X; \alpha, \gamma)$ grows in the order of $l(\alpha)^{-i(\alpha, \gamma)}$, while $g(X; \alpha_{-1}, \alpha_{1})$ grows in the order of $l(\alpha)^{- \max(i(\alpha, \alpha_{-1}), i(\alpha, \alpha_{1}))}$. From this and Inequality \ref{eqn:lengthTemp1}, we deduce that $i(\alpha, \gamma) \ge i(\alpha, \alpha_{\pm 1})$.

Meanwhile, $f(X; \alpha_{-1}, \alpha)$ grows in the order of $l(\alpha)^{-i(\alpha, \alpha_{-1})}$ while $g(X; \gamma, \beta_{0})$ grows in the order of $l(\alpha)^{-\max(i(\alpha, \gamma), i(\alpha, \beta_{0}))}$. Then Inequality \ref{eqn:lengthTemp5} implies that $i(\alpha, \alpha_{-1}) \ge i(\alpha, \gamma)$. Similarly, investigating each side of Inequality \ref{eqn:lengthTemp4}  yields that $i(\alpha, \alpha_{1}) \ge i(\alpha, \gamma)$. Thus, we obtain \begin{equation}\label{eqn:lengthEqn1Inter1}
i(\alpha, \alpha_{\pm 1}) = i(\alpha, \gamma).
\end{equation}

Next, we fix a pants decomposition containing $\gamma$ and pinch $\gamma$. We again investigate Inequality \ref{eqn:lengthTemp1} to deduce that $i(\gamma, \alpha) \ge i(\gamma, \alpha_{\pm 1})$. The reverse inequalities are now obtained from Inequality \ref{eqn:lengthTemp2} and \ref{eqn:lengthTemp3} and we conclude \begin{equation} \label{eqn:lengthEqn1Inter2}
i(\alpha, \gamma) = i(\alpha_{\pm 1}, \gamma).
\end{equation} 
Finally, we consider  an arbitrary interior curve $\eta$ disjoint from  $\gamma$, fix a pants decomposition containing $\gamma$ and $\eta$, and pinch $\eta$.  While pinching, $\cosh l_{X}(\gamma)/2$ remains bounded and the LHS of Inequality \ref{eqn:lengthTemp1} grows in the order of $l(\eta)^{-i (\alpha, \eta)}$, while the RHS grows in the order of $l(\eta)^{- \max (i(\alpha_{-1}, \eta), i(\alpha_{1}, \eta))}$. Hence, from Inequality \ref{eqn:lengthTemp1} we obtain $i(\alpha, \eta) \ge i(\alpha_{\pm 1}, \eta)$. A similar argument using \ref{eqn:lengthTemp2} and \ref{eqn:lengthTemp3} leads to the reverse inequality and we deduce\begin{equation} \label{eqn:lengthEqn1Inter3}
i(\alpha, \eta) = i(\alpha_{\pm 1}, \eta).
\end{equation}

Combined with Equations \ref{eqn:lengthEqn1Inter2} and \ref{eqn:lengthEqn1Inter3}, Fact \ref{fact:onlyTwist} implies that $\alpha$, $\alpha_{-1}$, $\alpha_{1}$ differ only by a fractional Dehn twists along $\gamma$. 
Now consider a finite-type subsurface $S' \subset S$ containing $\alpha, \alpha_1, \alpha_{-1}$, and $\gamma$. Since we have seen that $\gamma$ is essential, we may also assume that $\gamma$ is essential in $S'$. Let $\{B_1, \ldots, B_m, C_1, \ldots, C_n\}$ be a pants decomposition of $S'$ where $B_j$'s are boundary curves and $C_1 = \gamma$. We set the curves $\{C_1', \ldots, C_n', C_1'', \ldots, C_n''\}$ on $S'$ given by Fact \ref{fact:DehnThurston}. In particular, $C_j'$ and $C_j''$ are disjoint from $C_1 = \gamma$ for all $j \neq 1$. Hence, by Equation \ref{eqn:lengthEqn1Inter3}, we have that $i(\alpha, C_j) = i(\alpha_{\pm 1}, C_j)$, $i(\alpha, C_j') = i(\alpha_{\pm 1}, C_j')$, and $i(\alpha, C_j'') = i(\alpha_{\pm 1}, C_j'')$ for all $j \neq 1$. In addition, by Equation \ref{eqn:lengthEqn1Inter2} and $C_1 = \gamma$, we also have $i(\alpha, C_1) = i(\alpha_{\pm 1}, C_1)$. Therefore, it follows from Fact \ref{fact:onlyTwist} that $\alpha$ and $\alpha_{\pm 1}$ are related by a fractional Dehn twist along $C_1 = \gamma$.
Then Lemma \ref{lem:fracDehn} reads Equation \ref{eqn:lengthEqn1Inter1} as $\{\alpha_{1}, \alpha_{-1}\} =  \{T_{\gamma}^{1}(\alpha), T_{\gamma}^{-1}(\alpha)\}$. (Here we used the condition that $\alpha_{-1}$ and $\alpha_{1}$ are distinct) 
\end{proof}

	As one can observe, not all points in the entire Teichm\"uller space are involved in the proof. Hence, one can modify the statement so that the inequalities are checked along paths in the Teichm\"uller space along which certain curves are pinched.

Before stating the next lemma, we first introduce some notation. Let $\alpha$, $\gamma$ be two curves on $S$ with $k=i(\alpha, \gamma) > 1$. Then $\alpha$ is cut by $\gamma$ into $k$ segments $\{a_{1}, \ldots, a_{k}\}$ and $\gamma$ is cut by $\alpha$ into $k$ segments $\{c_{1}, \ldots, c_{k}\}$. Each segment $a_{i}$ then splits $\gamma$ into two segments, giving a (bi)partition of $\{c_{1}, \ldots, c_{k}\}$ into two disjoint collection. The segment with fewer $c_{j}$'s is denoted by $\gamma(a_{i})$ and its number of $c_{j}$'s is denoted by $N_{\gamma}(a_{i})$. If two numbers are equal, then take either of the two segments. See Figure \ref{fig:alphaDissectGamma}.

The annular neighborhood of $\gamma$ is separated by $\gamma$ into two sides, which we label by left and right. Then $\{a_{1}, \ldots, a_{k}\}$ is partitioned into three collections, $A_{1}$, $A_{2}$ and $A_{3}$. $A_{1}$ ($A_{2}$, resp.) consists of those segments departing from and arriving at $\gamma$ on the left side (right side, resp.). $A_{3}$ consists of those segments connecting two sides of $\gamma$.

\begin{figure}[ht]
	
	\centering
	
	\begin{subfigure}[c]{0.3\textwidth}
		\centering
		
		\begin{tikzpicture}
		\begin{scope}[shift={(-1, 0.7)}]
		\draw (-1.4, -0.3) .. controls (-0.9, -0.1) and (-0.55, 0) .. (0, 0) .. controls (0.55, 0) and (0.9, -0.1) .. (1.4, -0.3);
		\draw (-1.2, -1.2) .. controls (-0.9, -0.95) and (-0.5, -0.8) .. (0, -0.8) .. controls (0.5, -0.8) and (0.9, -0.95) .. (1.2, -1.2);
		\draw[thick] (0, 0) arc (90:130:0.15 and 0.4);
		\draw[thick, ->] (0, -0.8) arc(270:130:0.15 and 0.4);
		\draw[thick, dashed] (0, 0) arc (90:-90:0.15 and 0.4);
		\draw (0, -1.05) node {$\gamma$};
		\draw (-0.7, -1.25) node {$A_{1}$};
		\draw (0.7, -1.25) node {$A_{2}$};

		\draw[thick] (-1.3, -1.02) .. controls (-0.9, -0.75) and (-0.45, -0.65) .. (0, -0.65) .. controls (0.45, -0.65) and (0.9, -0.75) .. (1.3, -1.02);
		
		\draw[thick] (-1.33, -0.65) .. controls (-0.9, -0.45) and (-0.5, -0.32) .. (0, -0.32) .. controls (0.5, -0.32) and (0.9, -0.45) .. (1.33, -0.64);
		\end{scope}
		\end{tikzpicture}
		\caption{Two groups of segments of $\alpha \setminus \gamma$.}
	\end{subfigure}
	\begin{subfigure}[c]{0.69\textwidth}
		\centering
		\begin{tikzpicture}
		
		\begin{scope}[shift={(7, 0)}]

		\draw[thick, -<] (1.8, 0) arc(0:220:1.8);
		\draw[thick] (1.8, 0) arc(0:-145:1.8);

		\foreach \i in {1, ..., 10}{
			\fill[rotate=\i*36-18] (1.8, 0) circle (0.05);
		}
		
		\draw[thick, rotate=36] (0, -1.8) -- (0, 1.8);
		\draw[thick, rotate=-36] (0, -1.8) -- (0, 1.8);
		\draw[thick, rotate=72] (0, -1.8) arc (180:108:1.8*1.376381920471174);
		\draw[thick] (0, -1.8) arc (180:108:1.8*1.376381920471174);
		\draw[thick, rotate=-108] (0, -1.8) arc (180:36:1.8*0.324919696232906);

		\draw[rotate=-36] (0, -2) arc (-90:90:2);
		\draw[rotate=-126] (1.9, 0) -- (2.1, 0);
		\draw[rotate=54] (1.9, 0) -- (2.1, 0);
		
		\draw (-0.55, -0.3) node {$b_{1}$};
		\draw (1.6, -1.9) node {$N_{\gamma}(b_{1})$};
		
		\end{scope}
		
		\begin{scope}[shift={(2.9, 0)}]
		\draw[thick, ->] (1.8, 0) arc(0:323:1.8);
		\draw[thick] (1.8, 0) arc(0:-37:1.8);

		\foreach \i in {1, ..., 10}{
			\fill[rotate=\i*36-18] (1.8, 0) circle (0.05);
		}
		
		\draw[thick,rotate=108] (0, -1.8) -- (0, 1.8);
		\draw[thick, rotate=36] (0, -1.8) -- (0, 1.8);
		\draw[thick, rotate=144] (0, -1.8) arc (180:108:1.8*1.376381920471174);
		\draw[thick, rotate=72] (0, -1.8) arc (180:108:1.8*1.376381920471174);
		\draw[thick,rotate=-36] (0, -1.8) arc (180:36:1.8*0.324919696232906);
		
		\draw[rotate=144] (0, -2) arc (-90:18:2);
		\draw[rotate=54] (1.9, 0) -- (2.1, 0);
		\draw[rotate=162] (1.9, 0) -- (2.1, 0);
		
		\draw (0, -2) arc (-90:-126:2);
		\draw[rotate=-90] (1.9, 0) -- (2.1, 0);
		\draw[rotate=-126] (1.9, 0) -- (2.1, 0);
		
		\draw (-0.9, 0.22) node {$a_{1}$};
		\draw (-0.52, -1.01) node {$a_{2}$};
		\draw (-1.3, 2.1) node {$N_{\gamma}(a_{1})$};
		\draw (-0.77, -2.28) node {$N_{\gamma}(a_{2})$};
		\draw (1.6, -1.22) node {$\gamma$};
		
		\end{scope}
		
		\end{tikzpicture}
		
		\caption{Segments in $A_{1}$ and $A_{2}$. Boundaries are $\gamma$. Here $N_{\gamma}(a_{1}) = 3$, $N_{\gamma}(a_{2}) = 1$ and $N_{\gamma}(b_{1}) = 5$.}
	\end{subfigure}
	
	\caption{Grouping segments}
	\label{fig:alphaDissectGamma}
\end{figure}
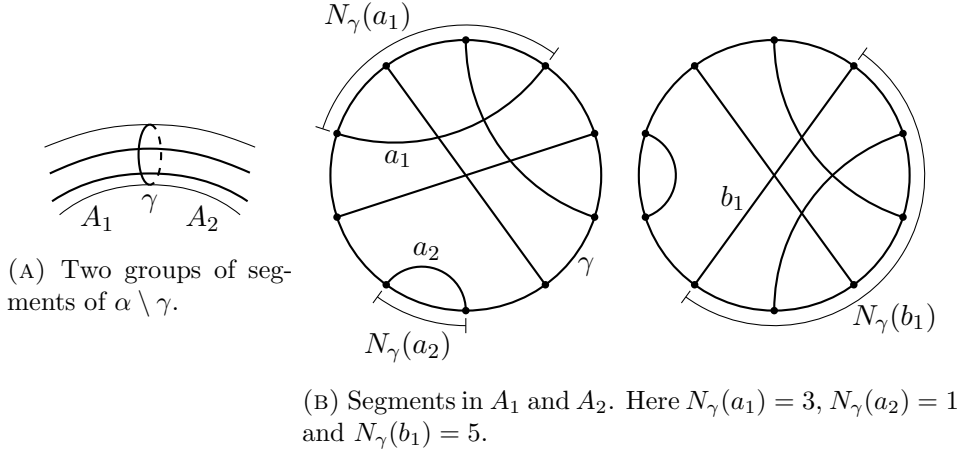

Recall that the \emph{index} of $(\alpha, \gamma)$ at an intersection point $p$ is $+1$ if the oriented basis made by the tangent vectors along $\alpha$ and $\gamma$ at $p$ agrees with the orientation of the surface $S$, and $-1$ otherwise.
Recall also that a fractional Dehn twist of a curve along another curve is a priori a multicurve, not necessarily a curve.

\begin{lem}\label{lem:joinEachSide}
Let $\alpha$, $\gamma$ be curves on $S$ with $k=i(\alpha, \gamma) > 1$. If $T_{\alpha}^{i}(\gamma)$ and $T_{\gamma}^{i}(\alpha)$ are single curves for every $i=0, 1, \ldots, k-1$, then \begin{enumerate}
\item $A_{3} = \emptyset$ and $k$ is even,
\item the indices of $(\alpha, \gamma)$ keeps alternating between $+1$ and $-1$ along each of $\alpha$ and $\gamma$, and
\item $N_{\gamma}(a_{i})$ is odd for each segment $a_{i} \in A_{1} \cup A_{2}$.
\end{enumerate}
If $k> 2$ moreover, then \begin{enumerate}
\setcounter{enumi}{3}
\item $N_{\gamma}(a_{i}) \neq N_{\gamma}(a_{i'})$ for all $a_{i} \in A_{1}$ and $a_{i'} \in A_{2}$.
\end{enumerate}
\end{lem}

\begin{proof}
We first suppose that there exists a segment of $\alpha \setminus \gamma$ joining the two sides of $\gamma$. Then some fractional Dehn twist $T_{\gamma}^i$ will tie this segment up into a closed curve, and the  other segments of $\alpha \setminus \gamma$ will combine to form at least one more curve. This contradicts the assumption that $\alpha_{i} =T_{\gamma}^i(\alpha)$ is a single curve. Thus, $A_{3} = \emptyset$ and $\alpha \setminus \gamma$ is partitioned into $A_{1}$ and $A_{2}$. Since $|A_{1}| = |A_{2}|$, their sum $k$ is even. This proves (1).

Now pick an arbitrary component $a_{i}$ of $\alpha \setminus \gamma$. Since $a_{i}$ departs from and arrives at $\gamma$ on the same side, the indices of $(\alpha, \gamma)$ at the two endpoints of $a_{i}$ are different. This implies that the indices alternate along $\alpha$. Similarly, the indices alternate along $\gamma$, proving (2).

Now fix an $a_{i}$ with endpoints $p$ and $q$, separating $\gamma$ into two segments $\Gamma_{1}$ and $\Gamma_{2}$. Without loss of generality, assume that the index of $(\alpha, \gamma)$ is 1 at $p$ and $-1$ at $q$. Since the indices of $(\alpha, \gamma)$ alternate along $\gamma$, the number of $c_{j}$'s along $\Gamma_{1}$ is odd. Similarly, the number of $c_{j}$'s along $\Gamma_{2}$ is odd, so their minimum $N_{\gamma}(a_{i})$ is also odd. Now (3) follows.

Now suppose further that $i(\alpha, \gamma) > 2$. If $N_{\gamma}(a_{i}) = N_{\gamma}(a_{i'})$ for some $a_{i} \in A_{1}$ and $a_{i'} \in A_{2}$, then some twist $T_{\gamma}^{j}$ will tie them into a single curve, while other segments will combine to form  at least one more curve. This again contradicts the assumption, so it cannot happen. It completes the proof of (4).
\end{proof}

\begin{proposition}\label{prop:lengthEqn1}
Let $\{\alpha_{i}, \beta_{i}, \gamma_{i}, \epsilon_{i}\}_{i \in \Z}$ be essential or boundary curves on $S$, where each of the collections $\{\alpha_{i}\}_{i \in \Z} \cup \{\gamma_{0}\}$ and $\{\gamma_{i}\}_{i \in \Z} \cup \{\alpha_{0}\}$  consists of distinct curves. 

Suppose that for each of \[
(\eta_{1}, \ldots, \eta_{4}) = \left\{ \begin{array}{c} (\alpha_{i}, \gamma_{0}, \alpha_{i-1}, \alpha_{i+1}) \\ 
(\alpha_{i-1}, \alpha_{i}, \gamma_{0}, \beta_{i}) \\  (\gamma_{i}, \alpha_{0}, \gamma_{i-1}, \gamma_{i+1}) \\ (\gamma_{i-1}, \gamma_{i}, \alpha_{0}, \epsilon_{i}) \end{array}\right\},
\]
Equation \ref{eqn:lengthIDSet1}: $$f(X;\eta_{1}, \eta_{2}) = g(X;\eta_{3}, \eta_{4})$$ where \[
	f(X;\eta_{1}, \eta_{2}):= 2 \cosh \frac{l_{X}(\eta_{1})}{2}  \cosh \frac{l_{X}(\eta_{2})}{2},
	\]\[
	g(X;\eta_{1}, \eta_{2}) := \cosh \frac{l_{X}(\eta_{1})}{2} +  \cosh \frac{l_{X}(\eta_{2})}{2},
	\] holds for every $X \in \T(S)$. Then  $$i(\alpha_{i}, \gamma_{0}) = 1 \mbox{ and }\{\alpha_{1}, \alpha_{-1}\} = \{T_{\gamma_{0}}^{1}(\alpha_{0}), T_{\gamma_{0}}^{-1}(\alpha_{0}) \}.$$
\end{proposition}

\begin{proof}
For convenience, we will denote $\alpha_{0}$ by $\alpha$ and $\gamma_{0}$ by $\gamma$. Since $f$ and $g$ are symmetric with respect to the curves involved, the assumption still holds after relabelling $\alpha_{i}$ as $\alpha_{-i}$ and $\beta_{i}$ as $\beta_{-i + 1}$. We will perform such a relabelling in the case $\alpha_{1}$ is equal to $T_{\gamma}^{-1}(\alpha)$. Similarly, we relabel $\gamma_{i}$ into $\gamma_{-i}$ and $\epsilon_{i}$ into $\epsilon_{-i+1}$ in case $\gamma_{1}$ is equal to $T_{\alpha}^{-1}(\gamma)$. 

\emph{Step 1}. Proving that $\alpha_{i} = T_{\gamma}^{i}(\alpha)$ and $\gamma_{i} = T_{\alpha}^{i}(\gamma)$ for $i \ge -1$.

Using Lemma \ref{lem:fiveIDSuff} we deduce that $i(\alpha, \gamma) > 0$ and $\{\alpha_{1}, \alpha_{-1}\} = \{T_{\gamma}^{1}(\alpha), T_{\gamma}^{-1}(\alpha)\}$. However, $\alpha_{1}$ is not equal to $T_{\gamma}^{-1}(\alpha)$ due to the relabelling procedure. Thus we obtain that $\alpha_{i} = T_{\gamma}^{i}(\alpha)$ for $i=-1, 0, 1$. 

We further assume $\alpha_{i} = T_{\gamma}^{i}(\alpha)$ for $i=-1, 0, \ldots, n$ as the induction hypothesis. Applying Lemma \ref{lem:fiveIDSuff} to curves $(\alpha_{n-2}, \ldots, \alpha_{n+2}, \gamma, \beta_{n}, \beta_{n+1})$, we deduce that $\alpha_{i+1} = T_{\gamma}^{1}(\alpha_{i}) = T_{\gamma}^{i+1}(\alpha)$. Thus, by mathematical induction, we conclude that $\alpha_{i} = T_{\gamma}^{i}(\alpha)$ for $i \ge -1$. Exactly the same argument shows that $\gamma_{i}=T_{\alpha}^{i}(\gamma)$ for $i \ge -1$.

\emph{Step 2}. Proving that $i(\alpha, \gamma) = 1$.

Let $k:= i(\alpha, \gamma)$, and assume that $k > 2$. Since we have proven $\alpha_{i} = T_{\gamma}^{i}(\alpha)$ for each $i > -1$, these are all single curves, not multicurves composed of several disjoint curves. Similarly, $T_{\alpha}^{i}(\gamma) = \gamma_{i}$ are all single curves. Then by Lemma \ref{lem:joinEachSide}, $k$ is even (hence $k/2$ is an integer) and the segments of $\alpha \setminus \gamma$ are partitioned into $A_{1} = \{a_{1}, \ldots, a_{k/2}\}$ and $A_{2} = \{b_{1}, \ldots, b_{k/2}\}$. Moreover, $N_{1} := \{N_{\gamma}(a_{i}) : a_{i} \in A_{1}\}$ and $N_{2}:= \{N_{\gamma}(b_{i}) : b_{i} \in A_{2}\}$ become disjoint sets of odd integers.

Without loss of generality, assume $\min N_{1} > \min N_{2}$ and $N_{\gamma}(a_{j}) = \min N_{1}$. We pick $b_{l}$ such that $N_{\gamma}(b_{l})= \max \{n \in N_{2} : n < \min N_{1}\}$. Note that $1\le N_{\gamma}(b_{l}) \le N_{\gamma}(a_{j}) - 2$.

Some fractional Dehn twist $T_{\gamma}^{i}$ ties $a_{j}$ with $b_{l}$ in a manner that $\gamma(a_{j})$, $\gamma(b_{l})$ overlap each other. In other words, $a_{j}$ is adjacent to $b_{l}$ in $\alpha_{i} = T_{\gamma}^{i}(\alpha)$. $b_{l}$ is then adjacent to yet another segment $a_{j'}$ in $\alpha_{i}$. We now define a curve $\sigma$ by concatenating $a_{j}$, $a_{j'}$ and $b_{l}$ twice, together with two segments $c_{1}$, $c_{2}$ along $\gamma$ as in Figure \ref{fig:4arcsCurves}.

\begin{figure}[ht]
\begin{tikzpicture}[scale=1]
\def\c{0.25};
\draw (-5, 0) -- (7, 0);

\foreach \i in {-3, -2, 2, 3, 4, 5}{
\draw (\i, -1) -- (\i, 1);
}
\draw (-4, -1) -- (-4, 1) arc (180:0:2.5) -- (1, -1) arc (0:-180:1) -- (-1, 1) arc (180:0:3.5) -- (6, -1);
\draw[dashed] (-4+\c, 0-\c) -- (-4+\c, 1) arc (180:0:2.5-\c) -- (1-\c, -1) arc (0:-180:1-\c) -- (-1+\c, 1) arc (180:0:3.5-\c) -- (6-\c, 0-\c) -- (1+\c, 0-\c) -- (1+\c, -1) arc (0:-180:1+\c) -- (-1-\c, 0-\c) -- cycle;

\draw (-5.28, 0) node {$\gamma$};
\draw (-4, -1.2) node {$\alpha_{i}$};
\draw (-3.55, 2.9) node {$a_{j}$};
\draw (5.4, 3.5) node {$a_{j'}$};
\draw (-1, -2.1) node {$\sigma$};
\draw (-3.5, -0.5) node {$B$};
\draw (-3.5, 0.5) node {$A$};
\draw (-1.5, -0.5) node {$F$};
\draw (-1.5, 0.5) node {$E$};
\draw (1.5, -0.5) node {$F$};
\draw (1.5, 0.5) node {$E$};
\draw (4.5, -0.5) node {$D$};
\draw (4.5, 0.5) node {$C$};
\draw (5.5, -0.5) node {$B$};
\draw (5.5, 0.5) node {$A$};
\draw (0, -1.5) node {$A'$};

\fill (-4, -0.03) -- (-1, -0.03) -- (-1, 0.03) -- (-4, 0.03);
\fill (6, -0.03) -- (1, -0.03) -- (1, 0.03) -- (6, 0.03);

\draw (-2.5, 0.23) node {$c_{1}$};
\draw (3.5, 0.23) node {$c_{2}$};

\draw (0.707106781186548, -1.707106781186548) arc (225:315:0.5);
\draw (1.6, -1.7) node {$b_{l}$};

\end{tikzpicture}
\caption{Configurations of $\alpha_{i}$ and $\sigma$. Here $\alpha_{i}$ and $\gamma$ are presented in minimal position. Note that $\sigma$ and $\alpha_{i}$ are intersecting at $6$ points.}
	\label{fig:4arcsCurves}
\end{figure}
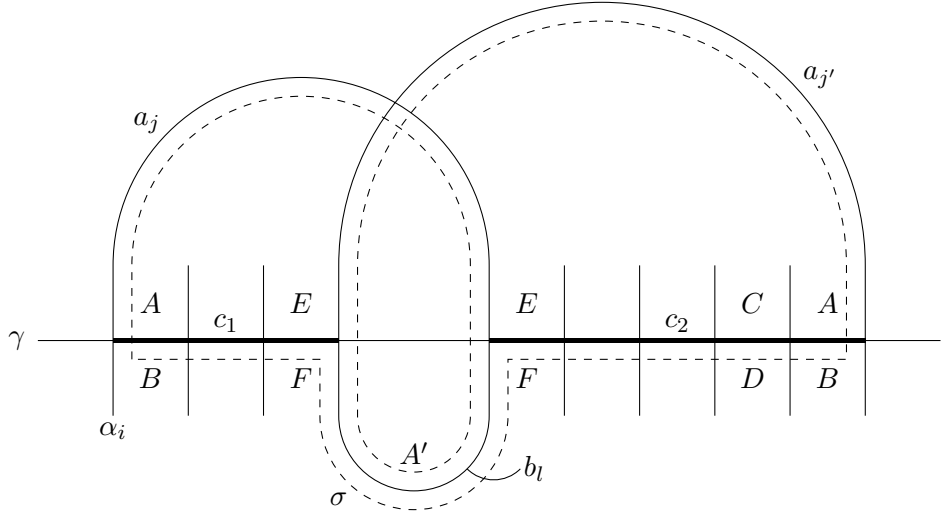

The number of segments of $\gamma \setminus \alpha_{i}$ present in Figure \ref{fig:4arcsCurves} is at most
\[
N_{\gamma}(a_{j}) + \max \{N_{\gamma}(a_{j'}), k - N_{\gamma}(a_{j'})\} - N_{\gamma}(b_{l}) \le N_{\gamma}(a_{j}) + (k - N_{\gamma}(a_{j})) - 1< k.
\]
Thus, $c_{1}$ and $c_{2}$ do not overlap and $\sigma$ is indeed a simple curve. We also note that $c_{1}$ contains at least \[
N_{\gamma}(a_{j}) - N_{\gamma}(b_{l}) \ge 2
\]
segments of $\gamma \setminus \alpha$, and so does $c_{2}$. We now claim the following lemma.

\begin{lem}\label{lem:mainMinPos}
Let $\sigma$ be the concatenation of $a_{j}, a_{j'}, b_{l}$ and parts of $\gamma$ as described in Step 2 of the proof of Proposition \ref{prop:lengthEqn1}. Then $i(\sigma, \gamma) = 4$ and $i(\sigma, T_{\gamma}^{\pm 1}(\alpha_{i})) \le i(\sigma, \alpha_{i}) + 2$.
\end{lem}

We presently postpone the proof of this lemma and first finish the proof of the theorem. Consider a pants decomposition on $S$ containing $\sigma$ and pinch $\sigma$. Then $f(\alpha_{i}, \gamma)$ grows in the order of $l(\sigma)^{-i(\alpha_{i}, \sigma) -4}$, while $g(T_{\gamma}^{1}(\alpha_{i}), T_{\gamma}^{-1}(\alpha_{i}))$ grows in the order of $l(\sigma)^{-i(\alpha_{i}, \sigma)-2}$ at most. This contradiction rules out the case that $i(\alpha, \gamma) > 2$.

In conclusion, $\alpha$ and $\gamma$ satisfy either \begin{enumerate}
\item $i(\alpha, \gamma) = 1$ or
\item $i(\alpha, \gamma) = 2$. 
\end{enumerate}

Let us assume the case (2). Note that Lemma \ref{lem:joinEachSide} asserts $i_{alg}(\alpha, \gamma) = 0$. Thus, $f(X;\alpha, \gamma) = g(X; T_{\gamma}^{1}(\alpha), T_{\gamma}^{-1}(\alpha))+ f(X;\delta_{2}, \delta_{3}) + f(X;\delta_{1}, \delta_{4})$ holds on all of $\T(S)$, where $\delta_{i}$ are curves as in Fact \ref{fact:lengthID2}. Recall that we also have an identity $f(X;\alpha, \gamma) = g(X; T_{\gamma}^{1}(\alpha), T_{\gamma}^{-1}(\alpha))$. 
However, their difference \[
2 \cosh \frac{l_{X}(\delta_{2})}{2} \cosh \frac{l_{X}(\delta_{3})}{2}  + 2\cosh \frac{l_{X}(\delta_{1})}{2} \cosh \frac{l_{X}(\delta_{4})}{2} 
\]
never vanishes. This contradiction excludes the case (2), and we conclude $i(\alpha, \gamma)=1$.
\end{proof}

\begin{proof}[Proof of Lemma \ref{lem:mainMinPos}]
We claim that $\sigma$ and $\gamma$ in Figure \ref{fig:4arcsCurves} are in minimal position. Using the bigon criterion, it suffices to show that there is no  embedded disc bounded by a component segment of  $\sigma \setminus \gamma$ and a component segment of $\gamma \setminus \sigma$. Note that $\sigma \setminus \gamma$ consists of 4 segments: one parallel to $a_{j}$, one parallel to $a_{j}'$, one parallel to $b_{l}$ and one parallel to the concatenation of $c_{1}, b_{l}$ and $c_{2}$. Here, recall that $a_{j}$, $a_{j}'$ and $b_{l}$ are components of $\alpha_{i} \setminus \gamma$. Since $\alpha_{i}$ and $\gamma$ are drawn in minimal position, none of $a_{j}$, $a_{j}'$ and $b_{l}$ can bound a disc together with $\gamma$. Now consider the segment of $\sigma$ that is parallel to $c_{1}$, $b_{l}$ and $c_{2}$. This segment is adjacent to two complementary regions of $S \setminus (\sigma \cup \gamma)$. One region is a quadrangle composed of two arcs from $\gamma$ and two arcs from $\sigma$, which does not count as a bigon. Another region is homotopic to a complementary region made by $b_{l}$ and $\gamma$. Again, since $\alpha_{i}$ and $\gamma$ is in a minimal position, this region also cannot be a bigon. This concludes the minimal position of $\sigma$ and $\gamma$ in Figure \ref{fig:4arcsCurves}, and $i(\sigma, \gamma) = 4$.

We next claim that $\sigma$ and $\alpha_{i}$ are also in minimal position.Note that $\sigma \setminus \alpha_{i}$ consists of two components that are near $b_{l}$ and other components that are parallel to $c_{1}$ or $c_{2}$. In Figure \ref{fig:4arcsCurves}, $\sigma \setminus \alpha_{i}$ consists of: \begin{itemize}
\item the long component $L_{1}$ passing through region $A$'s, $B$'s and $A'$, containing subsegments parallel to $a_{j}$, $b_{l}$ and $a_{j}'$, respectively;
\item the short component $L_{2}$ passing through region $F$'s, containing a subsegment parallel to $b_{l}$;
\item components parallel to $c_{1}$ or $c_{2}$ (such as the one passing through region $D$).
\end{itemize}

Suppose first that a component of $\sigma \setminus \alpha_{i}$ parallel to $c_{2}$  forms a bigon with $\alpha_{i}$. (In Figure \ref{fig:4arcsCurves}, consider the region containing part $D$.) Then such a bigon is homotopic to another bigon formed by $\alpha_{i}$ and $c_{2}$. This contradicts the fact that $\alpha_{i}$ and $\gamma$ are in minimal position. Hence, such a bigon does not exist. Similarly, no component of $\sigma \setminus \alpha_{i}$ parallel to $c_{1}$ can form a bigon with $\alpha_{i}$.

Now we are left with complementary regions of $\sigma \cup \alpha_{i}$ that are close to $b_{l}$. These are: \begin{itemize}
\item the region containing letter $A$ (and $A'$);
\item the region containing letter $B$;
\item the region containing letter $E$, and
\item the region containing letter $F$.
\end{itemize}

Suppose first that the region containing letter $A$ (and $A'$) is a bigon. This means that $L_{1}$ is homotopic (relative to the endpoints) to a component of $\alpha_{i} \setminus \sigma$, which we denote by $L$. Via homotopy, we can bring $L$ very close to $L_{1}$, from the side opposite to $a_{j}$, $b_{l}$ and $a_{j}'$. At this moment, the segment of $L$ parallel to $a_{j}$ is an element of $A_{1}$, whole value of $N_{\gamma}$ is $N_{\gamma}(a_{j}) - 2$. This contradicts the minimality of $N_{\gamma}(a_{j})$ among $N_{1}$.

Now suppose that the region containing letter $B$ is a bigon. Then by pushing the middle of $L_{1}$ toward $a_{j} \cup b_{l} \cup a_{j}'$ via homotopy, we obtain two bigons containing letter $B$. Each of these bigons consist of one segment of $\alpha_{i} \setminus \sigma$ and a horizontal segment that can be homotoped to $\gamma$. Hence, we obtain a bigon bordered by $\alpha_{i}$ and $\gamma$, contradicting their minimal position. 

The region containing letter $E$ is treated in a similar way. If it were a bigon, then we can push the middle of $L_{2}$ toward $b_{l}$ via homotopy. We then obtain two bigons containing letter $E$, each homotopic to a complementary region of $\alpha_{i} \cup \gamma$, which is absurd.

Finally, if the region containing letter $F$ were a bigon, then $b_{l}$ would be homotopic to another segment $b_{l'}$ satisfying $N_{\gamma}(b_{l'}) = N_{\gamma}(b_{l}) + 2$. This contradicts the maximality of $b_{l}$.

As a result of the discussion so far, the curves in Figure \ref{fig:4arcsCurves} are pairwise in minimal position. Together with the representative of $T_{\gamma}^{1}(\alpha_{i})$ drawn in Figure \ref{fig:4arcsCurvesTwist}, we can then deduce that \[
i(\gamma, \sigma) = 4\quad \textrm{and}\quad i(T_{\gamma}^{\pm1}(\alpha_{i}), \sigma) \le i(\alpha_{i}, \sigma) + 2.
\]

\begin{figure}[ht]
\begin{tikzpicture}[scale=1]
\def\c{0.25};
\draw (-5, 0) -- (7, 0);

\foreach \i in {-3, -2, 3, 4, 5}{
\draw (\i, 1) -- (\i, 0) -- (\i - 1, -0.5) --  (\i-1, -1);
}
\draw (-4.8, -0.4) -- (-4, 0)-- (-4, 1) arc (180:0:2.5) -- (1, 0.5) -- (0.4, 0.2);

\draw (2, 1) -- (2, 0.5) -- (1, 0) -- (1, -1) arc (0:-180:1) -- (-1, -0.5) -- (-0.4, -0.2);

\draw (-2, -1) -- (-2, -0.5) -- (-1, 0) -- (-1, 1) arc (180:0:3.5) -- (6, -0.3) -- (5, -0.6) --  (5, -1);

\draw (6, -1) -- (6, -0.6) -- (6.8, -0.36);
\draw[dashed] (-4+\c, 0-\c) -- (-4+\c, 1) arc (180:0:2.5-\c) -- (1-\c, -1) arc (0:-180:1-\c) -- (-1+\c, 1) arc (180:0:3.5-\c) -- (6-\c, 0-\c) -- (1+\c, 0-\c) -- (1+\c, -1) arc (0:-180:1+\c) -- (-1-\c, 0-\c) -- cycle;

\draw (-5.28, 0) node {$\gamma$};
\draw (-4, -1.23) node {$T_{\gamma}^{1}(\alpha_{i})$};
\draw (-3.55, 2.9) node {$a_{j}$};
\draw (5.4, 3.5) node {$a_{j'}$};
\draw (-1, -2.1) node {$\sigma$};

\foreach \i in {-3, -2, -1, 3, 4, 5}{
\draw[very thick] (\i, 0.5) -- (\i, 0) -- (\i-1, -0.5) -- (\i - 1, -0.7);
}
\draw[very thick] (1, 0.7) -- (1, 0.5) -- (0.4, 0.2);
\draw[very thick] (-1, -0.7) -- (-1, -0.5) -- (-0.4, -0.2);

\draw (0.707106781186548, -1.707106781186548) arc (225:315:0.5);
\draw (1.6, -1.7) node {$b_{l}$};

\end{tikzpicture}
\caption{Configurations of $T_{\gamma}^{1}(\alpha_{i})$ and $\sigma$. Two curves are intersecting at 8 points.}
	\label{fig:4arcsCurvesTwist}
\end{figure}
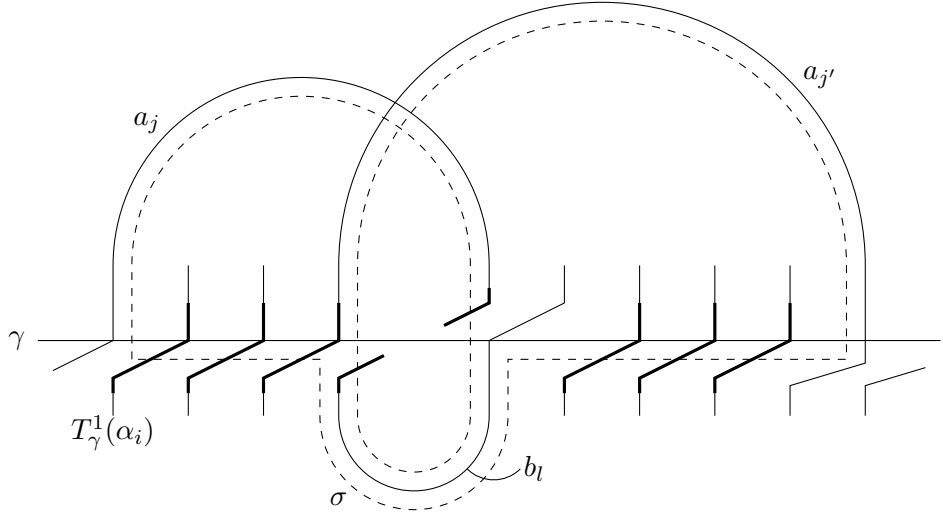
\end{proof}

Using this result, we can construct a subset of $\T(S)$ as follows. Let $\{\alpha_{i}\}_{i=-k-1}^{k+1}$, $\{\beta_{i}\}_{i=-k}^{k+1}$, $\{\gamma_{i}\}_{i=-k-1}^{k+1}$ and $\{\epsilon_{i}\}_{i=-k}^{k+1}$ be essential or boundary curves on $S$, where $\{\alpha_{i}\} \cup \{\gamma_{0}\}$ and $\{\gamma_{i}\} \cup \{\alpha_{0}\}$ consist of distinct curves and $k=i(\alpha_{0}, \gamma_{0}) \neq 1$. Since curves are compact, they are contained in a finite-type subsurface $S_{1}$ of $S$ bounded by some curves $C_{i_{1}}$, $\ldots$, $C_{i_{n}}$. 

If we declare the function \[
\mathfrak{h}(X; \eta_{1}, \ldots, \eta_{4}) := f(X; \eta_{1}, \eta_{2}) - g(X;\eta_{3}, \eta_{4})\]
for each  $(\eta_{1}, \ldots, \eta_{4})$ among $$\begin{aligned}
\{(\alpha_{i}, \gamma_{0}, \alpha_{i-1}, \alpha_{i+1})\}_{i=-k}^{k}, \hspace{1cm} & \{(\alpha_{i}, \alpha_{i-1}, \gamma_{0}, \beta_{i})\}_{i=-k+1}^{k},\\
\{(\gamma_{i}, \alpha_{0}, \gamma_{i-1}, \gamma_{i+1})\}_{i=-k}^{k}, \hspace{1cm} & \{(\gamma_{i}, \gamma_{i-1}, \alpha_{0}, \eta_{i})\}_{i=-k+1},^{k}
\end{aligned}$$ then it becomes an analytic function on $\T(S_{1})$. We index them into a single function $(\mathfrak{h}_{i}) : \T(S_{1}) \rightarrow \mathbb{R}^{8k+2}$. Now the proof of Proposition \ref{prop:lengthEqn1} indicates that $(\mathfrak{h}_{i})$ does not vanish identically on $\T(S_{1})$. Using Lemma \ref{lem:analytic}, we then construct a countable family $\mathcal{F} = \{F_{n}\}$ of submanifolds of $\T(S_{1})$ such that $(\mathfrak{h}_{i})$ does not vanish outside $\cup_{n} F_{n}$. Then $\tilde{F}_{n}:= \pi_{S_{1}}^{-1}(F_{n}) \subseteq \T(S)$ is nowhere dense in $\T(S)$ and their union $\tilde{F} = \cup_n \tilde{F}_n$ becomes a meagre set. \label{page.F}

\subsection{Second length identity: generalized shirt} \label{subsec:convofIDSet2}

We now prove the following converse of Lemma \ref{lem:lengthIDSet2}.

\begin{proposition}\label{prop:lengthEqn2}
Let $\{\alpha_{i}, \beta_{i}, \gamma_{i}, \epsilon_{i}\}_{i \in \Z}$ be essential or boundary curves on $S$, where each of the collections $\{\alpha_{i}\}_{i \in \Z} \cup \{\gamma_{0}\}$ and $\{\gamma_{i}\}_{i \in \Z} \cup \{\alpha_{0}\}$ consists of distinct curves. 
Further, let $\{\delta_{1}, \ldots, \delta_{4}\}$ be curves on $S$.

As in Lemma \ref{lem:lengthIDSet2}, suppose that Equation \ref{eqn:lengthIDSet2}: $$
f(X; \eta_{1}, \eta_{2}) = g(X; \eta_{3}, \eta_{4}) + f(X; \eta_{5}, \eta_{6}) + f(X; \eta_{7} ,\eta_{8})$$ holds for all $X \in \T(S)$ and with each  choice of $$(\eta_{1}, \ldots, \eta_{8})= \left\{\begin{array}{cc} (\alpha_{2i}, \gamma_{0}, \alpha_{2i-1}, \alpha_{2i+1}, \delta_{2}, \delta_{3}, \delta_{1}, \delta_{4}) \\ (\alpha_{2i+1}, \gamma_{0}, \alpha_{2i}, \alpha_{2i+2}, \delta_{1}, \delta_{3}, \delta_{2}, \delta_{4}) \\
	(\gamma_{2i},\alpha_{0}, \gamma_{2i-1}, \gamma_{2i+1}, \delta_{2}, \delta_{3}, \delta_{1}, \delta_{4}) \\ 
	(\gamma_{2i+1}, \alpha_{0}, \gamma_{2i}, \gamma_{2i+2}, \delta_{1}, \delta_{2}, \delta_{3}, \delta_{4}) \\
	(\alpha_{i-1}, \alpha_{i}, \gamma_{0}, \beta_{i}, \delta_{1}, \delta_{2}, \delta_{3}, \delta_{4}) \\
	(\gamma_{i-1}, \gamma_{i}, \alpha_{0}, \epsilon_{i}, \delta_{1}, \delta_{3}, \delta_{2}, \delta_{4}) \end{array}\right.$$ Then $\alpha_{0} \cup \gamma_{0}$ forms a spine of an immersed subsurface $\psi : S' \rightarrow S$, where $S'$ is a generalized shirt. Moreover, $\{\alpha_{-1}, \alpha_{1}\} = \{T_{\gamma}^{-1}(\alpha_0), T_{\gamma}^{1}(\alpha_0)\}$.

Further, we can label the peripheral curves of $S'$ by $\{\epsilon_{1}, \ldots, \epsilon_{4}\}$ such that:
\begin{itemize}
\item $\delta_{i}$, $\epsilon_{i}$ are both bounding (possibly distinct) punctures  or $\delta_{i} = \epsilon_{i}$, and 
\item $\alpha_{0}$ separates $\{\epsilon_{1}, \epsilon_{3}\}$ from $\{\epsilon_{2}, \epsilon_{4}\}$ and $\gamma_{0}$ separates $\{\epsilon_{1}, \epsilon_{2}\}$ from $\{\epsilon_{3}, \epsilon_{4}\}$.
\end{itemize}
\end{proposition} 

\begin{proof}
As in the proof of Proposition \ref{prop:lengthEqn1}, we will denote $\alpha_{0}$ by $\alpha$ and $\gamma_{0}$ by $\gamma$. Moreover, we may assume that $\alpha_{1} \neq T_{\gamma}^{-1}(\alpha)$ and $\gamma_{1} \neq T_{\alpha}^{-1}(\gamma)$.

We note that Lemma \ref{lem:fiveIDSuff} equally applies, since $f(X;\eta, \eta') \ge 0$ for any curves $\eta$ and $\eta'$. As in the proof of Proposition \ref{prop:lengthEqn1}, we thus obtain $k:=i(\alpha, \gamma)>0$ and $\alpha_{i} = T_{\gamma}^{i}(\alpha)$, $\gamma_{i} = T_{\alpha}^{i}(\gamma)$ for $i \ge -1$.

We now prove that $\delta_{1}$ and $\gamma$ are disjoint. From the assumption, we have \[
f(X;\alpha_{i}, \gamma) = g(X;\alpha_{i-1}, \alpha_{i+1}) + f(X;\delta_{1}, \delta_{4}) + f(X;\delta_{2}, \delta_{3})
\]
for even $i$. Here note that the geometric intersection number $i(\alpha_{i}, \gamma)$ between $\alpha_i$ and $\gamma$ is equal to $k$ since  $\alpha_i = \T_{\gamma}^i (\alpha)$ for $i \ge -1$.
Now, if $\gamma$ intersects $\delta_{1}$, then \[
i(\alpha_{n k}, \delta_{1}) = i(T_{\gamma}^{n k}(\alpha), \delta_{1}) \ge nk\, i(\gamma, \delta_{1}) - i(\alpha, \delta_{1}) \ge nk - i(\alpha, \delta_{1})
\]
by Fact \ref{fact:triangleIneqInters}. Thus, if we take $i=nk$ to be an even integer larger than $i(\alpha, \delta_{1}) + k + 1$ and pinch $\alpha_{i}$, then $f(X;\alpha_{i}, \gamma)$ grows in the order of $l(\alpha_{i})^{-k}$ while $f(X;\delta_{1}, \delta_{4})$ grows in the order at least of $l(\alpha_{i})^{-(k+1)}$, a contradiction. Thus $i(\gamma, \delta_{1}) = 0$ and similarly $i(\alpha, \delta_{1}) = 0$. Other $\delta_{i}$'s can be dealt with similarly, so we find that  $\delta_{i}$'s and  $\alpha \cup \gamma$ are disjoint.

Now, if $k=i(\alpha, \gamma) > 2$, then we construct $\sigma$ as in the step 2 of the proof of Proposition \ref{prop:lengthEqn1}. Since $\sigma \subseteq N(\alpha) \cup N(\gamma)$ is disjoint from all $\delta_{j}$'s, we see that $f(X;\delta_{j}, \delta_{j'})$ remains bounded while pinching $\sigma$. Accordingly, the same contradiction follows from Lemma \ref{lem:mainMinPos} by comparing each side of \[
f(X;\alpha_{i}, \gamma) = g(X; \alpha_{i-1}, \alpha_{i+1}) + f(X; \delta_{j}, \delta_{j'}) + f(X; \delta_{j''}, \delta_{j'''}).
\]
while pinching $\sigma$.

We are thus led to the same dichotomy as in the proof of Lemma \ref{lem:mainMinPos}: \begin{enumerate}
\item $i(\alpha, \gamma) = 1$ or
\item $i(\alpha, \gamma) = 2$ and $i_{alg}(\alpha, \gamma)=0$.
\end{enumerate}
As noted before, we cannot simultaneously have $f(X;\alpha, \gamma) = g(X;\alpha_{-1}, \alpha_{1})$ and $f(X;\alpha, \gamma) = g(X;\alpha_{-1}, \alpha_{1})+f(X;\delta_{2}, \delta_{3}) + f(X;\delta_{1}, \delta_{4})$ on all of $\T(S)$. The latter is already assumed true, while $i(\alpha, \gamma) = 1$ forces the former. This contradiction rules out the case $i(\alpha, \gamma) = 1$ and consequently, $\alpha \cup \gamma$ must be a spine of an immersed subsurface of $S'$ that is a generalized shirt.

Let $\{\delta'_{j}\}$ be the boundaries/punctures of $S'$ labelled as in Figure \ref{fig:spineof4holed}. 
At this moment, Fact \ref{fact:lengthID2} and the assumption gives the following set of identities on all of $\T(S)$;\begin{equation}\begin{aligned}\label{eqn:tempEqn1}
f(X;\delta'_{2}, \delta'_{3}) + f(X;\delta'_{1}, \delta'_{4}) & = f(X;\alpha, \gamma) - g(X; \alpha_{-1}, \alpha_{1})\\ & = f(X;\delta_{2}, \delta_{3}) + f(X;\delta_{1}, \delta_{4}),
\end{aligned}\end{equation}
\begin{equation}\begin{aligned}\label{eqn:tempEqn2}
f(X;\delta'_{1}, \delta'_{3}) + f(X;\delta'_{2}, \delta'_{4}) & = f(X;\alpha_{1}, \gamma) - g(X; \alpha, \alpha_{2})\\ & = f(X;\delta_{1}, \delta_{3}) + f(X;\delta_{2}, \delta_{4}).
\end{aligned}\end{equation}
Let $\Lambda$ ($\Gamma$, resp.) be the set of $\delta'_{j}$'s ($\delta_{j}$'s, resp.) that are not punctures.
Note that the LHS (RHS, resp.) of Equation \ref{eqn:tempEqn1} can attain the value 4 at one and all $X \in \T(S)$ if and only if all of $\delta'_{j}$ ($\delta_{i}$, resp.) are punctures. This settles the case when $S = S'$ is a 4-punctured sphere, and we now assume $\Lambda, \Gamma \neq \emptyset$.

Since $S'$ is an immersed subsurface, $\Lambda$ consists of disjoint curves. Thus we can fix a pants decomposition including $\Lambda$ and pinch simultaneously. The LHS of Equation \ref{eqn:tempEqn1} converges to 4 so the RHS should also do so. This is possible only when all of the $l_{X}(\delta_{i})$ terms on the RHS tends to $0$ during the pinching, so $\emptyset \neq \Gamma \subseteq \Lambda$.

We now pick some $\delta_{i}$ in $\Gamma$. Since $\Gamma \subseteq \Lambda$, $\delta_{i} = \delta'_{j}$ for some $j$. By applying one of the following permutations on the indices of $\delta_{i}$, $\delta'_{j}$ \begin{equation} \begin{aligned} \label{eqn.permutate}
(1, 2, 3, 4) \mapsto (1, 2, 3, 4), & \hspace{1cm} (1, 2, 3, 4) \mapsto (2, 1, 4, 3),\\
(1, 2, 3, 4) \mapsto (3, 4, 1, 2), & \hspace{1cm} (1, 2, 3, 4) \mapsto (4, 3, 2, 1),
\end{aligned}\end{equation}
under which Equation \ref{eqn:tempEqn1} and \ref{eqn:tempEqn2} remain unchanged, we may assume that $\delta_{1} = \delta'_{1}\in \Gamma$.

We now show $\delta_{i} = \delta'_{i}$ (up to the permutations above) for all $i$. We first increase $l_{X}(\delta'_{1}) = l_{X}(\delta_{1})$ to infinity, whilst fixing the lengths of $\delta'_{j} \in \Lambda \setminus \{\delta_{1}\}$ as $t$. \begin{itemize}
\item If $\delta'_{1} = \delta'_{3}$, then $\delta'_{2}$, $\delta'_{4}$ are distinct from $\delta'_{1}$. In this case, the LHS of Equation \ref{eqn:tempEqn2} grows in the order of $e^{ l_{X}(\delta'_{1})}$. This implies that at least one of $\delta_{3} = \delta_{1} = \delta'_{1}$ or $\delta_{2} = \delta_{4} = \delta_{1} = \delta'_{1}$ holds. However, the latter case is excluded by comparing $f(X; \delta'_{2}, \delta'_{4})$ and $f(X; \delta_{1}, \delta_{3})$. Thus, we conclude $\delta'_{1} = \delta'_{3} = \delta_{1} = \delta_{3}$ and $f(X; \delta'_{1}, \delta'_{3}) = f(X; \delta_{1}, \delta_{3})$.

At this moment, if $\delta'_{2}$ and $\delta'_{4}$ are punctures, then we have \[
\{ \delta'_{1} = \delta'_{3} = \delta_{1} = \delta_{3}\} \subseteq \Gamma \subseteq \Lambda = \{\delta'_{1} = \delta'_{3}= \delta_{1} = \delta_{3}\},
\]
which is the desired equality. Next, if $\delta'_{2}$ is an essential curve,  then we increase $l_{X}(\delta'_{2})$ to infinity whilst fixing lengths of $\delta'_{j} \in \Lambda \setminus \{\delta'_{2}\}$ as $t$.  Note that \[
\lim_{l_{X}(\delta'_{2}) \rightarrow \infty} \frac{\textrm{(LHS of Equation \ref{eqn:tempEqn2})}}{e^{l_{X}(\delta'_{2})/2}} = \left\{\begin{array}{cc}  +\infty & \delta'_{2} = \delta'_{4} \\ \cosh (t/2)  &\{\delta'_{2}, \delta'_{4} \}\subseteq  \Lambda, \delta'_{2} \neq \delta'_{4}  \\ 
1 & \{\delta'_{2}, \delta'_{4}\} = \{\delta'_{2}, \textrm{puncture}\},\end{array}\right.
\]
 \[
\lim_{l_{X}(\delta'_{2}) \rightarrow \infty} \frac{\textrm{(RHS of Equation \ref{eqn:tempEqn2})}}{e^{l_{X}(\delta'_{2})/2}} = \left\{\begin{array}{cc}  +\infty & \delta_{2} = \delta_{4} =\delta'_{2} \\ \cosh (t/2)  & \delta'_{2} \in \{\delta_{2}, \delta_{4}\} \subseteq \Lambda, \delta_{2} \neq \delta_{4} \\ 
1 &  \{\delta_{2} ,\delta_{4}\} = \{\delta'_{2}, \textrm{puncture}\} \\
0 & \delta'_{2} \notin \{\delta_{2}, \delta_{4}\}.
\end{array}\right.
\]
Since the two growth rates should match, the case $\delta'_{2} \notin \{\delta_{2}, \delta_{4}\}$ is impossible and $\delta'_{2}$ belongs to $\{\delta_{2}, \delta_{4}\}$. Similarly, by using Equation \ref{eqn:tempEqn1}, we deduce that $\delta'_{4} \in \{\delta_{2}, \delta_{4}\}$ whenever $\delta'_{4}$ is an essential curve. In conclusion we have $\{\delta_{2}, \delta_{4}\} = \{\delta'_{2}, \delta'_{4}\}$, and we have $(\delta'_1, \delta'_2, \delta'_3, \delta'_4) = (\delta_1, \delta_2, \delta_3, \delta_4)$ up to the permutation $(1, 2, 3, 4) \mapsto (3, 4, 1, 2)$ in Equation \ref{eqn.permutate}.
\item The case $\delta'_{1} = \delta'_{4}$ can be dealt with in the same way, by switching the role of Equation \ref{eqn:tempEqn1} and \ref{eqn:tempEqn2}. For example, now the LHS of Equation \ref{eqn:tempEqn1} grows in the order of $e^{l_{X}(\delta'_{1})}$, which forces $\delta_{1} = \delta_{4} = \delta'_{1}$ or $\delta_{2} = \delta_{3} = \delta_{1} = \delta'_{1}$, the latter being excluded by comparing $f(X; \delta'_{2}, \delta'_{3})$ and $f(X; \delta_{1}, \delta_{4})$. 

\item If $\delta'_{1} = \delta'_{2}$, then $\delta'_{3}$, $\delta'_{4}$ are distinct from $\delta'_{1}$. In this case, the LHS of Equation \ref{eqn:tempEqn1} and \ref{eqn:tempEqn2} grow in the order of $e^{l_{X}(\delta'_{1})/2}$. Accordingly, both $\delta_{3}$, $\delta_{4}$ are not $\delta'_{1}$. If moreover $\delta_{2}$ is not $\delta'_{1}$, then \[
\lim_{l_{X}(\delta'_{1}) \rightarrow \infty} \frac{\textrm{(RHS of Equation \ref{eqn:tempEqn2})}}{e^{l_{X}(\delta'_{1})/2}} = \left\{\begin{array}{cc}  1 & \delta_{3} \notin \Gamma \\ \cosh (t/2) & \delta_{3} \in \Gamma \setminus \{\delta_{1}\} \end{array}\right.
\]
according to whether $\delta_{3}$ is a curve or not. However,  note\[
\lim_{l_{X}(\delta'_{1}) \rightarrow \infty} \frac{\textrm{(LHS of Equation \ref{eqn:tempEqn2})}}{e^{l_{X}(\delta'_{1}) / 2}} = \left\{\begin{array}{cc} 2 & \delta'_{3}, \delta'_{4}, \notin \Lambda \\ 2 \cosh (t/2) & \delta'_{3}, \delta'_{4} \in \Lambda \setminus \{\delta'_{1}\} \\ 1+\cosh (t/2) & \textrm{otherwise},\end{array}\right.
\]
which gives a contradiction. Thus $\delta_{2} = \delta'_{1}$. Now we increase $l_{X}(\delta'_{3})$ and $l_{X}(\delta'_{4})$ separately to deduce that $\{\delta'_{3}, \delta'_{4}\} = \{\delta_{3}, \delta_{4}\}$.
\item The remaining case is that $\delta'_{1}$ is not equal to any of $\{\delta'_{2}, \delta'_{3}, \delta'_{4}\}$. We first set $l_{X}(\delta'_{3})$ as $3$ if $\delta'_{3} \in \Lambda$, and set lengths $l_{X}(\delta'_{j})$ for $\delta'_{j} \in \Lambda \setminus \{\delta'_{1}, \delta'_{3}\}$ as 2. Finally, we increase $l_{X}(\delta'_{1}) = l_{X}(\delta_{1})$ to infinity. We then observe\[
\lim_{l_{X}(\delta'_{1}) \rightarrow \infty} \frac{\textrm{(LHS of Equation \ref{eqn:tempEqn2})}}{e^{l_{X}(\delta'_{1})/2}} =\left\{\begin{array}{cc}  \cosh (3/2) & \delta'_{3} \in \Lambda \\ 1 & \textrm{otherwise},\end{array}\right.
\]\[
\lim_{l_{X}(\delta'_{1}) \rightarrow \infty} \frac{\textrm{(RHS of Equation \ref{eqn:tempEqn2})}}{e^{l_{X}(\delta'_{1})/2}} = \left\{\begin{array}{cc} \cosh (3/2) & \delta_{3} = \delta'_{3}  \in \Lambda \\ \cosh 1 & \delta_{3} \in \Lambda \setminus \delta'_{3} \\ 1 & \textrm{otherwise}. \end{array}\right.
\]
Thus we conclude that $\delta'_{3}$, $\delta_{3}$ are both bounding punctures or $\delta'_{3} = \delta_{3}$. Similar conclusion for $\delta'_{4}$, $\delta_{4}$ follows from Equation \ref{eqn:tempEqn1}. Finally, we increase $l_{X}(\delta'_{2})$ in Equation \ref{eqn:tempEqn2} to deduce the conclusion.
\end{itemize}
In any case, we can relabel $\delta_i$'s by one of the permutations in Equation \ref{eqn.permutate}, competing the proof.
\end{proof}

Using this result, we can construct a subset of $\T(S)$ as follows. Let $\{\alpha_{i}\}_{i=-(M+k+2)}^{M+k+2}$, $\{\beta_{i}\}_{i=-k}^{k}$, $\{\gamma_{i}\}_{i=-(M+k+2)}^{M+k+2}$, $\{\epsilon_{i}\}_{i=-k}^{k}$ and $\{\delta_{1}, \ldots, \delta_{4}\}$ be curves on $S$ such that: \begin{itemize}
\item $\{\alpha_{i}, \beta_{i}, \gamma_{i}, \epsilon_{i}\}$ are essential or boundary curves,
\item $\{\alpha_{i}\} \cup\{\gamma_{0}\}$, $\{\gamma_{i}\} \cup \{\alpha_{0}\}$ consist of distinct curves,
\item $i(\alpha_{0}, \gamma_{0}) = k$ and $i(\alpha_{0}, \delta_{i})$, $i(\gamma_{0}, \delta_{i}) \le M$, and 
\item $(\alpha_{0}, \gamma_{0}, \delta_{i})$ do not satisfy the conclusion of Proposition \ref{prop:lengthEqn2}. That is, either $\alpha_0 \cup \gamma_0$ does not form a spine of an immersed generalized shirt in $S$, or when they form such a spine of an immersed generlized shirt $S'$ in $S$, we have $\{\alpha_{-1}, \alpha_1\} \neq \{ T_{\gamma}^{-1}(\alpha_0), T_{\gamma}^1(\alpha_0)\}$ or there is no labelling of the peripheral curves of the $S'$ satisfying the conditions stated in Proposition \ref{prop:lengthEqn2}.
\end{itemize}
Since curves are compact, they are contained in a finite-type subsurface $S_{1}$ of $S$ bounded by some curves $C_{i_{1}}$, $\ldots$, $C_{i_{n}}$. 

Then \[
	\mathfrak{h}(X; \eta_{1}, \ldots, \eta_{8}) := f(X; \eta_{1}, \eta_{2}) - g(X;\eta_{3}, \eta_{4}) - f(X; \eta_{5}, \eta_{6}) - f(X; \eta_{7}, \eta_{8}),
\]
for the choices of $(\eta_{1}, \ldots, \eta_{8})$ specified as per Lemma \ref{lem:lengthIDSet2}, define  analytic functions on $\T(S_{1})$. We index them into a single function $(\mathfrak{h}_{i}) : \T(S_{1}) \rightarrow \mathbb{R}^{N}$ for some $N$. Now the proof of Proposition \ref{prop:lengthEqn2} indicates that $(\mathfrak{h}_{i})$ does not vanish identically on $T(S_{1})$. Using Lemma \ref{lem:analytic}, we then construct a countable family $\mathcal{G} = \{G_{n}\}$ of submanifolds of $\T(S_{1})$ such that $(\mathfrak{h}_{i}) \neq 0$ outside $\cup_{n} G_{n}$. Then $\tilde{G}_{n}:= \pi_{S_{1}}^{-1}(F_{n}) \subseteq \T(S)$ is nowhere dense in $\T(S)$ and their union $\tilde{G}$ becomes a meagre set.

We now gather all $\tilde{E}(\alpha, \beta)$ (page \pageref{page.E}), $\tilde{F}(\{\alpha_{i}\}, \ldots, \{\epsilon_{i}\})$ (page \pageref{page.F}), and $\tilde{G}(\{\alpha_{i}\}, \ldots, \{\epsilon_{i}\}, \{\delta_{i}\})$ that have been constructed so far, and denote their union by $V$. This is the union of a countable collection of meagre subsets of $\T(S)$, so $V$ is meagre. Since $\T(S)$ is locally homeomorphic to a complete metric space, 
we again invoke the Baire category theorem to deduce that $\T(S) \setminus V$ is dense in $\T(S)$. Hence Theorem \ref{thm:mcshane}, Proposition \ref{prop:lengthEqn1} and Proposition \ref{prop:lengthEqn2} imply the following proposition.

\begin{proposition}\label{prop:mainProp}
Suppose that  $X \in \T(S) \setminus V$ and let $\{\xi_{1}, \xi_{2}\}$, $\{\alpha_{i},\beta_{i}, \gamma_{i}, \epsilon_{i}\}_{i\in \Z}$ be essential or boundary curves on $S$. \begin{enumerate}
\item If $\xi_{1} \neq \xi_{2}$, then $l_{X}(\xi_{1}) \neq l_{X}(\xi_{2})$.
\item Suppose that at $X$, $\{\alpha_{i},\beta_{i}, \gamma_{i}, \epsilon_{i}\}_{i\in \Z}$ satisfy the identities of Lemma \ref{lem:lengthIDSet1} and each of $\{\alpha_{i}\}_{i \in \Z} \cup \{\gamma_{0}\}$, $\{\gamma_{i}\}_{i \in \Z} \cup \{\alpha_{0}\}$ contains no curves with the same length. Then $\alpha_0 \cup \gamma_0$ forms the spine of a generalized 1-holed torus and $\{\alpha_{-1}, \alpha_{1}\} = \{T_{\gamma_{0}}^{\pm 1} (\alpha_{0})\}$.
\item In addition, let $\{\delta_{1}, \ldots, \delta_{4}\}$ be curves on $S$. Suppose that at $X$, $\{\alpha_{i},\beta_{i}, \gamma_{i}, \epsilon_{i}\}_{i\in \Z}$ and $\{\delta_{1}, \ldots, \delta_{4}\}$ satisfy the identities of Lemma \ref{lem:lengthIDSet2} and each of $\{\alpha_{i}\}_{i \in \Z} \cup \{\gamma_{0}\}$, $\{\gamma_{i}\}_{i \in \Z} \cup \{\alpha_{0}\}$ contains no curves with the same length. 

Then $\alpha_{0} \cup \gamma_{0}$ forms a spine of an immersed subsurface $\psi : S' \rightarrow S$, where $S'$ is a generalized shirt. Moreover, $\{\alpha_{-1}, \alpha_{1}\} = \{ T_{\gamma}^{\pm1}(\alpha)\}$.

Further, we can label the peripheral curves of $S'$ by $\{\eta_{1}, \ldots, \eta_{4}\}$ such that:
\begin{itemize}
\item $\delta_{i}$, $\eta_{i}$ are both bounding (possibly distinct) punctures or $\delta_{i} = \eta_{i}$, and 
\item $\alpha_{0}$ separates $\{\eta_{1}, \eta_{3}\}$ from $\{\eta_{2}, \eta_{4}\}$ and $\gamma_{0}$ separates $\{\eta_{1}, \eta_{2}\}$ from $\{\eta_{3}, \eta_{4}\}$.
\end{itemize}
\end{enumerate}  
\end{proposition}

Suppose now that $S$ is a surface composed of at least two generalized pairs of pants. If $\eta_{1}$, $\eta_{2}$ are disjoint curves on $S$ and $\eta_{1}$ is essential, then we can perform the following procedure. We connect $\eta_{1}$ and $\eta_{2}$ with a simple segment $\tau$. Then $\eta_1$, $\eta_2$, and concatenation $\eta_{1}\tau\eta_2\tau^{-1}$ bound a pair of pants $P$ in $S$. Moreover, at least one of $\eta_{1}$ or $\eta_{2}$ is adjacent to yet another pair of pants $Q$, and $P \cup Q$ becomes an immersed generalized shirt. Here one of $\eta_{1}$, $\eta_{2}$ separates the shirt into $P$ and $Q$, and the other one becomes a boundary curve of $P \cup Q$. 
When both $\eta_1, \eta_2$ are boundary curves, we cal also find an immersed generalized shirt whose boundary component contains $\eta_1, \eta_2$. From this observation, we deduce the following lemma.

\begin{lem}\label{lem:disjointWitness}
	Let $S$ be a surface that is not a generalized pair of pants or a one-holed/punctured torus, $X \in \T(S) \setminus V$, and $\eta_{1}$, $\eta_{2}$ be essential or boundary curves on $X$. Then the following are equivalent: \begin{enumerate}
\item $\eta_{1}$ and $\eta_{2}$ are disjoint;
\item there exists essential or boundary curves $\{\alpha_{i},\beta_{i}, \gamma_{i}, \epsilon_{i}\}_{i\in \Z}$ on $X$ and curves $\{\delta_{1}, \ldots, \delta_{4}\}$ such that:
\begin{itemize}
\item each of  $\{\alpha_{i}\}_{i \in \Z} \cup \{\gamma_{0}\}$, $\{\gamma_{i}\}_{i \in \Z} \cup \{\alpha_{0}\}$ contains no curves with the same length;
\item $\{\eta_{1}, \eta_{2}\} =
\{\gamma_{0}, \delta_{1}\}$ or $\{\eta_{1}, \eta_{2}\} = 
\{\delta_1, \delta_2\}$; and
\item the identities of Lemma \ref{lem:lengthIDSet2} are satisfied.
\end{itemize}
\end{enumerate}
\end{lem}

This lemma describes how to detect the disjointness of two given curves on a surface by investigating the length identities of Lemma \ref{lem:lengthIDSet2}.
Note that in the above lemma, the implication $(1) \Rightarrow (2)$ does not require $X$ to be outside of $V$. Hence, from the above observation, we also have the following:

\begin{lem}\label{lem:disjointWitness2}
	Let $S$ be a surface that is not a generalized pair of pants or a one-holed/punctured torus, $X \in \T(S)$, and $\eta_{1}$, $\eta_{2}$ be essential or boundary curves on $X$. Suppose  that $\eta_1, \eta_2$ are disjoint. Then, there exists essential or boundary curves $\{\alpha_{i},\beta_{i}, \gamma_{i}, \epsilon_{i}\}_{i\in \Z}$ on $X$ and curves $\{\delta_{1}, \ldots, \delta_{4}\}$ such that:
	\begin{itemize}
		\item each of  $\{\alpha_{i}\}_{i \in \Z} \cup \{\gamma_{0}\}$, $\{\gamma_{i}\}_{i \in \Z} \cup \{\alpha_{0}\}$ contains no curves with the same length;
		\item $\{\eta_{1}, \eta_{2}\} =
		\{\gamma_{0}, \delta_{1}\}$ or $\{\eta_{1}, \eta_{2}\} = 
		\{\delta_1, \delta_2\}$; and
		\item the identities of Lemma \ref{lem:lengthIDSet2} are satisfied.
		\end{itemize}
	\end{lem}

\section{Surfaces with low complexity} \label{sec:lowcomplexity}

Before proving the main theorem in the general setting, we first deal with surfaces of low complexity. The case of generalized pair of pants is dealt with using the following lemma.

\begin{lem}\label{lem:pantsIsometry}
Let $X$, $X'$ be generalized pairs of pants with peripheral curves $\{\delta_{i}\}_{i=1, 2, 3}$ and $\{\delta_{i}'\}_{i=1. 2, 3}$, respectively.
\begin{enumerate}
\item If $l_{X}(\delta_{i}) = l_{X'}(\delta_{i}')$ for each $i$, then $X$ and $X'$ are isometric.
\item Suppose in addition that $l_{X}(\delta_{1}) \neq l_{X}(\delta_{2})$, and let $\eta$ ($\eta'$, resp.) be the unique simple geodesic segment perpendicular to $\delta_{1}$ and $\delta_{2}$ ($\delta_{1}'$ and $\delta_{2}'$, resp.). Then there exist exactly two isometries $\phi_{1}, \phi_{2} : X \rightarrow X'$ sending each $\delta_{i}$ to $\delta_{i}'$ and $\eta$ to $\eta'$. Here $\phi_{2}^{-1} \circ \phi_{1}$ becomes an orientation-reversing automorphism of $X$ fixing all boundaries setwise.
\end{enumerate}
\end{lem}

We now begin our discussion on one-holed/punctured tori and generalized shirts.

\begin{proposition}\label{prop:main1Torus}
Theorem \ref{thm:aeLengthSpectra} holds when $S$ is a one-holed/punctured torus.
\end{proposition}

\begin{proof}
Note that the assumption $\mathcal{L}(X) = \mathcal{L}(X')$ forces $\mathcal{L}(X')$ to be simple since $\mathcal{L}(X)$ is assumed to be so. That is, every element of $\mathcal{L}(X')$ has multiplicity one.

Let $\gamma'$ be an essential curve on $X'$. There exists another essential curve $\alpha'$ on $X'$ intersecting with $\gamma'$ once. We then set essential curves $\{\alpha'_{i}, \beta'_{i},\gamma'_{i}, \epsilon'_{i}\}_{i \in \Z}$ on $X'$ as the curves involved in Lemma \ref{lem:lengthIDSet1}: \begin{enumerate}
\item $\alpha'_{i} = T_{\gamma'}^{i}(\alpha')$, $\gamma'_{i} = T_{\alpha'}^{i}(\gamma')$, and
\item $\{T_{\alpha'_{i}}^{\pm 1} (\alpha'_{i-1})\} = \{\gamma', \beta'_{i}\}$ and $\{T_{\gamma'_{i}}^{\pm 1} (\gamma'_{i-1})\} = \{\alpha', \epsilon'_{i}\}$.
\end{enumerate}
Since $\{\alpha'_{i}, \beta'_{i}, \gamma_{i}', \epsilon_{i}'\}$ are essential, their lengths lie in $\mathcal{L}(X')$. Note also that $\{\alpha'_{i}\}_{i \in \Z}\cup \{\gamma'\}$ and $\{\beta'_{i}\}_{i \in \Z} \cup \{\alpha'\}$ are collections of distinct curves. Their lengths are distinct in $\mathcal{L}(X')$.

From the equality $\mathcal{L}(X) = \mathcal{L}(X')$ between simple length spectra, we can take essential or boundary curves $\{\alpha_{i}, \beta_{i}, \gamma_{i}, \epsilon_{i}\}_{i \in \Z}$ on $X$ such that \[
l_{X}(\alpha_{i}) = l_{X'}(\alpha'_{i}),\,\,l_{X}(\beta_{i}) = l_{X'}(\beta'_{i}),\,\,l_{X}(\gamma_{i}) = l_{X'}(\gamma'_{i}),\,\,l_{X}(\epsilon_{i}) = l_{X'}(\epsilon'_{i}).
\]
Note that $\{\alpha_{i}\}_{i \in \Z} \cup \{\gamma_{0}\}$ and $\{\gamma_{i}\}_{i \in \Z} \cup \{\alpha_{0}\}$ are comprised of distinct lengths. We then apply Proposition \ref{prop:mainProp} to deduce that $i(\alpha_{0}, \gamma_{0}) = 1$ and $\{\alpha_{-1}, \alpha_{1}\} = \{T_{\gamma_{0}}^{\pm 1}(\alpha_{0})\}$. Thus, $\alpha_{0} \cup \gamma_{0}$ serves as a spine of $X$. 
Moreover, $l_{X}(\alpha_{0})$, $l_{X}(\gamma_{0})$, $l_{X}(T_{\gamma_{0}}^{\pm 1}(\alpha_{0}))$ determine a unique isometry class of $X$ in the following way. First, three consecutive `twists' of $\alpha_{0}$ by $\gamma_{0}$ read the (unsigned) twist parameter at $\gamma_{0}$, or equivalently, the (unsigned) angle between the geodesics $\alpha_{0}$ and $\gamma_{0}$. Using $l_{X}(\alpha_{0})$, $l_{X}(\gamma_{0})$ and this angle, one can compute the length of (geodesic representative of) $\alpha_{0} \gamma_{0} \alpha_{0}^{-1} \gamma_{0}^{-1}$, the boundary curve of $X$. As a result, we obtain three boundary lengths of the pair of pants for $X$, and the twist for the gluing along $\gamma_{0}$. Since this information agrees with that of $X'$, we conclude that $X$ and $X'$ are isometric.

Let $\phi$ be the isometry from $X$ to $X'$. Then $f_2^{-1} \circ \phi \circ f_1$ becomes a (possibly orientation-reversing) homeomorphism on $S$ that sends $[f_2, X']$ to $[f_1, X]$ by pre-composition.
\end{proof}

We now move on to the case of generalized shirts. Note that a generalized shirt might not necessarily be embedded into another surface but only immersed. For instance, a generalized shirt is immersed in a closed surface of genus 2 and cannot be embedded. As such, we need a variant in the following format. 

\begin{proposition}\label{prop:mainShirt}
Suppose that $[f_1, Y] \in \T(S) \setminus V$ and $[f_2, Y'] \in \T(S)$ have the same simple length spectrum. Let $\psi' : X' \rightarrow Y'$ be an immersed subsurface of $Y'$ where $X'$ is a generalized shirt. Then there exists an immersed subsurface $\psi : X \rightarrow Y$ of $Y$ such that $X$ and $X'$ are isometric. In particular, Theorem \ref{thm:aeLengthSpectra} holds when $S$ is a generalized shirt.
\end{proposition}

\begin{proof}
Note that since each of the values in $\mathcal{L}(Y')$ have multiplicity 1, so do each of the values in  $\mathcal{L}(X')$.

Let $\gamma'$ be an essential curve on $X'$. We take another essential curve $\alpha'$ on $X'$ such that $i(\alpha', \gamma') = 2$ and $i_{alg}(\alpha', \gamma') = 0$. Then $\alpha' \cup \gamma'$ serves as a spine of $X'$, bounded by peripheral curves $\{\delta_{i}'\}_{i=1}^{4}$ labelled as in Figure \ref{fig:spineof4holed}. Let $P'$ ($Q'$, resp.) be the generalized pair of pants of $X'$ bounded by $\delta_{1}'$, $\delta_{2}'$ and $\gamma'$ ($\delta_{3}'$, $\delta_{4}'$ and $\gamma'$, resp.)

We then draw a simple geodesic segment $\kappa'_{P'}$ on $P'$, perpendicular to $\delta_1'$ and $\gamma'$. We given the orientation $\kappa'_{P'}$ so that it is from $\delta_1'$ to $\gamma'$. We use the inverse notation $\kappa'^{-1}_{P'}$ for the same segment with reversed orientation, and same for other oriented segments. Similarly we draw a simple segment $\kappa'_{Q'}$ on $Q'$ from $\delta'_{3}$ to $\gamma'$. Then there exists a unique segment $\xi'$ immersed along $\gamma'$ such that $\alpha'$ equals the concatenation $(\kappa_{P'}'^{-1} \delta_{1}' \kappa_{P'}' )\xi'( \kappa_{Q'}'^{-1} \delta'_{3} \kappa_{Q'}')\xi'^{-1}$. 

We now set essential curves $\{\alpha'_{i}, \beta'_{i},\gamma'_{i},\epsilon'_{i}\}_{i \in \Z}$ on $X'$ as the curves involved in Lemma \ref{lem:lengthIDSet2}: namely,  We label the peripheral curves of $X'$ by $\{\delta_{1}'$, $\ldots$, $\delta_{4}'\}$ in such a way that $\gamma'$ separates $\{\delta_{1}', \delta_{2}'\}$ from $\{\delta_{3}', \delta_{4}'\}$ and $\alpha'$ separates $\{\delta_{1}', \delta_{3}'\}$ from $\{\delta_{2}', \delta_{4}'\}$.

The corresponding essential or boundary curves $\{\alpha_{i}, \beta_{i}, \gamma_{i}, \epsilon_{i}\}_{i \in \Z}$ on $Y$ are taken by comparing the length spectra of $Y$ and $X'$. In other words, we require \[
l_{Y}(\alpha_{i}) = l_{X'}(\alpha_{i}'), \,\, l_{Y}(\beta_{i}) = l_{X'}(\beta'_{i}),\,\,l_{Y}(\gamma_{i}) = l_{X'}(\gamma'_{i}),\,\,l_{Y}(\epsilon_{i}) = l_{X'}(\epsilon'_{i}).
\] 
Note that each of the collections $\{\alpha_{i}\}_{i \in \Z} \cup \{\gamma_{0}\}$ and $\{\gamma_{i}\}_{i \in \Z} \cup \{\alpha_{0}\}$ is comprised of distinct lengths. We also take $\delta_{i}$'s appropriately: $\delta_{i}$ is taken to be  any puncture if the corresponding $\delta_{i}'$ is; otherwise $\delta_{i}$ is the essential or boundary curve on $Y$ having the same length with $\delta'_{i}$.

We then apply Proposition \ref{prop:mainProp} to deduce that $i(\alpha_{0}, \gamma_{0}) = 2$, $i_{alg}(\alpha_{0}, \gamma_{0}) = 0$ and $\{\alpha_{-1}, \alpha_{1}\} = \{T_{\gamma_{0}}^{\pm 1}(\alpha_{0})\}$. Thus, $\alpha_{0} \cup \gamma_{0}$ serves as a spine of a generalized shirt $X$, immersed in $Y$. Further, by the choice of $\delta_{i}$, we may assume that $\delta_{i}$ are indeed the peripheral curves of $X$, $\alpha_{0}$ separates $\{\delta_{1}, \delta_{3}\}$ from $\{\delta_{2}, \delta_{4}\}$, and $\gamma_{0}$ separates $\{\delta_{1}, \delta_{2}\}$ from $\{\delta_{3}, \delta_{4}\}$, as in $X'$ and in Figure \ref{fig:spineof4holed}. We then define generalized pairs of pants $P$, $Q$ and segments $\kappa_{P}$, $\kappa_{Q}$ and $\xi$ on $X$ analogously to $X'$.

From now on, we orient $X$, $X'$ such that $\alpha_{1}$ ($\alpha_{1}'$, resp.) becomes the positive twist $T_{\gamma}^{1}(\alpha)$ ($T_{\gamma'}^{1}(\alpha')$, resp.). By Lemma \ref{lem:pantsIsometry}, there exist unique orientation-preserving isometries $\phi_{P} : P \rightarrow P'$ and $\phi_{Q} : Q \rightarrow Q'$ sending $\kappa_{P}$ to $\kappa_{P'}$ and $\kappa_{Q}$ to $\kappa_{Q'}$. It remains to show that $\phi_{P}$ and $\phi_{Q}$ agree on $\gamma$.

In short, the twist at $\gamma$ ($\gamma'$, resp.) is read off using $l_{X}(\alpha)$ and $l_{X}(T_{\gamma}^{\pm 1}(\alpha))$ ($l_{X'}(\alpha')$ and $l_{X'}(T_{\gamma'}^{\pm 1} (\alpha'))$, resp.). Indeed, the signed length of $\xi$ is determined by $(l_{X}(T_{\gamma}^{1}(\alpha)), l_{X}(T_{\gamma}^{-1}(\alpha))) = (l_{X}(\alpha_{1}), l_{X}(\alpha_{-1}))$ and the boundary lengths $(l_{X}(\delta_{1}), \ldots, l_{X}(\delta_{4}))$. (See Proposition 3.3.11 and 3.3.12 of~\cite{buser1992compact} for an explicit calculation.) Similarly, the signed length of $\xi'$ is determined by lengths $(l_{X'}(\alpha_{1}'), l_{X'}(\alpha_{-1}'), l_{X'}(\delta_{1}'), \ldots, l_{X'}(\delta_{4}'))$. Since the lengths involved are identical, we conclude that the signed lengths of $\xi$ and $\xi'$ are also the same, and hence $\phi_{P}$ and $\phi_{Q}$ agree on $\gamma$. Thus $X$ is isometric to $X'$.
\end{proof}

As in Proposition \ref{prop:main1Torus}, the twist $\tau_{X'}(\gamma')$ of $X'$ at $\gamma'$ cannot be a multiple of $\pi$. This is because the lengths of $\{T_{\gamma'}^{i}(\alpha')\}_{i=-2, 0, 2}$ differ. If moreover, say, $\delta_{1}'$ and $\delta_{2}'$ have same lengths (e.g. they are punctures), then multiples of $\pi/2$ are also forbidden for $\tau_{X'}(\gamma')$. In any case, there exists only one isometry between $X$ and $X'$.

\begin{proposition}\label{prop:main2Torus}
Theorem \ref{thm:aeLengthSpectra} holds when $S$ is of type $S_{1, p, b}$ for $p+b=2$.
\end{proposition}

\begin{proof}
We first take a curve $\gamma'$ separating $X'$ into a one-holed torus and a generalized pair of pants. Inside that one-holed torus, there exists a curve $\delta'$ longer than $\gamma'$. Indeed, we may pick any pair of once-intersecting curves inside the given one-holed torus and twist one along the other sufficiently many times. Now if we take a curve $\alpha'$ on $X'$ such that $i(\delta', \alpha') = 0$, $i(\gamma', \alpha') = 2$, and $i_{alg}(\gamma', \alpha' ) =0$, then $\alpha' \cup \gamma'$ becomes a spine of an immersed subsurface $\psi_{0}' : X'_{0} \rightarrow X'$, where $X'_{0}$ is a generalized shirt. Let us label the boundaries of $X'_{0}$ by $\delta_{i}'$ as in Lemma \ref{lem:lengthIDSet2} so that $\delta_{1}' = \delta_{2}' = \delta'$ are the same when seen as curves in $X'$.

Let $\kappa_{i}'$ be the simple geodesic segment perpendicular to $\gamma'$ and $\delta_{i}'$ for $i=1, \ldots, 4$, oriented toward $\gamma'$. Further, let $\xi_{i}'$ be the arc immersed along $\gamma'$ such that \[
\alpha' =  (\kappa_{1}'^{-1} \delta_{1}' \kappa_{1}' )\xi_{1}' ( \kappa_{3}'^{-1} \delta_{3}' \kappa_{3}')\xi_{1}'^{-1} =(\kappa_{2}'^{-1} \delta_{2}' \kappa_{2}' )\xi_{2}' ( \kappa_{4}'^{-1} \delta_{4}' \kappa_{4}')\xi_{2}'^{-1}.
\]
Finally, we set $\zeta_{i}'$ to be a shortest simple geodesic segment perpendicular to $\delta_{i}'$ and $\delta_{i+2}'$ for $i=1, 2$. See Figure \ref{fig:twoholedtorus}.

\begin{figure}[ht]
\begin{tikzpicture}[scale=0.8]

\draw  (4.8, 1.5) .. controls (4.1, 1.5) and (4, 1.2) .. (3.3, 1.2) .. controls (2.6, 1.2) and (2.5, 1.5) .. (1.8, 1.5) .. controls (0.9, 1.5) and (-3, 0.9) .. (-3, 0) .. controls (-3, -0.9) and (0.9, -1.5) .. (1.8, -1.5) .. controls (2.5, -1.5) and (2.6, -1.2) .. (3.3, -1.2) .. controls (4, -1.2) and (4.1, -1.5) .. (4.8, -1.5);

\draw (1.8, 0) circle (0.6 and 0.3);
\draw (4.8, 0.3) arc (90:270:0.6 and 0.3);

\draw[thick] (1.2, 0) arc (0:-180:2.1 and 0.17);
\draw[thick, dashed] (1.2, 0) arc (0:180:2.1 and 0.17);

\draw[thick]  (4.2, 0) arc (0:-180:0.9 and 0.17);
\draw[thick, dashed] (4.2, 0) arc (0:180:0.9 and 0.17);

\draw[thick]  (4.8, 0.3) arc (-90:90:0.17 and 0.6);
\draw[thick, dashed] (4.8, 0.3) arc (270:90:0.17 and 0.6);

\draw[thick]  (4.8, -0.3) arc (90:-90:0.17 and 0.6);
\draw[thick, dashed] (4.8, -0.3) arc (90:270:0.17 and 0.6);

\draw[thick] (3.3, 1.2) arc (90:-90:0.17 and 1.2);
\draw[thick, dashed] (3.3, 1.2) arc (90:270:0.17 and 1.2);

\draw[thick] (1.8, 0) circle (1 and 0.5);

\draw (-0.7, -0.1) node[below] {$\delta_{1}' = \delta_{2}'$};
\draw (3.3, 1.5) node {$\gamma'$};
\draw (4.8, 1.8) node {$\delta_{3}'$};
\draw (4.8, -1.8) node {$\delta_{4}'$};
\fill[black!0] (1.3, 0.05) -- (1.65, 0.05) -- (1.65, 0.32) -- (1.3, 0.32);
\draw (1.5, 0.22) node {$\eta'$};
\draw (1.75, -1.1) node {$\sigma'$};
\draw (4.55, 0.05) node {$\alpha'$};

\draw[thick] (-3, -2.15) -- (-3, -2.3) -- (3.25, -2.3) -- (3.25, -2.15);
\draw[thick] (3.35, -2.15) -- (3.35, -2.3) -- (4.75, -2.3) -- (4.75, -2.15);

\draw (1.625, -2.6) node {$X_{1}'$};

\begin{scope}[shift={(-3, 0)}]
\draw[dashed]  (7.8 - 0.6*0.707106781186548, 0.3*0.707106781186548) .. controls  (7.8 - 0.6*0.707106781186548 + 0.6, 0.3*0.707106781186548 +0.3) and  (6.3, 0.9) .. (4.8, 0.9) .. controls (3, 0.9) and (3, -0.7) .. (4.8, -0.7) ..controls (6.3, -0.7) and (7.8 - 0.6*0.707106781186548 - 0.6, -0.3*0.707106781186548 + 0.7*0.3) .. (7.8 - 0.6*0.707106781186548, -0.3*0.707106781186548);
\draw (7.8 - 0.6*0.707106781186548, 0.3*0.707106781186548).. controls (7.8 - 0.6*0.707106781186548 - 0.6, +0.3*0.707106781186548 - 0.7*0.3) and (6.3, 0.7) .. (4.8, 0.7) .. controls (3, 0.7) and (3, -0.9) .. (4.8, -0.9) ..controls (6.3, -0.9) and (7.8 - 0.6*0.707106781186548 + 0.6, -0.3*0.707106781186548 - 0.3) .. (7.8 - 0.6*0.707106781186548, -0.3*0.707106781186548);
\end{scope}

\begin{scope}[shift={(8,  0)}]
\draw (0, 0.5) arc (270:330:1.4);
\draw (0, 0.5) arc (270:210:1.4);
\draw (0, -0.5) arc (90:30:1.4);
\draw (0, -0.5) arc (90:150:1.4);
\draw (2.0784609691, 0.7) arc (150:210:1.4);
\draw (-2.0784609691, 0.7) arc (30:-30:1.4);

\draw[shift={(2.0784609691, 0.7)}, rotate=60] (0, 0) arc (-90:90:0.2 and 0.5);
\draw[dashed, shift={(2.0784609691, 0.7)}, rotate=60] (0, 0) arc (270:90:0.2 and 0.5);
\draw[shift={(2.0784609691, -0.7)}, rotate=-60] (0, 0) arc (90:-90:0.2 and 0.5);
\draw[dashed, shift={(2.0784609691, -0.7)}, rotate=-60] (0, 0) arc (90:270:0.2 and 0.5);

\draw[shift={(-2.0784609691, 0.7)}, rotate=120] (0, 0) arc (90:-90:0.2 and 0.5);
\draw[dashed, shift={(-2.0784609691, 0.7)}, rotate=120] (0, 0) arc (90:270:0.2 and 0.5);
\draw[shift={(-2.0784609691, -0.7)}, rotate=-120] (0, 0) arc (-90:90:0.2 and 0.5);
\draw[dashed, shift={(-2.0784609691, -0.7)}, rotate=-120] (0, 0) arc (270:90:0.2 and 0.5);


\draw[thick] (0, 0.15) arc (-90:-22:1.7);

\draw (0, 0.35) arc (-90:-145:1.55);

\fill[black!0] (0, 0) circle (0.2 and 0.5);

\draw (0, -0.35) arc (90:95:1.55) arc  (95:34.5:1.55);
\draw[thick] (0, -0.15) arc (90:84:1.6) arc (85.5:158:1.7);

\draw[thick, red] (0, -0.25) arc (90:156.7:1.6) arc (157:23.3:1.6);
\draw[thick, blue] (0, 0.25) arc (-89:-156.7:1.6) arc (-157:-23.3:1.6);

\draw (0, -0.5) arc (-90:90:0.2 and 0.5);
\draw[dashed] (0, -0.5) arc (270:90: 0.2 and 0.5);

\draw (1.887, 0) arc (200:172:1.4);
\draw[thick] (1.887, 0) arc (160:202.5:1.4);

\begin{scope}[rotate=180]
	\draw (1.887, 0) arc (200:172:1.4);
\draw[thick] (1.887, 0) arc (160:202.5:1.4);
	
\end{scope}


\draw  (1.15, 0.3) node {$\kappa_{3}'$};
\draw (0.67, -0.95) node {$\kappa_{4}'$};

\draw (-0.65, 0.95) node {$\kappa_{1}'$};
\draw  (-1.15, -0.25)  node {$\kappa_{2}'$};
\draw (-2, 1.3) node {$\delta_{1}'$};
\draw (-2, -1.3) node {$\delta_{2}'$};
\draw (2, 1.3) node {$\delta_{3}'$};
\draw (2, -1.3) node {$\delta_{4}'$};

\draw[blue] (0.66, 0.95) node {$\zeta_{1}'$};
\draw[red] (-0.67, -0.95) node {$\zeta_{2}'$};
\end{scope}

\end{tikzpicture}

\caption{Curves on the surface $S_{1, p, b}$ with $p+b=2$} \label{fig:twoholedtorus}
\end{figure}
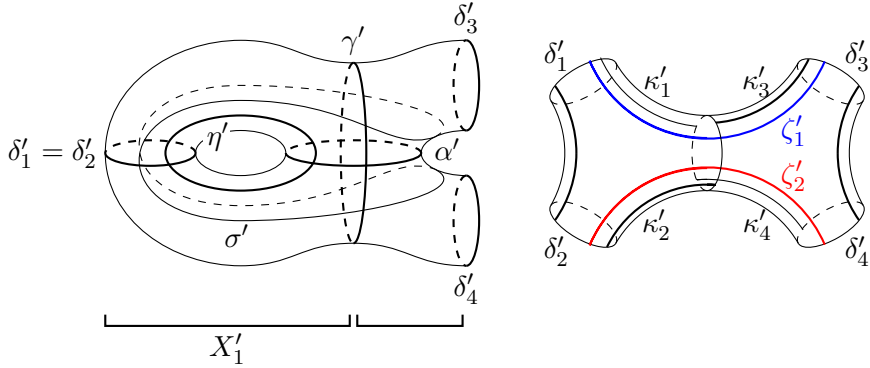

We now set curves $\{\alpha'_{i}, \beta'_{i},\gamma'_{i}, \epsilon'_{i}\}_{i \in \Z}$ on $X'$ as in Lemma \ref{lem:lengthIDSet2}. The corresponding essential or boundary curves $\{\alpha_{i}, \beta_{i}, \gamma_{i}, \epsilon_{i}\}_{i \in \Z}$ on $X$ are taken by comparing the lengths, i.e., requiring \[
l_{X}(\alpha_{i}) = l_{X'}(\alpha_{i}'), \,\, l_{X}(\beta_{i}) = l_{X'}(\beta'_{i}),\,\,l_{X}(\gamma_{i}) = l_{X'}(\gamma'_{i}),\,\,l_{X}(\epsilon_{i}) = l_{X'}(\epsilon'_{i}).
\] 
Note that $\{\alpha_{i}\}_{i \in \Z} \cup \{\gamma_{0}\}$ and $\{\gamma_{i}\}_{i \in \Z} \cup \{\alpha_{0}\}$ are comprised of distinct lengths. We also take $\delta_{i}$ appropriately: $\delta_{i}$ is taken as any puncture if the corresponding $\delta_{i}'$ is; otherwise $\delta_{i}$ is the essential or boundary curve on $X$ having the same length as $\delta'_{i}$. Here $\delta_{1} = \delta_{2}$ since $\delta_{1}'$ and $\delta_{2}'$ are both essential curves and have the same length. From now on, we fix the orientation of $X$ so that $\alpha_{1} = T_{\gamma_{0}}^{1}(\alpha_{0})$ and similarly for $X'$.

We first cut $X$ along $\delta_{1}$ to obtain an immersed subsurface $\psi_{0} : X_{0} \rightarrow X$. Proposition \ref{prop:mainProp}(3) tells us that $X_{0}$ is a generalized shirt. Thus, we can also define $\kappa_{i}$, $\xi_{i}$, $\zeta_{i}$ on $X_{0}$ analogously. Now, Proposition \ref{prop:mainShirt} gives an isometry $\phi_{0} : X_{0} \rightarrow X'_{0}$ sending each $\delta_{i}$ to $\delta_{i}'$. In particular, $\phi_{0}$ becomes orientation-preserving due to our choice of orientations. Moreover, $\kappa_{i}$, $\xi_{i}$, $\zeta_{i}$ are sent to the corresponding $\kappa_{i}'$, $\xi_{i}'$, $\zeta_{i}'$ with orientations preserved.

We further take $\eta'$, $\sigma'$ on $X'$ such that $i(\delta', \eta') = i(\eta', \alpha') = 1$, $i(\eta', \gamma') = 0$, and $i(\alpha', \sigma') = 0$, $i(\delta', \sigma') = 2$ and $i_{alg}(\delta', \sigma') = 0$. Then $\delta' \cup \eta'$ ($\delta' \cup \sigma'$, resp.) becomes a spine of an immersed subsurface $\psi_{1}' : X'_{1} \rightarrow X'$ ($\psi_{2}' : X'_{2} \rightarrow X'$, resp.) where $X_{1}'$ ($X'_{2}$, resp.) is a one-holed torus (generalized shirt, resp.). See Figure \ref{fig:twoholedtorus}.

Similarly, one can copy $\eta'$, $\sigma'$ (and other necessary curves) to $X$ using the length spectra. Then $X$ cut along $\gamma$ becomes an immersed subsurface $\psi_{1} : X_{1} \rightarrow X$ where $X_{1}$ is a one-holed torus, and $X$ cut along $\alpha$ becomes an immersed subsurface $\psi_{2} : X_{2} \rightarrow X$ where $X_{2}$ is a generalized shirt, and these are immersed along boundaries. Furthermore, Proposition \ref{prop:main1Torus} gives an isometry $\phi_{1} : X_{1} \rightarrow X'_{1}$, sending $\kappa_{1}$ to $\kappa_{1}'$ and $\kappa_{2}$ to $\kappa_{2}'$. Proposition \ref{prop:mainShirt} also gives an isometry $\phi_{2} : X_{2} \rightarrow X'_{2}$, sending $\zeta_{1}$ to $\zeta_{1}'$ and $\zeta_{2}$ to $\zeta_{2}'$.

At this moment, $\phi_{1}$ may or may not agree with $\phi_{0}$ on $X_{0} \cap X_{1}$, depending on whether $\phi_{1}$ is orientation-preserving or not. Once $\phi_{1}$ is shown to be orientation-preserving, the gluing of $\phi_{0}$ and $\phi_{1}$ becomes an isometry between $X$ and $X'$, completing the proof. Suppose to the contrary that $\phi_{1}$ is orientation-reversing. For clearer explanation, we from now on flip the orientation of $X'$ to make $\phi_1$ orientation-preserving, while we have that $\phi_{0}$ is orientation-reversing. We then show that the (unsigned) distance between $\zeta_{1}$ and $\zeta_{2}$ along $\delta_{1}$ differs to the analogous one on $X'$. This will then contradict the fact that $\phi_{2}$ is an isometry, which must preserve the unsigned twist of $X_{2}$ at $\delta_{1}$.

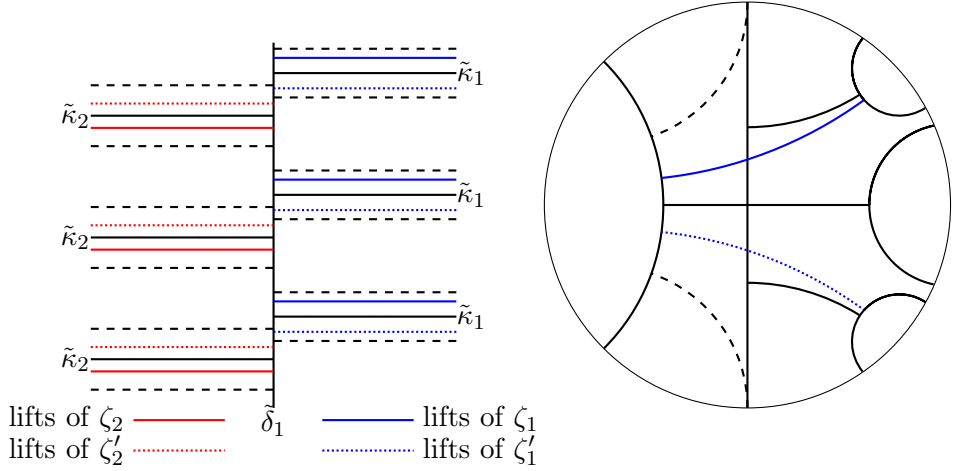
\begin{figure}[ht]
\begin{tikzpicture}[scale=0.8]
\begin{scope}[shift={(-1, -3)}, scale=0.9]
\draw[thick] (0, 0) -- (0, 6);
\foreach \i in {0, 1, 2}{
\draw[thick] (-3, 0.8 + 2*\i) -- (0, 0.8+ 2*\i);
\draw (-3.25, 0.8+2*\i) node {$\tilde{\kappa}_{2}$}; 
\draw[thick] (3, 1.5 + 2*\i) -- (0, 1.5+2*\i);
\draw (3.25, 1.5+2*\i) node {$\tilde{\kappa}_{1}$};
\draw[thick, dashed] (-3, 0.3 + 2*\i) -- (0, 0.3+ 2*\i);
\draw[thick, dashed] (-3, 1.3 + 2*\i) -- (0, 1.3+ 2*\i);
\draw[thick, dashed] (3, 1.1+2*\i) -- (0, 1.1+2*\i);
\draw[thick, dashed] (3, 1.9+2*\i) -- (0, 1.9+2*\i);
\draw[thick, red, densely dotted] (-3, 1+2*\i)  -- (0, 1+2*\i);
\draw[thick, red] (-3, 0.6+2*\i)  -- (0, 0.6+2*\i);
\draw[thick, blue] (3, 1.75 + 2*\i) -- (0, 1.75+2*\i);
\draw[thick, blue, densely dotted] (3, 1.25+2*\i) -- (0, 1.25+2*\i);
}
\draw (0, -0.2) node {$\tilde{\delta}_{1}$};

\draw[thick, red] (-2.3, -0.2) -- (-0.8, -0.2);
\draw[thick, red, densely dotted] (-2.3, -0.7) -- (-0.8, -0.7);

\draw (-3.4, -0.18) node  {lifts of $\zeta_{2}$};
\draw (-3.4, -0.7) node  {lifts of $\zeta_{2}'$};

\draw[thick, blue] (0.8, -0.2) -- (2.3, -0.2);
\draw[thick, blue, densely dotted] (0.8, -0.7) -- (2.3, -0.7);

\draw (3.4, -0.18) node {lifts of $\zeta_{1}$};
\draw (3.4, -0.7) node {lifts of $\zeta_{1}'$};

\end{scope}

\begin{scope}[shift={(6, 0)}]
\draw[thick] (0, -3) -- (0, 3);
\draw[thick] (-3*0.707106781186548,3*0.707106781186548) arc (45:-45:3);
\draw[thick] (1.8, 0) arc (180:102:1.2) arc (102:258:1.2);
\draw[thick, rotate=42] (2.32, 0) arc (180:99:0.7) arc (99:261:0.7);
\draw[thick] (-3*0.414213562373095, 0) -- (1.8, 0);
\draw[thick, rotate=-42] (2.32, 0) arc (180:99:0.7) arc (99:261:0.7);
\draw[thick] (0, 1.15) arc (-90:-57.8:3.1);
\draw[thick] (0, -1.15) arc (90:57.8:3.1);
\draw[blue, thick, rotate=6] (-1.23, 0.53) arc (-90:-60:6.2);
\draw[blue, densely dotted, thick, rotate=-6] (-1.23, -0.53) arc (90:60:6.2);
\draw[thick, dashed] (0, 3) arc (0:-71:2.1);
\draw[thick, dashed] (0, -3) arc (0:71:2.1);
\draw (0, 0) circle (3);
\end{scope}

\end{tikzpicture}
\caption{Description on the hyperbolic plane}
\label{fig:4holedHolonomy}
\end{figure}

We parametrize $\delta_{1}$ by arc length $\lambda$ so that $\kappa_{1}$ is located on the right side of $\delta_{1}$ while $\lambda$ increases, as in Figure \ref{fig:4holedHolonomy}. On $X$, we denote the signed displacement from  $\kappa_{1}$ ($\kappa_{2}$, resp.) to $\zeta_{1}$ ($\zeta_{2}$, resp.) along $\delta_{1}$ by $d_{1}$ ($d_{2}$, resp.). Further, we denote the signed displacement from $\kappa_{1}$ to $\kappa_{2}$ along $\delta_{1}$ by $D$. Here, the signed displacement is taken inside the range $[-l_{X}(\delta_{1})/2, l_{X}(\delta_{1})/2]$. Similarly, we define the displacements $d_{1}'$, $d_{2}'$ and $D'$  for curves and segments on $X'$. 

From the assumption, the twists at $\gamma$, $\gamma'$ are nonzero and opposite: this forces $d_{1}=-d_{1}' \neq 0$ and $d_{2} = -d_{2}' \neq 0$. Furthermore, $d_{1}$, $d_{2}$ have opposite signs since intersecting points of $\kappa_1$ and $\zeta_1$ with $\delta_1$ do not separate the intersection of $\kappa_2$ with $\delta_1$ from the one of $\zeta_2$ with $\delta_1$ on $\delta_1$. Finally, note that a lift $\tilde{\zeta}_{1}$ must cross the corresponding lift $\tilde{\gamma}$. This forces $\tilde{\zeta}_{1}$ to be sandwiched between $\tilde{\kappa}_{1}$ and a geodesic from $\tilde{\delta}_{1}$ `spiraling toward' $\gamma$ (the black dashed line in the right of Figure \ref{fig:4holedHolonomy}). If we denote the displacement between $\tilde \kappa_1$ and the spiraling geodesic by $L$, then we observe that $2L < l_{X}(\delta_{1})/2$, as depicted in Figure \ref{fig:hornedHexagon}.

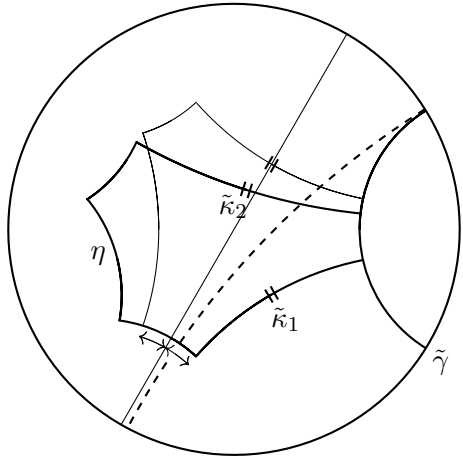
\begin{figure}[h]
\begin{tikzpicture}
\begin{scope}[scale=1]

\begin{scope}[rotate=30]
\draw[thick] (0, 0) circle (3);
\draw[red, thick, rotate=-60 ] (0, -1.67) arc (90:77:1.859610778443114) arc (77:112:1.85961077844311);
\draw[red] (-2.2, -0.5) node {$\tilde \delta_1$};

\draw[<->, rotate=-60] (0, -1.82) arc (90:77:1.709610778443114);
\draw[<->, rotate=-60] (0, -1.82) arc (90:103:1.709610778443114);

\draw[thick, rotate=0] (0, -1) arc (90:108.5:4) arc (108.5:71.5:4);
\begin{scope}[shift={(0, -5)}]
\draw[thick,rotate=0.6] (0, 4.1) -- (0, 3.9);
\draw[thick, rotate=-0.6] (0, 4.1) -- (0, 3.9);
\end{scope}

\draw[ rotate=-180 ] (0, -1.67) arc (90:77:1.859610778443114) arc (77:103:1.85961077844311);
\draw[thick, rotate=60 ] (0, -1.67) arc (90:32.1:1.859610778443114) arc (32.1:147.9:1.85961077844311);

\draw[rotate=120] (0, -1) arc (90:108.5:4) arc (108.5:71.5:4);
\begin{scope}[rotate=120, shift={(0, -5)}]
\draw[thick,rotate=0.6] (0, 4.1) -- (0, 3.9);
\draw[thick, rotate=-0.6] (0, 4.1) -- (0, 3.9);
\end{scope}

\draw[rotate=240] (0, -1) arc (90:108.5:4) arc (108.5:71.5:4);

\draw[thick, rotate=132.5] (0, -0.55) arc(90:78:7.631818181818182) arc (78:101.7:7.631818181818182);
\begin{scope}[rotate=132.5, shift={(0, -8.181818181818182)}]
\draw[thick,rotate=0.35] (0, 7.631818181818182 + 0.1) -- (0, 7.631818181818182 - 0.1);
\draw[thick, rotate=-0.35] (0, 7.631818181818182 + 0.1) -- (0, 7.631818181818182 - 0.1);
\end{scope}

\draw[thick, rotate=207] (0, -1.7) arc (90:65.8:1.797058823529412) arc (65.8:99:1.797058823529412);
\draw[thick, rotate=254] (0, -1.6) arc (90:66:2.0125) arc (66:115.7:2.0125);

\draw[rotate=-60] (0, -3) -- (0, 3);

\draw (0, -1.38) node {$\tilde{\kappa}_{1}$};
\draw (-1.74, 0.6) node {$\tilde{\eta}$};
\draw (0.15, 0.27) node {$\tilde{\kappa}_{2}$};
\draw (1.5, -2.9) node {$\tilde{\gamma}$};

\draw (3, 1.8) node {$\mathcal{G}$};

\end{scope}

\draw[thick, dashed] (3*0.846193166127564, 3*0.53287627607073) arc (122.2:152:11.343461538461538);
\fill[black!0, rotate=-30] (0.03, -1.75) -- (0.03, -1.92) -- (0.15, -1.92) -- (0.15, -1.75);

\draw[<->, rotate=-30] (0, -1.82) arc (90:77:1.709610778443114);
\draw[<->, rotate=-30] (0, -1.82) arc (90:103:1.709610778443114);

\end{scope}
\end{tikzpicture}
\caption{A hexagon bounded by $\tilde{\delta}_{1}$, $\tilde{\gamma}$, $\tilde{\kappa}_{1}$ and other geodesics.}
\label{fig:hornedHexagon}
\end{figure}

Indeed, the distance between $\tilde{\eta}$ and $\tilde{\kappa}_1$ and the distance between $\tilde{\eta}$ and $\tilde{\kappa}_2$ in Figure \ref{fig:hornedHexagon} are equal to $l_{X}(\delta_1)/2$ and the distance between $\tilde{\kappa}_1$ and $\tilde{\kappa}_2$ is $l_X(\gamma)/2$. We denote by $\mathcal{G}$ the bi-infinite geodesic orthogonal to $\tilde{\delta}_1$ and sharing an endpoint with the dashed line in Figure \ref{fig:hornedHexagon}. Then the distance between $\mathcal{G}$ and $\tilde{\kappa}_1$ gives an upper bound of $L$. 

 If $\mathcal{G}$ were bisecting $\tilde{\delta}_1$, then the distance between $\tilde{\kappa}_1$ and $\tilde{\kappa}_2$ is same as the distance between $\tilde{\eta}$ and $\tilde{\kappa}_2$, which is also same as the distance between $\tilde{\kappa}_1$ and $\tilde{\eta}$. Based on this observation, we explain step by step using Figure \ref{fig:hornedHexagon} why $2L$ is smaller than the length of $\tilde \delta_1$, from which  $2L < l_X(\delta_1)/2$ follows.  Recall that we have fixed the lift $\tilde{\gamma}$ of $\gamma$ and the lift $\tilde{\kappa}_{1}$  of $\kappa_1$ (of the same length) perpendicular to $\tilde{\gamma}$. Draw a line $\Delta$ perpendicular to $\tilde{\kappa}_{1}$  that is not $\tilde{\gamma}$ (in Figure \ref{fig:hornedHexagon}, it is the geodesic line containing $\tilde{\delta}_{1}$). Now, we will draw the hexagon that is half of the left half of $X_{2}’$ as follows. 

For a moment, we consider the length of $\tilde \delta_1$ as a variable: let $t > 0$. We will vary the endpoint of $\tilde{\delta}_{1}$ that is other than $\Delta \cap \tilde{\kappa}_{1}$, setting that $\tilde{\delta}_{1}$ has length $t$. Then, we decide $\tilde{\eta}$ perpendicular to $\tilde{\delta}_{1}$, a segment $\tilde{\delta}_{2}$ perpendicular to $\tilde{\eta}$ with the same length as $\tilde{\delta}_{1}$, and $\tilde{\kappa}_{2}$ perpendicular to both $\tilde{\delta}_{2}$ and $\tilde{\gamma}$. This can be done as follows: when the length of $\tilde{\delta}_{1}$ is determined, the line $N$ (the geodesic line containing $\tilde{\eta}$) is determined. Then we draw a simultaneous perpendicular to $N$ and $\tilde{\gamma}$, and reflect $\tilde{\delta}_{1}$ and $\tilde{\kappa}_{1}$ with respect to it and get $\tilde \delta_2$ and $\tilde \kappa_2$ respectively. We get the desired hexagon, and as a direct consequence of the construction, $\tilde{\kappa}_{1}$ and $\tilde{\kappa}_{2}$ have the same length, and $\tilde \delta_1$ and $\tilde \delta_2$ do so.

Now recall that $L$ is the distance between the dashed line and $\tilde{\kappa}_{1}$, which  is fixed regardless of $t$. When $t$ increases, the distance between $\tilde \kappa_1$ and $\tilde \kappa_2$ along $\tilde \gamma$ decreases. This implies that the
length of $\gamma$ decreases and the length of $\delta_1(=2t)$ increases. Now, recall that $\mathcal{G}$ is the perpendicular to $\Delta$ witht the endpoint same as the endpoint of the dashed line which is not shared with $\tilde \gamma$, as declared. When $t = 2 d(\mathcal{G}, \tilde{\kappa}_{1})$, we see that the hexagon is symmetric with respect to $\mathcal{G}$. Then, $\tilde{\kappa}_{1}, \tilde{\kappa}_{2}$ and $\tilde{\eta}$ all have the same length, and the hexagon has 120 degree symmetry. In particular, the length of $\delta$ equals the length of $\gamma$. At this point, still $t > 2 L$. Hence, at the situation $t \leq 2L$, we have that the length of $\delta$ is smaller than the length of $\gamma$, which contradicts the assumption. Therefore, we must have $2L < l_X(\delta_1)/2$.

Now we calculate the unsigned displacements between $\zeta_{1}$ and $\zeta_{2}$ ($\zeta_{1}'$ and $\zeta_{2}'$, respectively). The former is $|D + d_{1} + d_{2} + nl_{X}(\delta_{1})| $ for some $n \in \Z$ and the latter is $|D + d_{1}' + d_{2}' + n'l_{X}(\delta_{1})| = |D-d_{1}-d_{2} + n' l_{X}(\delta_{1})|$ for some $n' \in \Z$. If some of them are equal, then either \[
2D = -(n+n') l_{X}(\delta_{1})  \quad \mbox{or} \quad 2(d_{1}+ d_{2}) = (n' - n) l_{X}(\delta_{1}).
\]

First, note that if $2|D|$ were either $0$ or $l_X(\delta_1)$, the two surfaces $X$ and $X'$ would be related to each other by combinations of orientation-reversing map and Dehn twists along $\delta_1$. Hence, we may assume that $2|D|$ is neither $0$ nor $l_X(\delta_1)$. Recall that $|D| \le l_X(\delta_1)/2$. Then 
the former case is excluded since $2|D|$ cannot be a multiple of $ l_{X}(\delta_{1})$. For the latter case, note first that $d_{1}$, $d_{2}$ are nonzero values having the same sign: their sum cannot vanish. However, since $|d_{1}|$, $|d_{2}| <  l_{X}(\delta_{1})/4$, $2(d_{1}+d_{2})$ cannot become other multiples of $ l_{X}(\delta_{1})$.
\end{proof}

The proof equally applies to the case of genus 2 surface. In both cases, only one isometry is allowed between $X$ and $X'$.

\begin{proposition}\label{prop:main5sphere}
Suppose that $[f_1, Y] \in \T(S) \setminus V$ and $[f_2, Y'] \in \T(S)$ have the same simple length spectrum. Let $\psi' : X' \rightarrow Y'$ be an immersed subsurface of $Y'$ where $X'$ is a surface of type $S_{0, p, b}$ where $p + b = 5$. Then there exists an immersed subsurface $\psi : X \rightarrow Y$ of $Y$ such that $X$ and $X'$ are isometric. In particular, Theorem \ref{thm:aeLengthSpectra} holds when $S$ is of type $S_{0, p, b}$ for $p+ b=5$.
\end{proposition}

\begin{proof}
Let us take curves $\gamma_{1}'$, $\gamma_{2}'$, $\alpha_{0}'$, $\alpha_{1}'$, $\alpha_{2}'$ on $X'$ as in Figure \ref{fig:5holes}, and label the peripheral curves as $\delta_{i}'$. Then we obtain two immersed subsurface $\psi_{i}' : X_{i}' \rightarrow X'$ for $i=1, 2$, where $X_{i}'$ is the generalized shirt with spine $\alpha_{i}' \cup \gamma_{i}'$. In addition, cutting $X'$ along $\alpha_{2}'$ also gives another immersed subsurface $\psi_{0}' : X_{0}' \rightarrow X'$. 

We also draw simple geodesic segments $\kappa_{1}'$ (resp. $\kappa_2'$) perpendicular to $\delta_{2}'$ (resp. $\delta_4'$) and $\gamma_{1}'$ (resp. $\gamma_2'$), $\eta'$ perpendicular to $\gamma_{1}'$ and $\gamma_{2}'$, and $\zeta'$ perpendicular to $\delta_{4}'$ and $\gamma_{1}'$.

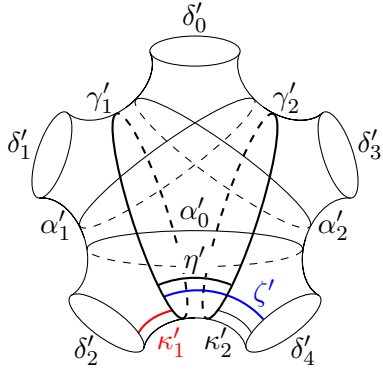
\begin{figure}[ht]
\begin{tikzpicture}

\draw[thick, blue] (-0.416, -1.27) arc(113: 40.4: 1.15 and 1.1);
\draw[red, thick] (-0.751, -1.75) arc (160:110: 0.744 and 0.5);

\foreach \i in {0, ..., 4}{
\draw[rotate=72*\i] (0, 2) circle (0.6 and 0.2);
\draw[rotate=36+72*\i] (0, 1.55) arc (-90:-8:0.688 and 0.47) arc (-8:-172:0.688 and 0.47);
}
\foreach \i in {0,  3}{
\draw[thick, dashed, rotate=-18 + 72*\i] (0.65, 1.425) arc (85:360-85:0.24 and 1.432);
\draw[thick, rotate=-18+72*\i] (0.65, 1.425) arc (85:-85:0.24 and 1.432);
}

\foreach \i in {1,  2}{
\draw[dashed, rotate=-18 + 72*\i] (0.65, 1.425) arc (85:360-85:0.24 and 1.432);
\draw[ rotate=-18+72*\i] (0.65, 1.425) arc (85:-85:0.24 and 1.432);
}

\draw[rotate=-18-72] (0.65, 1.425) arc (85:360-85:0.24 and 1.432);
\draw[dashed, rotate=-18-72] (0.65, 1.425) arc (85:-85:0.24 and 1.432);

\draw (0.751, -1.75) arc (20:70: 0.744 and 0.5);
\draw[thick] (-0.482, -1.11) arc (113:67:1.23 and 1.18);

\draw (-1.26, 1.38) node {$\gamma_{1}'$};
\draw (1.22, 1.38) node {$\gamma_{2}'$};
\draw (-1.85, -0.32) node {$\alpha_{1}'$};
\draw (1.84, -0.32) node {$\alpha_{2}'$};
\draw[red] (-0.312, -1.87) node {$\kappa_{1}'$};
\draw (0.312, -1.87) node {$\kappa_{2}'$};
\fill[black!0, shift={(0, -0.78)}] (-0.25, -0.17) -- (0.25, -0.17) -- (0.25, 0.15) -- (-0.25, 0.15);
\draw (0, -0.78) node {$\eta'$};
\draw[blue] (0.92, -1.22) node {$\zeta'$};
\draw (0, -0.138) node {$\alpha_{0}'$};

\draw (0, 2.45) node {$\delta_{0}'$};
\draw[rotate=72] (0, 2.45) node {$\delta_{1}'$};
\draw[rotate=144] (0, 2.45) node {$\delta_{2}'$};
\draw[rotate=216] (0, 2.45) node {$\delta_{4}'$};
\draw[rotate=288] (0, 2.45) node {$\delta_{3}'$};

\end{tikzpicture}
\caption{Configuration of curves on $X'$ of type $S_{0, p, b}$, $p+b=5$.}
\label{fig:5holes}
\end{figure}

Now we take curves $\{\gamma_{i}, \alpha_{i}\}$ on $Y$ by requiring $l_{Y}(\gamma_{i}) = l_{X'}(\gamma_{i}')$ and $l_{Y}(\alpha_{i}) = l_{X'}(\alpha_{i}')$. Then as in the previous proofs, using other auxiliary curves, Proposition \ref{prop:mainProp} detects the intersection patterns of curves, which means that: \begin{itemize}
\item some curves of $Y$ can be labelled as $\delta_{i}$ such that $l_{X}(\delta_{i}) = l_{X'}(\delta_{i}')$ for each $i$;
\item for $i=1, 2$, $\alpha_{i} \cup \gamma_{i}$ serves as a spine of generalized shirt $X_{i}$ immersed in $X$ by $\psi_{i} : X_{i} \rightarrow Y$;
\item $\alpha_{0} \cup \gamma_{1}$ also serves a spine of generalized shirt $X_{0}$ immersed in $X$ by $\psi_{0} : X_{0} \rightarrow Y$;
\item $\sigma \in \{\alpha_{i}, \gamma_{i}\}$ separates $\{\delta_{i_{1}}, \delta_{i_{2}}\}$ from $\{\delta_{i_{3}}, \delta_{i_{4}}, \delta_{i_{5}}\}$ if and only if $\sigma' \in \{\alpha_{i}', \gamma_{i}'\}$ separates $\{\delta_{i_{1}}', \delta_{i_{2}}'\}$ from $\{\delta_{i_{3}}', \delta_{i_{4}}', \delta_{i_{5}}'\}$.
\end{itemize} Moreover, $\psi_1$ and $\psi_2$ induce an immersion $\psi:X \to Y$ of a surface $X$ of type $S_{0, p, b}$ for $p + b = 5$, and we may assume that each $\delta_i$ is a peripheral curve of $X$ again from Proposition \ref{prop:mainProp}. We then orient $X$, $X'$ by requiring that $T_{\gamma_{1}}(\alpha_{1})$ and $T_{\gamma_{1}'}(\alpha_{1}')$ have the same length. We also define segments $\kappa_{i}$, $\eta$, $\zeta$ on $X$, analogously to those on $X'$.

Now Proposition \ref{prop:mainShirt} gives isometries $\phi_{i} : X_{i} \rightarrow X_{i}'$ that send each boundary to the corresponding boundary. This in particular implies that $\phi_{i}$ sends $\kappa_{i}$ to $\kappa_{i}'$ and $\eta$ to $\eta'$ for $i=1, 2$; $\phi_{0}$ sends $\kappa_{1}$ to $\kappa_{1}'$ and $\zeta$ to $\zeta'$.

Moreover, due to our choice of orientations, $\phi_{1}$ can be chosen as orientation-preserving. If $\phi_{2}$ is also orientation-preserving, the proof is done by gluing $\phi_{1}$ and $\phi_{2}$. Suppose to the contrary that $\phi_{2}$ is orientation-reversing. Our goal is to show that the (unsigned) distance between $\kappa_{1}$ and $\zeta$ differs to that between $\kappa_{1}'$ and $\zeta'$. This will then contradict the fact that $\phi_{0}$ is an isometry, which must preserve the unsigned twist of $\gamma_{1}$ at $X$.

We now parametrize $\gamma_{1}$ by arc length $\lambda$ so that $\kappa_{1}$ is located on the left side of $\delta_{1}$ while $\lambda$ increases, as in Figure \ref{fig:5holedHolonomy}. On $X$, we denote the signed displacement from  $\eta$ to $\zeta$ along $\gamma_{1}$ by $d$.  Further, we denote the signed displacement from $\kappa_{1}$ to $\eta$ along $\gamma_{1}$ by $D$. Similarly, we define the displacements $d'$ and $D'$ for curves and segments on $X'$.  

\begin{figure}[ht]
\begin{tikzpicture}[scale = 0.9]
\begin{scope}[shift={(-1, -3)}, scale=0.9]
\draw[thick] (0, 0) -- (0, 6);
\foreach \i in {0, 1, 2}{
\draw[thick, red] (-3, 0.8 + 2*\i) -- (0, 0.8+ 2*\i);
\draw (-3.25, 0.8+2*\i) node {$\tilde{\kappa}_{1}$}; 
\draw[thick] (3, 1.5 + 2*\i) -- (0, 1.5+2*\i);
\draw (3.25, 1.5+2*\i) node {$\tilde{\eta}$};
\draw[thick, dashed] (3, 0.5+2*\i) -- (0, 0.5+2*\i);
\draw[thick, blue] (3, 1.9 + 2*\i) -- (0, 1.9+2*\i);
\draw[thick, blue, densely dotted] (3, 1.1+2*\i) -- (0, 1.1+2*\i);
}
\draw (0, -0.2) node {$\tilde{\gamma}_{1}$};

\draw[thick, blue] (0.8, -0.2) -- (2.3, -0.2);
\draw[thick, blue, densely dotted] (0.8, -0.7) -- (2.3, -0.7);

\draw (3.4, -0.18) node {lifts of $\zeta$};
\draw (3.4, -0.7) node {lifts of $\zeta'$};

\end{scope}

\end{tikzpicture}
\caption{Description on the lifts of curves in Figure \ref{fig:5holes}. Here the black dashed lines are lifts of the geodesic segment perpendicular to $\gamma_{1}$ and $\delta_{0}$.}  \label{fig:5holedHolonomy}
\end{figure}
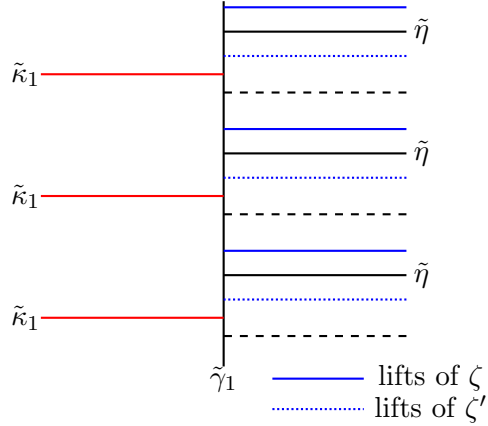

From the assumption, the twist at $\gamma_{2}$, $\gamma_{2}'$ are nonzero and opposite: this forces $d=-d' \neq 0$. We also observe that $|d|$, $|d'|$ is bounded by half of $l_{X}(\gamma_{1}) = l_{X'}(\gamma_{1}')$. This is because the geodesic perpendicular to $\gamma_{1}$ and $\delta_{0}$ is equidistant from $\eta$ along $\gamma_{1}$, and $\zeta$, $\zeta'$ cannot go across it. (See Figure \ref{fig:5holedHolonomy}: the black dashed lines are lifts of the geodesic perpendicular to $\gamma_{1}$ and $\delta_{0}$.)

Now we calculate the unsigned displacements between $\kappa_{1}$ and $\zeta$, as well as between $\kappa_{1}'$ and $\zeta'$. The former is $|D + d + n l_{X}(\gamma_{1})| $ for integers $n$ and the latter is $|D + d'+ n l_{X}(\gamma_{1})| = |D-d+ n l_{X}(\gamma_{1})|$ for integers $n$. If some of them are equal, then either \[
2D = n  l_{X}(\gamma_{1})\,\, \textrm{for some integer} \,\, n \quad \textrm{or} \quad 2d= n  l_{X}(\gamma_{1})\, \textrm{for some integer} \,\, n.
\]

If $2|D|$ were $0$ or $l_X(\gamma_1)$, the two surfaces $X$ and $X'$ would be related to each other.  Hence, we may assume that $2|D|$ is neither $0$ nor $l_X(\gamma_1)$. The former case is excluded since $|D|\le l_X(\gamma_1)/2$, and hence $2|D|$ cannot be a multiple of $  l_{X}(\gamma_{1})$. For the latter case, $2d$ is neither 0 (since $\gamma_{2}$ has nonzero twist) nor other multiples of $l_{X}(\gamma_{1})$ (since $|d| < l_{X}(\gamma_{1})/2$). This ends the proof.
\end{proof}

\section{Proof of the main theorem} \label{sec:proofofmainthm}

We are now ready to prove the main theorem, which we state here again. 
Roughly, our idea in the proof is to reconstruct a hyperbolic structure from a given length spectrum, with subsurfaces of low complexity as building blocks. A similar idea also appeared in Grothendieck's program \cite{Grothendieck_program} (see also \cite{Luo_survey}).

\aeLengthSpectra*

\begin{proof}
Suppose that $X \in \T(S) \setminus V$, $X' \in T(S)$ and $\mathcal{L}(X) = \mathcal{L}(X')$. 
We fix a pants decomposition $\C$ and an exhaustion $\{S_n\}$ of $S$ as described in Proposition \ref{prop:goodPantsDec}. That is, $S_n$ forms an increasing sequence of finite-type subsurfaces with $S = \cup_n S_n$ and for each $n \in \N$, the boundary $\partial S_n$ consists of curves in $\C$  and $S_{n+1}$ is made by attaching a generalized pair of pants or one-holed torus to $S_n$ along only one curve.
Furthermore, after modifying the pants decomposition as in the proof of Proposition \ref{prop:main2Torus}, one may assume the following: if $C_{k}\in \mathcal{C}$ bounds a one-holed torus that hosts another curve $C_{k'} \in \mathcal{C}$, then $l_{X'}(C_{k'}) > l_{X}(C_{k})$. Since taking different pants decompositions does not alter the simple length spectrum, we may assume so.

 We denote the subsurface of $X'$ corresponding to $S_{n}$ by $X'_{n}$. To clarify, $X'$ is decomposed into generalized pair of pants $P'_{i}$, glued with each other along boundaries, and $X'_{n} = \cup_{i=1}^{n} P'_{i}$ for each $n$. 

We can first rule out the cases of generalized pair of pants and one-holed/punctured torus since they were treated in the last section. Thus, we may begin with $X_{2}'$, a subsurface made out of two generalized pairs of pants. These cases were dealt with in Proposition \ref{prop:mainShirt} and \ref{prop:main2Torus}, so we can assume the isometric embedding $\psi_{2}$ of $X'_{2}$ into $X$. From this isometric embedding, we inductively extend along the exhaustion $X_n'$ and eventually obtain the desired isometry $X' \to X$. The fact that our base case $X_2'$ consists of two generalized pairs of pants play a role in this extension.

Now suppose that $\psi_{n} : X'_{n} \rightarrow X$ is an isometric embedding.  Let us denote the subsurface $X_{n+1}' \setminus X_{n}'$ by $P'$. Then $P'$ is attached to a pair of pants $Q' \subseteq X_{n}'$ along a curve $\gamma_{1}'$ constituting $\mathcal{C}'$. Since $X_{n}'$ contains at least two pairs of pants, $Q'$ is connected to yet another pair of pants $R' \subseteq X_{n}'$ along a curve $\gamma_{2}' \neq \gamma_{1}'$ comprising $\mathcal{C}'$. Let $Q$, $R$ be their respective images on $X$.

Let us first assume that $P'$ is a generalized pair of pants. 
Since $P'$ is the subsurface $X_{n+1}' \setminus X_n'$ and the exhaustion $\{S_n\}$ is made by attaching a generalized pair of pants or one-holed torus along only one boundary component at each step, we have $P' \cap Q' = \gamma_1'$. By the same reason, we also have $Q' \cap R' = \gamma_2'$, and hence $P', Q', R'$ are pairs of pants located in that order in Figure \ref{fig:5holes}.
We now define curves $\{\alpha_{i}', \gamma_{i}', \delta_{i}'\}$ on $P'$, $Q'$, $R'$ as in Figure \ref{fig:5holes}, and designate $\alpha_{i}$, $\gamma_{i}$, $\delta_{i}$ on $X$ by comparing lengths. Then the proof of Proposition \ref{prop:main5sphere} shows that $\alpha_{1} \cup \gamma_{1}$ becomes a spine of an immersed generalized shirt in $X$, which is divided into $Q$ and another immersed generalized pair of pants $P$. Moreover, the proposition gives an isometry $\phi : P' \cup Q' \cup R' \rightarrow P \cup Q \cup R$ sending each of $\{\alpha_{i}',\gamma_{i}', \delta_{i}'\}$ to corresponding $\{\alpha_{i}, \gamma_{i}, \delta_{i}\}$. 
In fact, more can be said from the proof of Proposition \ref{prop:main5sphere}. Note that the restriction $\psi_{n}|_{Q'\cup R'} : Q' \cup R' \rightarrow Q \cup R$ is also an isometry. In particular, it sends each of $\alpha_{i}'$, $\gamma_{i}'$, $\delta_{i}'$ on $Q' \cup R'$ to the corresponding curve on $Q \cup R$. The proof then guarantees that such an isometry can be extended to the entire isometry $\phi:P' \cup Q' \cup R' \rightarrow P \cup Q \cup R$. Thus, $\psi_{n}$ and $\phi$ can be glued on $Q' \cup R'$.

Denoting by $X_n$ the image of $X_n'$ under $\psi_n$, it remains to show that $P \cap X_n$ is a single curve, say, the image or $C_n'$ by $\psi_n$. We first claim that $P$ and $X_n$ have disjoint interiors. We observe that $X_n$ cannot cover all of $\operatorname{int}(P)$; otherwise, $P'$ has at least two boundary components and they are contained in $X_n'$, which was forbidden. Hence, if $X_n$ intersects $\operatorname{int}(P)$, then there exists a 
boundary component $\eta$ of $X_n$ that intersects $\operatorname{int}(P)$. Moreover, since $P$ is a generalized pair of pants, $\eta$ cannot be contained in $\operatorname{int}(P)$. This implies that there is a boundary curve, one of $\delta_1, \delta_2$, that $\eta$ intersects. Without loss of generality, we may assume that $\eta$ intersects $\delta_1$. 
Let $\eta'$ be the boundary curve of $X_n'$ corresponding to $\eta$. Since $P'$ is attached to $X_n'$ only along the curve $\gamma_1'$, two curves $\delta_1'$  and $\eta'$ are disjoint. Hence, we apply Lemma \ref{lem:disjointWitness2} to two $X'$, $\delta_1'$, and $\eta'$ and obtain auxiliary curves given by Lemma \ref{lem:disjointWitness2} to satisfy the length identity therein.  Since $\mathcal{L}(X) = \mathcal{L}(X')$, we can also find auxiliary curves in $X$ for $\delta_1$ and $\eta$ so that the same length identity holds. Since $X \in \T(S) \setminus V$, it follows from Lemma \ref{lem:disjointWitness} that $\delta_1$ and $\eta$ are disjoint, contradiction.

Furthermore, since $P'$ is attached to $X_{n}'$ only along $\gamma'$, $\delta'_{1}$ and $\delta_{2}'$ are not boundary curves of $X_{n}'$. Thus, $\delta_{1}$ and $\delta_{2}$ are also different from the boundary curves of $X_{n}$, and $P$ is also attached to $X_{n}$ only along $\gamma$. Thus, gluing $\psi_{n}$ and $\phi$ on $Q' \cup R'$ is sufficient to construct $\psi_{n+1}$.

If $P'$ is a one-holed torus, we still have to investigate whether $\phi$ respects the gluing at $\delta_{1}' = \delta_{2}'$. This time, we observe that $\phi|_{P' \cup Q'}$ becomes an isometry from $P' \cup Q'$, as an immersed generalized shirt, onto $P \cup Q$. Again, the proof of Proposition \ref{prop:main2Torus} asserts that $\phi|_{P' \cup Q'}$ can be extended to the isometry $P' \cup Q'$, as a 2-holed/punctured torus this time, onto $P \cup Q$. Thus, $\psi_{n+1}$ is well-defined also in this case.

Since $\{X_{n}'\}$ is an exhaustion of $X'$, we obtain an isometric embedding $\psi : X' \rightarrow X$ after this induction process. We now claim that $\psi$ is surjective. To show this, suppose not: $\psi(X')$ is a subsurface of $X$. Since $X$ is connected, the only possibility is that $X'$ has a boundary curve $C_{n}'$ while $C=\psi(C_{n}')$ is an interior curve of $X$. In this case, $C$ has a minimally intersecting curve $\eta$: if it is non-separating, then there exists another curve $\eta$ such that $i(C, \eta) = 1$; otherwise, there exists another curve $\eta$ such that $i(C, \eta) = 2$ and $i_{alg}(C, \eta)= 0$. Since $\mathcal{L}(X') = \mathcal{L}(X)$, there exists a curve $\eta' \subseteq X'$ such that $l_{X'}(\eta') = l_{X}(\eta)$. Since $\psi:X' \rightarrow X$ is an isometric embedding and $\mathcal{L}(X)$ is simple, it follows that $\psi(\eta') = \eta$. As such, $i(\psi(\eta'), C) = i(\eta, C) > 0$, which contradicts the assumption that $C$ is a boundary curve of $\psi(X')$.
\end{proof}

%
%

\section{Further questions} \label{sec:furtherquestions}

We conclude this article by suggesting some further questions.

\begin{enumerate} 

\item It can be asked whether the meagre set $V$ we constructed is optimal. Indeed, it is not known whether the isometry classes of \emph{all} hyperbolic surfaces are determined by their \emph{simple} length spectra. 
\item For surfaces of finite type, our argument descends to the moduli space. Indeed, we know that $\Mod^{+}(S_{g, p, b})$ (the usual mapping class group which consists of the isotopy classes of orientation-preserving homeomorphisms) acts on $\T_{g, p, b}$ properly discontinuously, whose quotient is the moduli space $\mathcal{M}(S)$. Thus, for example, we have that $V/\Mod^{+}(S_{g, p, b})$ is a meagre, i.e. countable union of submanifolds of positive codimensions, and (unmarked) hyperbolic surfaces outside it will be distinguished by their simple length spectra.

We hope that a similar argument can be made for surfaces of infinite type. For example, the mapping class group acts on the quasiconformal Teichm{\"u}ller space of some Riemann surfaces discretely and faithfully (See~\cite{fujikawa2004MCG}). Similar discussion for Fenchel-Nielsen Teichm{\"u}ller space is expected.

\item While studying Question (1), Aougab et al. suggested in~\cite{aougab2020covers} that finite covers of a closed topological surface might be probed via simple lifts of closed curves. Our result produces at least one (actually abundant) hyperbolic structure whose simple length spectrum is topologically rigid on a co-meagre subset of the Teichm\"uller space. This structure may help deal with this problem.

\end{enumerate}

%
%

\appendix

\section{Pants decompositions of surfaces} \label{appendix:A1}

We begin with a version of Richards' classification of surfaces. Besides of the genus $g$ (which may be infinite), each surface $S$ is associated with three invariants: \begin{itemize}
\item the space $X$ of ends of $\textrm{int}(S)$,
\item the space $Y$ of non-planar ends, and 
\item the space $Z$ of boundaries.
\end{itemize}

Here $X$ is a compact, separable, totally disconnected space and $Y$, $Z$ are disjoint closed subsets of $X$, where $Z$ consists of isolated points of $X$. Then $S$ is made from a sphere by removing $X \setminus Z$, then removing disjoint open discs, each containing one element of $Z$ and not containing any other elements of $X$, and then attaching $g$ handles that accumulate to points of $Y$. (See Theorem 3 of~\cite{richards1963surface})

We first consider the case $|X| \le 3$. This corresponds to finite-type surfaces $S_{g, p, b}$ with $p+b \le 3$, the Loch Ness monster with $p + b \le 2$, 1-punctured/bordered Jacob's ladder or the tripod surface. All these surfaces admit the pants decompositions desired in Proposition \ref{prop:goodPantsDec}. (See Figure \ref{fig:sporadic}.)

\begin{figure}
\begin{tikzpicture}
\begin{scope}[scale=1]
\begin{scope}[shift={(0, -1.5)}]
\draw (0, 0) -- (2.2, 0);
\draw (0, 0.6) -- (2.2, 0.6);
\draw (0, 0) .. controls (-0.3, 0) and (-0.4, -0.25) .. (-0.7, -0.25) .. controls (-1, -0.25) and (-1.3, -0.05) .. (-1.3, 0.3) .. controls (-1.3, 0.65) and (-1, 0.85) .. (-0.7, 0.85) .. controls (-0.4, 0.85) and (-0.3, 0.6) .. (0, 0.6);
\draw (-0.68, 0.3) circle (0.22 and 0.15);
\draw[thick, dashed] (-0.9, 0.3) arc (0:180:0.2 and 0.07);
\draw[thick] (-0.9, 0.3) arc (0:-180:0.2 and 0.07);
\draw[thick, dashed] (0.1, 0) arc (-90:90:0.1 and 0.3);
\draw[thick] (0.1, 0) arc (270:90:0.1 and 0.3);
\draw[thick, dashed] (1.3, 0) arc (-90:90:0.1 and 0.3);
\draw[thick] (1.3, 0) arc (270:90:0.1 and 0.3);

\begin{scope}[yscale=-1, shift={(-4.1, 3)}]
\fill[black!0] (4.45, -3.5) .. controls (4.55, -3.54) and (4.65, -3.7) .. (4.4, -3.9)  -- (5, -3.9) .. controls (4.75, -3.7) and (4.85, -3.54)..  (4.95, -3.5);
\draw(4.45, -3.5) .. controls (4.55, -3.54) and (4.65, -3.7) .. (4.4, -3.9) .. controls (4.17, -4.1) and (4.2, -4.3) .. (4.2, -4.35) .. controls (4.2, -4.4) and (4.2, -4.7) .. (4.7, -4.7) .. controls (5.2, -4.7) and (5.2, -4.4) .. (5.2, -4.35) .. controls (5.2, -4.3) and (5.23, -4.1) .. (5, -3.9) .. controls (4.75, -3.7) and (4.85, -3.54)..  (4.95, -3.5);
\draw (4.7, -4.28) circle (0.2 and 0.18);
\draw[thick] (4.847, -3.7) arc (0:180: 0.147 and 0.06);
\draw[thick, densely dashed] (4.847, -3.7) arc (0:-180: 0.147 and 0.06);
\draw[thick] (4.7, -4.46) arc (90:270:0.06 and 0.12);
\draw[thick, densely dashed] (4.7, -4.46) arc (90:-90:0.06 and 0.12);
\end{scope}

\begin{scope}[yscale=-1, shift={(-2.9, 3)}]
\fill[black!0] (4.45, -3.5) .. controls (4.55, -3.54) and (4.65, -3.7) .. (4.4, -3.9)  -- (5, -3.9) .. controls (4.75, -3.7) and (4.85, -3.54)..  (4.95, -3.5);
\draw(4.45, -3.5) .. controls (4.55, -3.54) and (4.65, -3.7) .. (4.4, -3.9) .. controls (4.17, -4.1) and (4.2, -4.3) .. (4.2, -4.35) .. controls (4.2, -4.4) and (4.2, -4.7) .. (4.7, -4.7) .. controls (5.2, -4.7) and (5.2, -4.4) .. (5.2, -4.35) .. controls (5.2, -4.3) and (5.23, -4.1) .. (5, -3.9) .. controls (4.75, -3.7) and (4.85, -3.54)..  (4.95, -3.5);
\draw (4.7, -4.28) circle (0.2 and 0.18);
\draw[thick] (4.847, -3.7) arc (0:180: 0.147 and 0.06);
\draw[thick, densely dashed] (4.847, -3.7) arc (0:-180: 0.147 and 0.06);
\draw[thick] (4.7, -4.46) arc (90:270:0.06 and 0.12);
\draw[thick, densely dashed] (4.7, -4.46) arc (90:-90:0.06 and 0.12);
\end{scope}

\foreach \i in {-1, 0, 1}{
\fill (2.8 - \i*0.25, 0.3) circle (0.03);
}

\end{scope}

\draw[decorate,decoration={brace,amplitude=10pt},xshift=-4pt,yshift=0pt] (4, -6) -- (4, 3.6);

\begin{scope}[shift={(4.5, 3.3)}, xscale=-1]
\draw (0.3, 0) -- (0, 0) .. controls (-0.3, 0) and (-0.4, -0.25) .. (-0.7, -0.25) .. controls (-1, -0.25) and (-1.3, -0.05) .. (-1.3, 0.3) .. controls (-1.3, 0.65) and (-1, 0.85) .. (-0.7, 0.85) .. controls (-0.4, 0.85) and (-0.3, 0.6) .. (0, 0.6) -- (0.3, 0.6);
\draw (-0.68, 0.3) circle (0.22 and 0.15);
\draw[thick, dashed] (-0.9, 0.3) arc (0:180:0.2 and 0.07);
\draw[thick] (-0.9, 0.3) arc (0:-180:0.2 and 0.07);
\draw[thick, dashed] (0.1, 0) arc (270:90:0.1 and 0.3);
\draw[thick] (0.1, 0) arc (-90:90:0.1 and 0.3);
\end{scope}

\begin{scope}[shift={(4.5, 2)}, xscale=-1]
\draw (0.3, 0) -- (-0.1, 0);
\draw (0.3, 0.6) -- (-0.1, 0.6);
\draw[thick] (-0.1, 0) arc (270:90:0.1 and 0.3);
\draw[thick] (-0.1, 0) arc (-90:90:0.1 and 0.3);
\end{scope}

\begin{scope}[shift={(4.4, -0.1)}]
\draw (-0.2, 0) -- (2.2, 0);
\draw (-0.2, 0.6) -- (2.2, 0.6);
\draw[thick, dashed] (0.1, 0) arc (-90:90:0.1 and 0.3);
\draw[thick] (0.1, 0) arc (270:90:0.1 and 0.3);
\draw[thick, dashed] (1.3, 0) arc (-90:90:0.1 and 0.3);
\draw[thick] (1.3, 0) arc (270:90:0.1 and 0.3);

\begin{scope}[yscale=-1, shift={(-4.1, 3)}]
\fill[black!0] (4.45, -3.5) .. controls (4.55, -3.54) and (4.65, -3.7) .. (4.4, -3.9)  -- (5, -3.9) .. controls (4.75, -3.7) and (4.85, -3.54)..  (4.95, -3.5);
\draw(4.45, -3.5) .. controls (4.55, -3.54) and (4.65, -3.7) .. (4.4, -3.9) .. controls (4.17, -4.1) and (4.2, -4.3) .. (4.2, -4.35) .. controls (4.2, -4.4) and (4.2, -4.7) .. (4.7, -4.7) .. controls (5.2, -4.7) and (5.2, -4.4) .. (5.2, -4.35) .. controls (5.2, -4.3) and (5.23, -4.1) .. (5, -3.9) .. controls (4.75, -3.7) and (4.85, -3.54)..  (4.95, -3.5);
\draw (4.7, -4.28) circle (0.2 and 0.18);
\draw[thick] (4.847, -3.7) arc (0:180: 0.147 and 0.06);
\draw[thick, densely dashed] (4.847, -3.7) arc (0:-180: 0.147 and 0.06);
\draw[thick] (4.7, -4.46) arc (90:270:0.06 and 0.12);
\draw[thick, densely dashed] (4.7, -4.46) arc (90:-90:0.06 and 0.12);
\end{scope}

\begin{scope}[yscale=-1, shift={(-2.9, 3)}]
\fill[black!0] (4.45, -3.5) .. controls (4.55, -3.54) and (4.65, -3.7) .. (4.4, -3.9)  -- (5, -3.9) .. controls (4.75, -3.7) and (4.85, -3.54)..  (4.95, -3.5);
\draw(4.45, -3.5) .. controls (4.55, -3.54) and (4.65, -3.7) .. (4.4, -3.9) .. controls (4.17, -4.1) and (4.2, -4.3) .. (4.2, -4.35) .. controls (4.2, -4.4) and (4.2, -4.7) .. (4.7, -4.7) .. controls (5.2, -4.7) and (5.2, -4.4) .. (5.2, -4.35) .. controls (5.2, -4.3) and (5.23, -4.1) .. (5, -3.9) .. controls (4.75, -3.7) and (4.85, -3.54)..  (4.95, -3.5);
\draw (4.7, -4.28) circle (0.2 and 0.18);
\draw[thick] (4.847, -3.7) arc (0:180: 0.147 and 0.06);
\draw[thick, densely dashed] (4.847, -3.7) arc (0:-180: 0.147 and 0.06);
\draw[thick] (4.7, -4.46) arc (90:270:0.06 and 0.12);
\draw[thick, densely dashed] (4.7, -4.46) arc (90:-90:0.06 and 0.12);
\end{scope}

\foreach \i in {-1, 0, 1}{
\fill (2.8 - \i*0.25, 0.3) circle (0.03);
}

\end{scope}

\begin{scope}[shift={(4.5, -1.5)}]
\draw (-0.3, 0) -- (0.1, 0);
\draw (-0.3, 0.6) -- (0.1, 0.6);
\draw[thick] (0.1, 0) arc (270:90:0.1 and 0.3);
\draw[thick, dashed] (0.1, 0) arc (-90:90:0.1 and 0.3);
\draw (0.1, 0.6) arc (-90:-30:0.6);
\draw (0.1, 0) arc (90:30:0.6);
\draw[thick, shift={(0.1+0.45*1.732050807568877, 0.75)}, rotate=-30] (0, 0) circle (0.3 and 0.1);
\draw[thick, shift={(0.1+0.45*1.732050807568877, -0.15)}, rotate=30] (0, 0) circle (0.3 and 0.1);
\draw[shift={(0.1+0.6*1.732050807568877, 0.6)}] (0, 0) arc (150:210:0.6);
\foreach \i in {1, 2}{
\draw[thick] (0.1, 0) arc (270:90:0.1 and 0.3);
\draw[thick, dashed] (0.1, 0) arc (-90:90:0.1 and 0.3);
\draw (0.1, 0.6) arc (-90:-30:0.6);
}
\end{scope}

\begin{scope}[shift={(5.2, -3.7)}]
\draw (-1.1, 0)  -- (2.2, 0);
\draw (-1.1, 0.6) -- (-1, 0.6) .. controls (-0.8, 0.6) and (-0.7, 0.7) .. (-0.7, 0.9); 
\draw (-0.3, 0.9) .. controls (-0.3, 0.7) and (-0.2, 0.6) .. (0.2, 0.6) -- (2.2, 0.6);
\draw[thick] (-0.5, 0.9) circle (0.2 and 0.07);

\draw[thick, dashed] (0.1, 0) arc (-90:90:0.1 and 0.3);
\draw[thick] (0.1, 0) arc (270:90:0.1 and 0.3);
\draw[thick, dashed] (1.3, 0) arc (-90:90:0.1 and 0.3);
\draw[thick] (1.3, 0) arc (270:90:0.1 and 0.3);

\begin{scope}[yscale=-1, shift={(-4.1, 3)}]
\fill[black!0] (4.45, -3.5) .. controls (4.55, -3.54) and (4.65, -3.7) .. (4.4, -3.9)  -- (5, -3.9) .. controls (4.75, -3.7) and (4.85, -3.54)..  (4.95, -3.5);
\draw(4.45, -3.5) .. controls (4.55, -3.54) and (4.65, -3.7) .. (4.4, -3.9) .. controls (4.17, -4.1) and (4.2, -4.3) .. (4.2, -4.35) .. controls (4.2, -4.4) and (4.2, -4.7) .. (4.7, -4.7) .. controls (5.2, -4.7) and (5.2, -4.4) .. (5.2, -4.35) .. controls (5.2, -4.3) and (5.23, -4.1) .. (5, -3.9) .. controls (4.75, -3.7) and (4.85, -3.54)..  (4.95, -3.5);
\draw (4.7, -4.28) circle (0.2 and 0.18);
\draw[thick] (4.847, -3.7) arc (0:180: 0.147 and 0.06);
\draw[thick, densely dashed] (4.847, -3.7) arc (0:-180: 0.147 and 0.06);
\draw[thick] (4.7, -4.46) arc (90:270:0.06 and 0.12);
\draw[thick, densely dashed] (4.7, -4.46) arc (90:-90:0.06 and 0.12);
\end{scope}

\begin{scope}[yscale=-1, shift={(-2.9, 3)}]
\fill[black!0] (4.45, -3.5) .. controls (4.55, -3.54) and (4.65, -3.7) .. (4.4, -3.9)  -- (5, -3.9) .. controls (4.75, -3.7) and (4.85, -3.54)..  (4.95, -3.5);
\draw(4.45, -3.5) .. controls (4.55, -3.54) and (4.65, -3.7) .. (4.4, -3.9) .. controls (4.17, -4.1) and (4.2, -4.3) .. (4.2, -4.35) .. controls (4.2, -4.4) and (4.2, -4.7) .. (4.7, -4.7) .. controls (5.2, -4.7) and (5.2, -4.4) .. (5.2, -4.35) .. controls (5.2, -4.3) and (5.23, -4.1) .. (5, -3.9) .. controls (4.75, -3.7) and (4.85, -3.54)..  (4.95, -3.5);
\draw (4.7, -4.28) circle (0.2 and 0.18);
\draw[thick] (4.847, -3.7) arc (0:180: 0.147 and 0.06);
\draw[thick, densely dashed] (4.847, -3.7) arc (0:-180: 0.147 and 0.06);
\draw[thick] (4.7, -4.46) arc (90:270:0.06 and 0.12);
\draw[thick, densely dashed] (4.7, -4.46) arc (90:-90:0.06 and 0.12);
\end{scope}

\foreach \i in {-1, 0, 1}{
\fill (2.8 - \i*0.25, 0.3) circle (0.03);
}

\end{scope}

\begin{scope}[shift={(5.5, -6.4)}]
\draw (-1.3, 0) -- (-0.9, 0);
\draw (-1.3, 0.6) -- (-0.9, 0.6);
\draw[thick] (-0.9, 0) arc (270:90:0.1 and 0.3);
\draw[thick, dashed] (-0.9, 0) arc (-90:90:0.1 and 0.3);
\draw (-0.9, 0.6) .. controls (-0.5, 0.6) and (-0.3, 1.1) .. (0.1, 1.1) -- (2.2, 1.1);
\draw (-0.9, 0) .. controls (-0.5, 0) and (-0.3, -0.5) .. (0.1, -0.5) -- (2.2, -0.5);
\draw (2.2, 0.5) --  (0, 0.5) arc (90:270:0.2) -- (2.2, 0.1);

\begin{scope}[shift={(0, 0.5)}]

\draw[thick, dashed] (0.1, 0) arc (-90:90:0.1 and 0.3);
\draw[thick] (0.1, 0) arc (270:90:0.1 and 0.3);
\draw[thick, dashed] (1.3, 0) arc (-90:90:0.1 and 0.3);
\draw[thick] (1.3, 0) arc (270:90:0.1 and 0.3);

\begin{scope}[yscale=-1, shift={(-4.1, 3)}]
\fill[black!0] (4.45, -3.5) .. controls (4.55, -3.54) and (4.65, -3.7) .. (4.4, -3.9)  -- (5, -3.9) .. controls (4.75, -3.7) and (4.85, -3.54)..  (4.95, -3.5);
\draw(4.45, -3.5) .. controls (4.55, -3.54) and (4.65, -3.7) .. (4.4, -3.9) .. controls (4.17, -4.1) and (4.2, -4.3) .. (4.2, -4.35) .. controls (4.2, -4.4) and (4.2, -4.7) .. (4.7, -4.7) .. controls (5.2, -4.7) and (5.2, -4.4) .. (5.2, -4.35) .. controls (5.2, -4.3) and (5.23, -4.1) .. (5, -3.9) .. controls (4.75, -3.7) and (4.85, -3.54)..  (4.95, -3.5);
\draw (4.7, -4.28) circle (0.2 and 0.18);
\draw[thick] (4.847, -3.7) arc (0:180: 0.147 and 0.06);
\draw[thick, densely dashed] (4.847, -3.7) arc (0:-180: 0.147 and 0.06);
\draw[thick] (4.7, -4.46) arc (90:270:0.06 and 0.12);
\draw[thick, densely dashed] (4.7, -4.46) arc (90:-90:0.06 and 0.12);
\end{scope}

\begin{scope}[yscale=-1, shift={(-2.9, 3)}]
\fill[black!0] (4.45, -3.5) .. controls (4.55, -3.54) and (4.65, -3.7) .. (4.4, -3.9)  -- (5, -3.9) .. controls (4.75, -3.7) and (4.85, -3.54)..  (4.95, -3.5);
\draw(4.45, -3.5) .. controls (4.55, -3.54) and (4.65, -3.7) .. (4.4, -3.9) .. controls (4.17, -4.1) and (4.2, -4.3) .. (4.2, -4.35) .. controls (4.2, -4.4) and (4.2, -4.7) .. (4.7, -4.7) .. controls (5.2, -4.7) and (5.2, -4.4) .. (5.2, -4.35) .. controls (5.2, -4.3) and (5.23, -4.1) .. (5, -3.9) .. controls (4.75, -3.7) and (4.85, -3.54)..  (4.95, -3.5);
\draw (4.7, -4.28) circle (0.2 and 0.18);
\draw[thick] (4.847, -3.7) arc (0:180: 0.147 and 0.06);
\draw[thick, densely dashed] (4.847, -3.7) arc (0:-180: 0.147 and 0.06);
\draw[thick] (4.7, -4.46) arc (90:270:0.06 and 0.12);
\draw[thick, densely dashed] (4.7, -4.46) arc (90:-90:0.06 and 0.12);
\end{scope}
\foreach \i in {-1, 0, 1}{
\fill (2.8 - \i*0.25, 0.3) circle (0.03);
}

\end{scope}

\begin{scope}[shift={(0, 0.1)}, yscale=-1]

\draw[thick, dashed] (0.1, 0) arc (-90:90:0.1 and 0.3);
\draw[thick] (0.1, 0) arc (270:90:0.1 and 0.3);
\draw[thick, dashed] (1.3, 0) arc (-90:90:0.1 and 0.3);
\draw[thick] (1.3, 0) arc (270:90:0.1 and 0.3);

\begin{scope}[yscale=-1, shift={(-4.1, 3)}]
\fill[black!0] (4.45, -3.5) .. controls (4.55, -3.54) and (4.65, -3.7) .. (4.4, -3.9)  -- (5, -3.9) .. controls (4.75, -3.7) and (4.85, -3.54)..  (4.95, -3.5);
\draw(4.45, -3.5) .. controls (4.55, -3.54) and (4.65, -3.7) .. (4.4, -3.9) .. controls (4.17, -4.1) and (4.2, -4.3) .. (4.2, -4.35) .. controls (4.2, -4.4) and (4.2, -4.7) .. (4.7, -4.7) .. controls (5.2, -4.7) and (5.2, -4.4) .. (5.2, -4.35) .. controls (5.2, -4.3) and (5.23, -4.1) .. (5, -3.9) .. controls (4.75, -3.7) and (4.85, -3.54)..  (4.95, -3.5);
\draw (4.7, -4.28) circle (0.2 and 0.18);
\draw[thick] (4.847, -3.7) arc (0:180: 0.147 and 0.06);
\draw[thick, densely dashed] (4.847, -3.7) arc (0:-180: 0.147 and 0.06);
\draw[thick] (4.7, -4.46) arc (90:270:0.06 and 0.12);
\draw[thick, densely dashed] (4.7, -4.46) arc (90:-90:0.06 and 0.12);
\end{scope}

\begin{scope}[yscale=-1, shift={(-2.9, 3)}]
\fill[black!0] (4.45, -3.5) .. controls (4.55, -3.54) and (4.65, -3.7) .. (4.4, -3.9)  -- (5, -3.9) .. controls (4.75, -3.7) and (4.85, -3.54)..  (4.95, -3.5);
\draw(4.45, -3.5) .. controls (4.55, -3.54) and (4.65, -3.7) .. (4.4, -3.9) .. controls (4.17, -4.1) and (4.2, -4.3) .. (4.2, -4.35) .. controls (4.2, -4.4) and (4.2, -4.7) .. (4.7, -4.7) .. controls (5.2, -4.7) and (5.2, -4.4) .. (5.2, -4.35) .. controls (5.2, -4.3) and (5.23, -4.1) .. (5, -3.9) .. controls (4.75, -3.7) and (4.85, -3.54)..  (4.95, -3.5);
\draw (4.7, -4.28) circle (0.2 and 0.18);
\draw[thick] (4.847, -3.7) arc (0:180: 0.147 and 0.06);
\draw[thick, densely dashed] (4.847, -3.7) arc (0:-180: 0.147 and 0.06);
\draw[thick] (4.7, -4.46) arc (90:270:0.06 and 0.12);
\draw[thick, densely dashed] (4.7, -4.46) arc (90:-90:0.06 and 0.12);
\end{scope}
\foreach \i in {-1, 0, 1}{
\fill (2.8 - \i*0.25, 0.3) circle (0.03);
}

\end{scope}

\end{scope}

\end{scope}
\end{tikzpicture}
\caption{Pants decompositions of few-ended surfaces}
\label{fig:sporadic}
\end{figure}
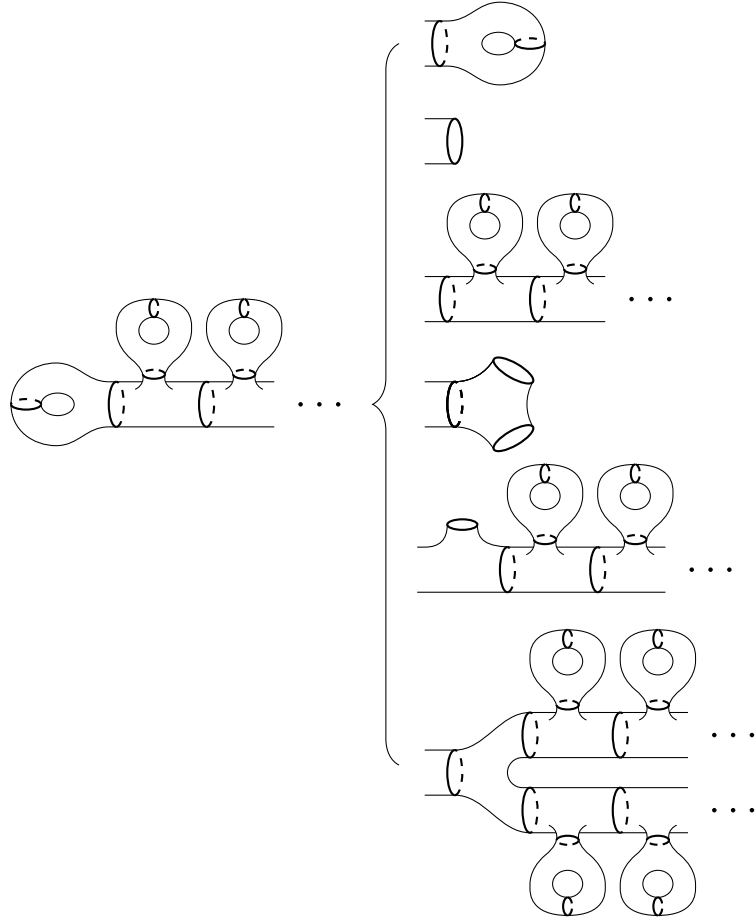

From now on, we consider the case $|X| \ge 4$. Let $\Sigma$ be a sphere, $K$ be a Cantor set on $\Sigma$, let $Y \subseteq X \subseteq K$ be closed sets of $\Sigma$, and let $Z$ be a subset of $X \setminus Y$ consisting of isolated points. For convenience, let $\mathcal{I}_{0} := \{(k, i) : k \in \Z_{>0}, 1 \le i \le 2^{k}\}$. Since $K$ is a Cantor set, there exist open discs $\{ U_{k, i}\}_{(k, i) \in \mathcal{I}_{0}}$ such that  \begin{itemize}
\item $\overline{U_{k, 1}}$, $\ldots$, $\overline{U_{k, 2^{k}}}$ are disjoint for each $k$,
\item $U_{k, i}$ contains $\overline{U_{k+1, 2i-1}} \cup \overline{U_{k+1, 2i}}$, and 
\item $\bigcap_{k} \left(\cup_{i=1}^{2^{k}} U_{k, i}\right) = K$.
\end{itemize}

Since $|X| \ge 4$, there exists disjoint $U_{k_{t}, i_{t}}$ for $t=1, \ldots, 4$ such that $X \cap U_{k_{t}, i_{t}} \neq \emptyset$ for each $t$ and $X \subseteq \cup_{t=1}^{4} U_{k_{t}, i_{t}}$. By relabelling, we may assume that $U_{2,i}$ intersects $X$ for $i =1, 2, 3, 4$. Now, we declare a subset $\mathcal{I}$ of $\mathcal{I}_{0}$ as follows: \[
\left\{ \begin{array}{ccc} \big[(k, 2j-1) \in \mathcal{I}\, \,\textrm{and} \,\,(k, 2j) \in \mathcal{I}\big] \,\, &\Leftrightarrow &U_{k, 2j-1} \cap X \neq \emptyset \neq U_{k, 2j} \cap X, \\ 
\big[(k, 2j-1) \notin \mathcal{I}\,\, \textrm{and} \,\,(k, 2j) \notin \mathcal{I}\big] \,\, &\Leftrightarrow& \big[U_{k, 2j-1} \cap X = \emptyset \,\,\textrm{or}\, \,U_{k, 2j} \cap X=\emptyset \big]. \end{array}\right.
\]
For convenience, we exclude $(1, 1)$ from $\mathcal{I}_{0}$ as well. ($\ast$)

We claim that $\Sigma \setminus\Big( \{ \partial U_{k, i} : (k, i) \in \mathcal{I}\} \cup X \Big)$ is composed of pairs of pants and cylinders, where the cylinders are precisely of the form \[
\left\{ U_{k, i} \setminus \{p\}: \begin{matrix}
\textrm{$p \in X$ is an isolated point in $X$}, \\
(k, i) = \min \big\{ (l, j) \in \mathcal{I}_{0} : U_{l, j} \cap X = \{p\}\big\}
\end{matrix} \right\}
\]
(Here, we employ the lexicographic order on $\mathcal{I}_{0}$.)

First observe that $(1, 2), (2, 1), \ldots, (2, 4) \in \mathcal{I}_{0}$, and that $\Sigma \setminus \big( \overline{U_{2, 1}} \cup \ldots \cup \overline{U_{2, 4}}\big)$ is decomposed into two pairs of pants thanks to the modification $(\ast$). 

Now let $z$ be a point on $U_{2, 1} \setminus\Big( \{ \partial U_{k, i} : (k, i) \in \mathcal{I}\} \cup X \Big)$. Pick the maximal $(M, t) \in \mathcal{I}_{0}$ such that $U_{M, t} \ni z$; then $z \in U_{M, t} \setminus \big( \overline{U_{M+1, 2t-1}} \cup \overline{U_{M+1, 2t}} \big)$. Among the ancestors of $U_{M, t}$, pick the most ancient one $U_{l, j}$ such that $U_{l, j} \cap X = U_{M, t} \cap X$.  Here, we know that $l \ge 2$ because $U_{1,1}$ contains $U_{2, 2} \cap X$, a nonempty subset of $X$ disjoint from $U_{M, t} \cap X \subseteq U_{2, 1} \cap X$. Let $(l, j')$ be the other immediate child of the parent of $(l, j)$, i.e., \[
\{(l, j), (l, j')\} = \big\{ (l, 2 \lceil j/2 \rceil - 1), (l, 2 \lceil j/2 \rceil) \big\}.
\]
Due to the minimality of $(l, j)$, we know that $U_{l, j'} \cap X = (U_{l-1, \lceil j/2 \rceil} \cap X) \setminus (U_{l, j} \cap X)$ is nonempty. Hence, both $(l, j)$ and $(l', j')$ belong to $\mathcal{I}$.

We now have two cases:
\begin{enumerate}
\item If $U_{l, j} \cap X$ is a singleton $\{p\}$, then $p$ is an isolated point of $X$. Since $U_{l, j} \cap X = \{p\}$, no two distinct descendants of $U_{k, i}$ can both intersect $X$. Hence, $U_{k, i} \cap \Big(\{ \partial U_{k, i} : (k, i) \in \mathcal{I}\} \cup X \Big) = U_{k, i} \setminus \{p\}$.
\item If $U_{l, j} \cap X$ is not a singleton, then \[
\# \{ (m, u): 2^{m-l} (j-1) < u \le 2^{m-l} j , U_{m, u} \cap X \neq \emptyset\} \ge 2
\]
for some $m>l$. Indeed, if not, then $U_{l, j} \cap X$ is an intersection of nested sequence of connected compact sets, which is connected in $K$, hence is a singleton. 

Pick a minimal $m$. Here, recall that $U_{M, t} \cap X = U_{l, j} \cap X$. This implies that \begin{equation}\label{eqn:pantsNoOther}
U_{M, t'} \cap X =\emptyset \quad (t' \neq t \,\textrm{and}\,2^{M-l}(j-1) < u \le 2^{M-l} j). 
\end{equation}
In other words, among the descendants of $U_{l, j}$ at some level between $l$ and $M$, only one can intersect $X$. Hence, $m$ is greater than $k$.

Note that if $U_{m, u}$ intersects $X$ for some $u$, so does its parent. Also note that at most 2 distinct indices can have the same parent. This implies that $\# \{ (m, u): 2^{m-l} (j-1) < u \le 2^{m-l} j , U_{m, u} \cap X \neq \emptyset\}$ cannot be greater than 2 and is exactly 2. Let $U_{m, u'}, U_{m, u''}$ be the ones intersecting $X$. By Display \ref{eqn:pantsNoOther}, $U_{m, u'}$ and $U_{m, u''}$ must be descendants of $U_{M, t}$. Moreover, due to the minimality of $m$, we have \begin{equation}\label{eqn:pantsExact1}
\# \{ (k, i): 2^{k-l} (j-1) <i \le 2^{k-l} j , U_{k, i} \cap X \neq \emptyset\} = 1 \quad(l \le k < m).
\end{equation}

We now observe the descendants of $U_{l, j}$ recorded by $\mathcal{I}$. Each open set $U_{k, i}$ contained in $U_{l, j}$ is either disjoint from $U_{m, u'}$ and $U_{m, u''}$, is a parent of one of $U_{m, u'}$ or $U_{m, u''}$, or is a descendant of one of $U_{m, u'}$ and $U_{m, u''}$ (themselves included). Those of the first category cannot intersect $X$, because $U_{l, j} \cap X$ is partitioned into $U_{m, u'} \cap X$ and $U_{m, u'} \cap X$. Hence, they are not recorded by $\mathcal{I}$. Those of the second category are not recorded by $\mathcal{I}$ either, because of Display \ref{eqn:pantsExact1}. Hence, the descendants of $U_{l, j}$ recorded by $\mathcal{I}$ must be descendants of $U_{m, u'}$ and $U_{m, u''}$.  In other words, $\partial U_{k, i}$'s for $(k, i) \in \mathcal{I}$ either lie outside $U_{l, j}$ or is contained in $U_{m, u'}$ and $U_{m, u''}$. Hence, the connected component of $U_{2, 1} \setminus\Big( \{ \partial U_{k, i} : (k, i) \in \mathcal{I}\} \cup X\Big)$ containing $z$ is precisely $U_{l, j} \cap \big(\overline{U_{m, u'}} \cup \overline{U_{m, u''}}\big)$, a pair of pants.
\end{enumerate}
In the above discussion, we note that each pair of pants have boundaries coming from three  indices in $\mathcal{I}$, where one is an ancestor (and is the direct parent in $\mathcal{I}$) of the other two, with the exception of $U_{1, 2}^{c} \cap U_{2, 1}^{c} \cap U_{2, 2}^{c}$. Conversely, for each $(k, i) \in \mathcal{I}$, $\partial U_{k, i}$ either is the boundary of two pairs of pants (when $(k, i)$ has its decendants in $\mathcal{I}$), or is the boundary of a pair of pants and a cylinder (when $U_{k, i}$ is the largest open set intersecting $X$ precisely at $p$).

We now draw the seam. For each $(k, i) \in \mathcal{I}$, pick two points $p_{k, i}, q_{k, i}$ on $\partial U_{k, i}$. Now, on each pair of pants $P$ bounded by $U_{k, i}, U_{k', i'}$ and $U_{k'', i''}$, where the latter two are descendants of the former one, we draw three disjoint continuous paths in $P$ connecting $p_{k, i}$ to $p_{k', i'}$, $q_{k, i}$ to $q_{k'', i''}$, and $q_{k', i'}$ to $p_{k'', i''}$. On a cylinder containing a single point $p \in X$ and bounded by $\partial U_{k, i}$, we draw two continuous paths, disjoint in the interior, that connect $p_{k, i}$ and $q_{k, i}$ to $p$. We declare the union of these paths to be a seam $\mathcal{S}$. Then a connected component of $\mathcal{S}$ is an embedded arc in $\Sigma \setminus X$, as it is an embedded arc when restricted to each pair of pants and cylinder, and near each boundaries $\partial U_{k, i}$'s for $(k, i) \in \mathcal{I}$.

Now, we remove the cylinders containing a point of $Z$. This way, we can attach boundary components to ends in $Z$. Note that $\mathcal{S}$ restricted to $\Sigma \setminus \{\textrm{cylinders}\}$ still serves as a seam.

When $g$ is nonzero and finite (which implies that $Y = \emptyset$), we cut $\Sigma \setminus X$ along $\partial U_{1, 2}$ and insert a surface $\Sigma_{g, 2}$ with genus $g$ and with two boundaries. We can also reconstruct the seam on $\Sigma_{g, 2}$.

It remains to realize the non-planar ends $Y$ when they exist. First, isolated points in $Y$ are associated with disjoint cylinders. We replace each cylinder with half of (seamed) Jacob's ladder, making sure that the existing seam and the seam on the Jacob's ladder are glued up well. Next, for each $(k, i) \in \mathcal{I}$ such that $U_{k, i} \cap Y \neq \emptyset$, we cut the surface along $\partial U_{k, i}$ and insert a surface $\Sigma_{1, 2}$ with genus 1 and with two boundaries. We can also reconstruct the seam on the inserted surface. This way, each point in $Y$ is approached by genera while those outside $Y$ is not.

See Figure \ref{fig:infinitesurface} for the resulting pants decomposition.

\section{Analytic functions} \label{appendix:A}
\label{section:analytic}

In this section, we prove the following lemma.

\begin{lem}\label{lem:analytic}
Let $f : U \rightarrow \mathbb{R}^{m}$ be an analytic function on a domain $U \subseteq \mathbb{R}^{n}$ that does not vanish identically. Then there exists a countable family of submanifolds $\{S_{i}\}_{i \in \mathbb{N}}$ of $U$ such that $f \neq 0$ outside $\cup_{i} S_{i}$.
\end{lem}

\begin{proof}
It suffices to prove for $m=1$. We define the sets $$C_{j} = \{x \in U : \partial_{\alpha} f(x) = 0 \mbox{ for all index } \alpha \mbox{ with }|\alpha| \le j\}$$ for $j \ge 0$. We observe that $C_{0} \supseteq C_{1} \supseteq \cdots$ and $\cap_{i} C_{i} = \emptyset$. Indeed, the existence of a point $x \in \cap_{i} C_{i}$ will imply $f \equiv 0$ due to the analyticity of $f$.

We now define $S_{i} = C_{i-1} \setminus C_{i}$. It follows from the previous observation that $\bigcup_{i \in \mathbb{N}} S_{i} = C_{0}$, and $f \neq 0$ outside $C_{0}$. It remains to show that $S_{i}$ is contained in a finite union of submanifolds. For each index $\alpha$ we define $S_{\alpha, j} = \{ x : \partial_{\alpha} f(x) = 0$ but $\partial_{\alpha'} f(x) \neq 0\}$, where $\alpha'_{i} = \alpha_{i} + \delta_{ij}$. Then the implicit function theorem tells us that $S_{\alpha, j}$ is a submanifold. Since $S_{i} \subseteq \cup_{|\alpha| = i, 1\le j\le n} S_{\alpha, j}$, the proof is done.
\end{proof}

\section{Proof of Lemma \ref{lem:pinching}} \label{appendix:B}
\label{section:pinching}

This appendix stands for the proof of the following lemma. A convergence of Fenchel-Nielsen coordinates as the moduli go to infinity was also dealt with in \cite{bourque2015conformal}.

\pinching*

\begin{proof}
Here $C_{1}$ is adjacent to one or two pairs of pants. In each case, we draw the simple geodesic segments perpendicular to the boundaries of pairs of pants as described in Figure \ref{fig:pantsC1}. We also pick basepoints $\eta \cap \kappa_{3} = p$, $\eta^{\pm} \cap \kappa_{3}^{\pm} = p^{\pm}$ in each case. The complement of these pants in $S$ is denoted by $R$.

\begin{figure}[ht]
\begin{tikzpicture}
\draw[thick, dashed] (0.17, 1.8) .. controls (0.9, 1.8) and (1.5, 1.4) .. (1.5, 0.6);
\draw[thick] (1.5, 0.6) .. controls (1.5, -0.1) and (0.7, -0.4) .. (0.8, -1.15);
\draw[thick, dashed] (0.17, 1.2) .. controls (-0.2, 1.2) and (-0.5, 1) .. (-0.5, 0.6) .. controls (-0.5, 0.2) and (-0.3, 0) .. (0, 0);
\draw[thick] (-0.17, 1.2) .. controls (0.2, 1.2) and (0.5, 1) .. (0.5, 0.6) .. controls (0.5, 0.2) and (0.3, 0) .. (0, 0);

\draw (0, 2) .. controls (-0.8, 2) and (-1.5, 1.5) .. (-1.5, 0.6) .. controls (-1.5, -0.2) and (-0.8, -0.4) .. (-1, -1) arc (180:360:1 and 0.25) .. controls (0.8, -0.4) and (1.5, -0.2) .. (1.5, 0.6) .. controls (1.5, 1.6) and (0.8, 2) .. cycle;
\draw[dashed] (1, -1) arc (0:180:1 and 0.25);
\draw (0, 1) arc (90:450:0.3 and 0.5);
\draw (0, 2) arc (90:270:0.2 and 0.5);
\draw[dashed] (0, 2) arc (90:-90:0.2 and 0.5);
\draw[thick] (-0.17, 1.8) .. controls (-0.7, 1.8) and (-1.27, 1.4) .. (-1.27, 0.65) .. controls (-1.27, -0.2) and (-0.7, -0.4) .. (-0.8, -1.15);
\draw[thick] (0, 0) arc (90:180+13.8:0.22 and 1);
\draw[thick, dashed] (0, 0) arc (90:180-166.2:0.22 and 1);

\draw (0, 2.3) node {$C_{1}$};
\draw (0.7, 0.6) node {$\eta$};
\draw (-0.6, -0.52) node {$\kappa_{1}$};
\draw (0.64, -0.52) node {$\kappa_{2}$};
\draw (-0.2, -1.45) node {$\kappa_{3}$};

\begin{scope}[shift={(4.8, 0.3)}]
\draw (0, 0.5) arc (270:330:1.4);
\draw (0, 0.5) arc (270:210:1.4);
\draw (0, -0.5) arc (90:30:1.4);
\draw (0, -0.5) arc (90:150:1.4);
\draw (2.0784609691, 0.7) arc (150:210:1.4);
\draw (-2.0784609691, 0.7) arc (30:-30:1.4);

\draw[shift={(2.0784609691, 0.7)}, rotate=60] (0, 0) arc (-90:90:0.2 and 0.5);
\draw[dashed, shift={(2.0784609691, 0.7)}, rotate=60] (0, 0) arc (270:90:0.2 and 0.5);
\draw[shift={(2.0784609691, -0.7)}, rotate=-60] (0, 0) arc (90:-90:0.2 and 0.5);
\draw[dashed, shift={(2.0784609691, -0.7)}, rotate=-60] (0, 0) arc (90:270:0.2 and 0.5);

\draw[shift={(-2.0784609691, 0.7)}, rotate=120] (0, 0) arc (90:-90:0.2 and 0.5);
\draw[dashed, shift={(-2.0784609691, 0.7)}, rotate=120] (0, 0) arc (90:270:0.2 and 0.5);
\draw[shift={(-2.0784609691, -0.7)}, rotate=-120] (0, 0) arc (-90:90:0.2 and 0.5);
\draw[dashed, shift={(-2.0784609691, -0.7)}, rotate=-120] (0, 0) arc (270:90:0.2 and 0.5);

\draw[thick] (0, -0.25) arc (-90:0:1.887 and 0.25);
\draw[thick, dashed] (0, 0.15) arc (90:180.7:1.887 and 0.15);
\draw[thick] (0, 0.4) arc (-90:-25.7:1.5);
\draw[thick, dashed] (0, 0.35) arc (-90:-155:1.55);
\draw[thick] (0, -0.4) arc (90:25.7:1.5);
\draw[thick, dashed] (0, -0.42) arc (90:155:1.48);

\fill[black!0] (0, 0) circle (0.2 and 0.5);

\draw (0, -0.5) arc (-90:90:0.2 and 0.5);
\draw[dashed] (0, -0.5) arc (270:90: 0.2 and 0.5);

\draw[thick, dashed] (1.887, 0) arc (0:95:1.887 and 0.25);
\draw[thick] (-1.887, 0) arc (180:275.7:1.887 and 0.15);

\draw[thick] (1.887-0.15, 0) arc (180:144.8:1.4+0.15);
\draw[thick] (1.887-0.15, 0) arc (180:215.2:1.4+0.15);

\draw[thick] (-1.887+0.15, 0) arc (0:35.2:1.4+0.15);
\draw[thick] (-1.887+0.15, 0) arc (0:-35.2:1.4+0.15);

\draw (0, 0.8) node {$C_{1}$};
\draw (0.66, 1) node {$\kappa_{1}^{-}$};
\draw (0.67, -1) node {$\kappa_{2}^{-}$};
\draw (2.25, -0.4) node {$\kappa_{3}^{-}$};
\draw (1.3, -0.38) node {$\eta^{-}$};

\draw (-0.65, 1) node {$\kappa_{1}^{+}$};
\draw (-0.67, -1) node {$\kappa_{2}^{+}$};
\draw (-2.2, -0.4) node {$\kappa_{3}^{+}$};
\draw (-1.2, -0.35) node {$\eta^{+}$};

\end{scope}
\end{tikzpicture}
\caption{Pants containing $C_1$} \label{fig:pantsC1}
\end{figure}
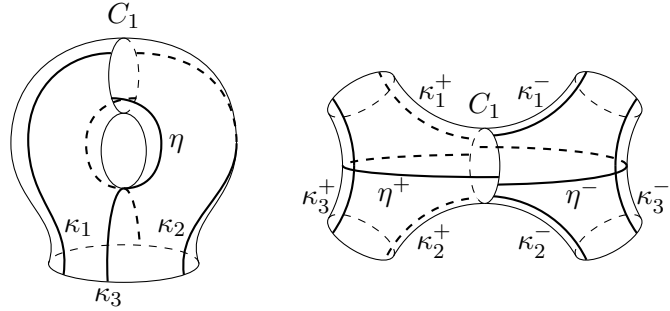

For the sake of simplicity, we explain for the case that $C_{1}$ is non-separating. In order to define representations for $X_{r}$, we first fix a unit vector $\vec{V}$ on $\Ha$ based at some $\tilde{p} \in \Ha$. Also, let $\vec{v}$ be a unit vector on $X_{r}$ based at $p$. Then there exists representations $\Gamma_{r} : \pi_{1}(S, p) \rightarrow \PR$ corresponding to $X_{r}$'s such that $\vec{v}$ is lifted to $\vec{V}$.

We now investigate the monodromy of loops $\alpha$ in $\pi_{1}(S)$. Suppose first that $\alpha$ and $C_{1}$ are disjoint. On each $X_{r}$, $\alpha$ is homotopic to a concatenation of the following segments: \begin{itemize}
\item geodesics on $R$ (meeting $\partial R$ orthogonally), 
\item geodesics along $\partial R$ or 
\item geodesic $\kappa_{3}$.
\end{itemize} 
The angles among such geodesics are kept perpendicular during the pinching. Moreover, the lengths of the first two types are unchanged during the pinching. The length of $\kappa_{3}$ continuously grows and converges to a finite value. Moreover, the length proportion of segments of $\kappa_{3}$ cut by $p$ also varies continuously and converges to a finite value. Consequently, the image $\vec{V}_{\alpha, r} = \Gamma_{r}(\alpha)(\vec{V})$ of $\vec{V}$ by the monodromy along $\alpha$ varies continuously along $r$, converging to a limit $\vec{V}_{\alpha, 0}$ as $r \rightarrow 0$. 

Let us now consider a segment $\beta$ from $p$ to $C_{1}$, where $\beta$ is not transversing $C_{1}$ but only meets at one endpoint. Concatenating $\beta$ with a segment along $\eta$, from $C_1$ to $p$, it follows that $\beta$ is homotopic to a concatenation of a loop $\alpha \in \pi_{1}(S)$ disjoint from $C_{1}$ and a segment along $\eta$. This segment along $\eta$ is exactly half of entire $\eta$. Then, we can characterize the lift $\bar{C}_{1}$ of $C_{1}$ on $\Ha$ as follows: the geodesic transport $\vec{V}_{\alpha, r}'$ of $V_{\alpha, r}$ by distance $\pm l_{X_{r}}(\eta)/2$ becomes a normal vector to $\bar{C}_{1}$. Such lifts bound a convex region in $\Ha$, which we denote by $K$. As $r \rightarrow 0$, those lifts converge to points in $\partial \Ha$, namely, the endpoints of the geodesic along $\vec{V}_{\alpha, 0}$ for various $\alpha$. Accordingly, $K$ becomes the full $\Ha$.

We now discuss the asymptotic behavior (along the pinching process) of a curve $\alpha$ with $i(\alpha, C_{1}) = k > 0$. On each $X_{r}$, $\alpha$ is homotopic to a concatenation of geodesics $\{\beta_{i}\}_{i=1}^{k}$ along $C_{1}$ and $\{\gamma_{i}\}_{i=1}^{k}$ orthogonal to $C_{1}$. It follows that \begin{equation}\label{eqn:lengthComp}
\sum_{i} l_{X_{r}} (\gamma_{i})  \le l_{X_{r}}(\alpha) \le \sum_{i} l_{X_{r}}(\beta_{i}) + \sum_{i} l_{X_{r}} (\gamma_{i}).
\end{equation}

We claim that $l_{X_{r}}(\beta_{i})$ is of class $O(r)$ during the pinching. This is because $\gamma_{i}$, $\gamma_{i+1}$ never crosses $\eta$'s. To see this, note that the geodesics orthogonally departing from $C_{1}$ are parametrized by their departure point on $C_{1}$. If we slightly perturb $\eta$, it will still return to $C_{1}$ but not orthogonally. Moreover, the threshold of this perturbation is locally uniform along $r$. Thus, if $\gamma_{i}$ were $\eta$ at a moment $X_{r}$, then it would stay at $\eta$ forever. Note also that the shear among the lifts $\bar{\eta}$ of $\eta$ each side of $\bar{C}_{1}$ is kept constant during the pinching. As a result, if a lift $\bar{\alpha}$ of $\alpha$ is sandwiched by two $\eta$'s on each side of $\bar{C}_{1}$ marking twist $\tau$, then $l_{X_{r}}(\beta_{i})$ is always dominated by $\tau l(\bar{C}_{1}) = \tau r$. See Figure \ref{fig:twistC1}.

\begin{figure}[ht]
\begin{tikzpicture}[scale=0.85]
\draw[thick, ->] (0, -4) -- (0, 4);
\foreach \i in {1, ..., 4}{
	\draw[dashed] (-4, 1.8*\i - 4.3) -- (0, 1.8*\i - 4.3);
	\draw (4, 1.8*\i - 5) -- (0, 1.8*\i - 5);
	\draw[very thick, ->] (0.9, 1.8*\i - 5) -- (0, 1.8*\i - 5) -- (0, 1.8*\i - 4.3) -- (-0.9, 1.8*\i - 4.3);
}
\draw[very thick, ->] (4, -2.8) .. controls (-2, -2.8) and (2, 2.4) .. (-4, 2.4);
\draw[very thick, dashed] (4, -2.8) -- (0, -2.8);
\draw[very thick, dashed] (-4, 2.4) -- (0, 2.4);

\draw (0.45, 3.9) node {$\bar{C}_{1}$};
\draw (-3, -3.8) node {$K$};
\draw (3, -3.8) node {$K'$};
\end{tikzpicture}
\caption{Twist at $C_1$} \label{fig:twistC1}
\end{figure}
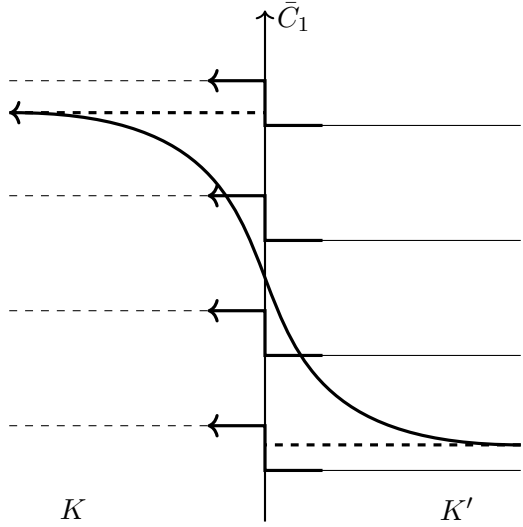

Meanwhile, $l_{X_{r}}(\gamma_{i})$ grows exponentially. To elaborate this, we first fix a lift $\bar{\gamma}_{i}$ of $\gamma_{i}$ in $K$, which meets lifts $\bar{C}_{1}$, $\bar{C}_{1}'$ of $C_{1}$ at endpoints. As explained before, $\bar{C}_{1}$ and $\bar{C}_{1}'$ converge to limits $c$, $c' \in \partial \Ha$, respectively. (Here $c \neq c'$ because $\gamma_{i}$ is nontrivial) Let $\hat{\gamma}_{i}$ be the geodesic connecting $c$ and $c'$, and fix an arbitrary point $p_{ref}$ on $\hat{\gamma}_{i}$. Let $p_{perp}$ be the foot of perpendicular from $p_{ref}$ to $\bar{\gamma}_{i}$. Note that $d(p_{perp}, p_{ref}) \rightarrow 0$ as $r \rightarrow 0$.

$\bar{C}_{1}$ is adjacent to a sequence of lifts of $\eta$. Among them we pick two consecutive lifts $\bar{\eta}_{+}$, $\bar{\eta}_{-}$ inside $K$, sandwiching the ray $\bar{\gamma}_{i}$. (Recall that $\bar{\gamma}_{i}$ will not cross $\bar{\eta}_{+}$ or $\bar{\eta}_{-}$ during pinching) $\bar{\eta}_{+}$ meets $\bar{C}_{1}$ at one endpoint and meets another lift $\bar{C}_{1, +}$ of $C_{1}$ at another endpoint. Similarly, $\bar{\eta}_{-}$ meets $\bar{C}_{1}$ and another lift $\bar{C}_{1, -}$ of $C_{1}$ at endpoints. $\bar{C}_{1, +}$ and $\bar{C}_{1, -}$ also converge to points $c_{+}$ and $c_{-}$ on $\partial \Ha$, respectively, as $r \rightarrow 0$.

\begin{figure}[ht]

\def\c{1.1}
\begin{tikzpicture}[scale=1]
\draw[thick] (-3, 0) -- (2.8, 0);
\draw (0.35, 0) -- (0.35, 4);
\draw (-0.35, 0) -- (-0.35, 4);

\draw[thick] (0.1 + 0.377027781, 0) arc (180:150:8.1);
\draw[thick] (0.1 - 0.590169943749474, 0) arc (0:52:5);

\draw[thick] (0.1 - 1.11803398875, 2.2390679775) arc (116.56505117707798:72.847578259788:2.5);

\draw[thick] (-0.17, 0.8) -- (-0.17, 2.481934729198171);
\draw (-0.17, 0.9) -- (-0.17, 4);
\draw[thick] (0.14, 0.9) arc (174:171:30.45);
\fill (-0.17, 0.8) circle (0.05);
\fill (0.14, 0.9) circle (0.05);
\draw[dashed] (-1.5, 2.2390679775) -- (1.5, 2.2390679775);
\draw[dashed] (-1.5, 2.505) -- (1.5, 2.505);



\fill[black!0, shift={(-0.55, 0.75)}] (-0.32, -0.2) -- (0.32, -0.2) -- (0.32, 0.2) -- (-0.32, 0.2) -- cycle;
\fill[black!0, shift={(0.62, 0.88)}] (-0.32, -0.2) -- (0.32, -0.2) -- (0.32, 0.2) -- (-0.32, 0.2) -- cycle;
\draw (-0.55, 0.75) node {$p_{ref}$};
\draw (0.62, 0.88) node {$p_{perp}$};

\draw (-0.3, -0.2) node {$c_{ -}$};
\draw (0.38, -0.2) node {$c_{+}$};
\draw (0, 5) node {$c$};
\draw (0, 4.7) node {$\uparrow$};

\draw (-0.2, 4.2) node {$\hat{\gamma}_{i}$};

\draw (-2, 3.2) node {$\bar{\eta}_{ -}$};
\draw (1.76, 3.6) node {$\bar{\eta}_{+}$};

\fill[opacity=0.18] (0.1 + 0.377027781, 0)  -- (0.1 - 0.590169943749474, 0) arc (0:26.56505117707798:5) arc (116.56505117707798:72.847578259788:2.5) arc (162.847578259788:180:8.1);
\draw (-0.5, 1.5) arc (75:105:1.6);
\draw (-1.5, 1.5) node {$K$};
\draw (1.8, 2.505) node {$L$};
\draw (-1.8, 2.2390679775) node {$L_{-}$};

\begin{scope}[shift={(5.8, 1.2)}]
\fill[black!18] (-0.707106781186548*\c, 0.707106781186548*\c) arc (135:45:1*\c) arc (135:180:\c) -- (-1.414213562373095*\c +1*\c, 0) arc (0:45:\c) -- cycle;
\draw[thick] (-2.8*\c, 0) -- (2.8*\c, 0);
\draw[thick] (-1.414213562373095*\c - 1*\c, 0) arc (180:0:1*\c);
\draw[thick] (1.414213562373095*\c + 1*\c, 0) arc (0:180:1*\c);
\draw[thick] (-0.707106781186548*\c, 0.707106781186548*\c) arc (135:45:1*\c);
\draw (-1.414213562373095*\c , 0) --  (-0.707106781186548*\c, 0.707106781186548*\c) -- (0, 0) -- (0.707106781186548*\c, 0.707106781186548*\c) -- (1.414213562373095*\c , 0);

\draw(-1.414213562373095*\c - 1*\c, -0.3) node {$b_{-}$};
\draw(-1.414213562373095*\c , -0.3*\c) node {$\frac{a_{-} + b_{-}}{2}$};
\draw(-1.414213562373095*\c + 1*\c, -0.3) node {$a_{-}$};
\draw (0, -0.22) node {$d$};
\draw(1.414213562373095*\c + 1*\c, -0.3) node {$b_{+}$};
\draw(1.414213562373095*\c , -0.3*\c) node {$\frac{a_{+} + b_{+}}{2}$};
\draw(1.414213562373095*\c - 1*\c, -0.3) node {$a_{+}$};

\end{scope}

\end{tikzpicture}
\caption{Pinching described on $\Ha$}
\label{fig:pinching}
\end{figure}
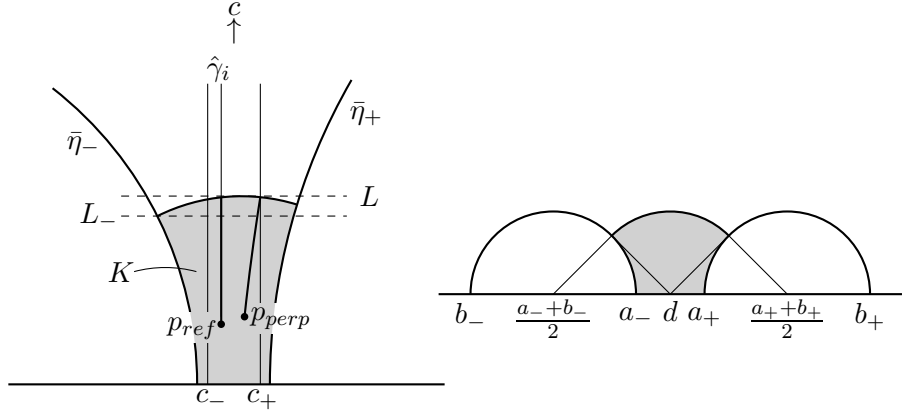

Let us work on the upper half plane model with $c = \infty$, $c_{+}= 1$ and $c_{-}= 0$. We fix points $a_{\pm}$, $b_{\pm}$, $d$ on the real line as in Figure \ref{fig:pinching}, and define $r_{1, \pm}$ and $r_{2, \pm}$ as follows: \[
r_{1, \pm} := \left|a_{\pm} - \frac{a_{\pm} + b_{\pm}}{2}\right|, \quad r_{2, \pm} := \left| d - \frac{a_{\pm} + b_{\pm}}{2}\right|.
\]
Note that $a_{\pm}$, $d$ remain bounded while $|b_{\pm}| \rightarrow \infty$ as $r \rightarrow 0$. Thus, $r_{1, \pm} / r_{2, \pm}$ tends to 1 during the pinching.

We now consider three horocycles based at $\infty$: $L$ passing thourgh the highest point of $\bar{C}_{1}$, and $L_{\pm}$ passing though $\bar{C}_{1} \cap \bar{\eta}_{\pm}$. We record their Euclidean $y$-coordinates of $L$, $L_{\pm}$, and $p_{perp}$ by $y_{L}$, $y_{L_{\pm}}$, and $y_{p_{perp}}$, respectively. Then we have \begin{equation}\label{eqn:pinchingLimit}
\lim_{r\rightarrow 0} \frac{y_{L_{\pm}}}{y_{L}} = \lim_{r\rightarrow 0} \frac{r_{1, \pm}}{r_{2, \pm}} = 1
\end{equation}
and 
\begin{equation}\label{eqn:sandwichLength}
\ln \frac{\min(y_{L_{+}}, y_{L_{-}})}{y_{p_{perp}}}  \le d(p_{perp}, \bar{C}_{1}) \le \ln\frac{y_{L}}{y_{p_{perp}}}.
\end{equation}
from the geometry. 

Let $w$ be the Euclidean width of $\overline{C}_1$. Then we finally relate the length $r$ of $C_{1}$ with $y_{L}$ as follows: \begin{multline*}
\ln \left( \frac{1}{y_{L}} \int_{\bar{C_{1}}} dx \right) = \ln w - \ln y_{L} \le \ln r = \ln \left( \int_{\bar{C}_{1}} \frac{ds}{y} \right) 
\\
\le \ln \left( \frac{y_{L}}{\min(y_{L_{+}}, y_{L_{-}})^{2}} \int_{\bar{C}_{1}} dx \right)= \ln w - \ln y_{L}+\ln \frac{y_{L}^{2}}{\min(y_{L_{+}}, y_{L_{-}})^{2}}.
\end{multline*}
Here the Euclidean width $w$ of $\bar{C}_{1}$ is equal to \[
\frac{r_{2, +}^{2} - r_{1, +}^{2}}{r_{2, +}} + \frac{r_{2, -}^{2} - r_{1, -}^{2}}{r_{2, -}}.
\]
Recall that $r_{1, \pm} / r_{2, \pm} \rightarrow 1$ as $r \rightarrow 0$. Moreover, \[
\lim_{r \rightarrow 0} [(r_{2, +} - r_{1, +}) + (r_{2, -} - r_{1, -})] = \lim_{r \rightarrow 0} (a_{+} - a_{-}) = 1.
\]
Using this, we conclude that $w \rightarrow 2$ and $\ln y_{L} / \ln r \rightarrow -1$ as $r \rightarrow 0$. From this conclusion and Equation \ref{eqn:pinchingLimit}, \ref{eqn:sandwichLength}, we obtain that \[
\lim_{r \rightarrow 0} \frac{d(p_{perp}, \bar{C}_{1})}{- \ln r}=1.
\]

The same logic applies to $\bar{C}'_{1}$, the other end. Thus we obtain\[
\lim_{r \rightarrow 0} \frac{l_{X_{r}} (\gamma_{i})}{-\ln r} = 2.
\]
Similar discussion also holds for other $\gamma_{i}$'s. Since $l_{X_{r}}(\beta_{i})$'s are of class $O(r)$, we conclude that \begin{multline*}
2 i(\alpha, C_{1}) = \lim_{r \rightarrow 0} \sum_{i=1}^{k} \frac{l_{X_{r}} (\gamma_{i})}{-\ln r} \le \liminf_{r \rightarrow 0} \frac{l_{X_{r}}(\alpha)}{-\ln r} \\\le \limsup_{r \rightarrow 0} \frac{l_{X_{r}}(\alpha)}{-\ln r} \le \lim_{r\rightarrow 0}  \sum_{i=1}^{k}\left( \frac{l_{X_{r}} (\gamma_{i})}{-\ln r} + \frac{l_{X_{r}} (\beta_{i})}{-\ln r} \right)= 2i(\alpha, C_{1}). \qedhere
\end{multline*}
\end{proof}

\medskip
\bibliographystyle{alpha}
\bibliography{geodesic}

\end{document}